\documentclass{amsart}[11pt]

\usepackage[margin=1in]{geometry}
\usepackage{amsmath,amsthm,amsfonts,amssymb,amscd}
\usepackage{enumerate}
\usepackage{fancyhdr}
\usepackage{mathrsfs}
\usepackage{xcolor}
\usepackage{graphicx}
\usepackage{listings}
\usepackage{hyperref}
\usepackage{bbm}
\usepackage{tikz-cd}
\usepackage{commath}
\usepackage{amsmath,amsthm,amssymb, amsfonts,mathtools}
\usepackage{latexsym}
\usepackage{url}
\usepackage{xcolor}
\usepackage{graphicx}
\usepackage{float}
\usepackage{caption}
\usepackage{subcaption}
\usepackage{verbatim}
\usepackage{enumitem}
\usepackage{amsrefs}
\usepackage{array}
\usepackage{booktabs}
\usepackage{color}
\usepackage[utf8]{inputenc}
\usepackage{tcolorbox}
\usepackage{bbm}
\usepackage{commath}
\usepackage{enumitem}
\usepackage{hyperref}
\usepackage[T1]{fontenc}

\newtheorem{thm}{Theorem}[section]
\newtheorem{theo}[thm]{Theorem}

\newtheorem{lem}[thm]{Lemma}	
\newtheorem{proposition}[thm]{Proposition}

\newtheorem{defi}[thm]{Definition}
\newtheorem{assumption}[thm]{Assumption}

\theoremstyle{remark}
\newtheorem{remark}[thm]{Remark}
\newtheorem{example}[thm]{Example}

\newtheorem*{lemma*}{Lemma}

\newcommand{\N}{\mathbb{N}}

\newcommand{\R}{\mathbb{R}}
\newcommand{\E}{\mathbb{E}}
\newcommand{\F}{\mathcal{F}}
\newcommand{\Prob}{\mathbb{P}}
\newcommand{\1}{\mathbbm{1}}
\newcommand{\bb}[1]{\mathbb{#1}}
\newcommand{\br}[1]{\lbrace #1 \rbrace}

\newcommand{\toinf}{\rightarrow\infty}
\newcommand{\tto}{\rightarrow}

\newcommand{\W}{\Omega}
\newcommand{\mc}[1]{\mathcal{ #1}}

\usepackage{bm}

\DefineSimpleKey{bib}{arxiveprint}
\DefineSimpleKey{bib}{arxivid}
\DefineSimpleKey{bib}{arxivclass}
\makeatletter
\catcode`\'=11 
  \def\arxiveprint{%
    \resolve@inner{\bib@arxiveprint}
  }
  \def\bib@arxiveprint#1{%
    \begingroup
        #1\relax
        \bib@resolve@xrefs
        \bib@field@patches
        \bib'setup
        \let\PrintPrimary\@empty
        {%
          \IfEmptyBibField{arxivid}{\url{https://arxiv.org/}}
          {%
            \href{https://arxiv.org/abs/\bib'arxivid}{\nolinkurl{arXiv:\bib'arxivid}}%
            \IfEmptyBibField{arxivclass}{}{~\nolinkurl{[\bib'arxivclass]}}
          }
        }\bib'transition
        \setbib@@
    \endgroup
  }
\catcode`\'=12 
\makeatother
\BibSpec{arxiv}{%
    +{}  {\PrintAuthors}                {author}
    +{,} { \textit}                     {title}
    +{.} { }                            {part}
    +{:} { \textit}                     {subtitle}
    +{,} { \PrintContributions}         {contribution}
    +{.} { \PrintPartials}              {partial}
    +{,} { }                            {journal}
    +{}  { \textbf}                     {volume}
    +{}  { \PrintDatePV}                {date}
    +{,} { \issuetext}                  {number}
    +{,} { \eprintpages}                {pages}
    +{,} { }                            {status}
    +{,} { \PrintDOI}                   {doi}
    +{,} { available at \eprint}        {eprint}
    +{,} { available at \arxiveprint}   {arxiveprint}
    +{}  { \parenthesize}               {language}
    +{}  { \PrintTranslation}           {translation}
    +{;} { \PrintReprint}               {reprint}
    +{.} { }                            {note}
    +{.} {}                             {transition}
    +{}  {\SentenceSpace \PrintReviews} {review}
}

\author{Z.W. Bezemek}
\email[Zachary William Bezemek]{zwb@duke.edu}
\address[Zachary William Bezemek]{Duke University, Department of Mathematics\\ 120 Science Drive, Durham, NC 27708, USA}
\author{M. Heldman}
\email[Max Heldman]{maxh@vt.edu}
\address[Max Heldman]{Virginia Tech, Department of Mathematics\\ 225 Stanger Street, Blacksburg, VA 24061, USA}
\thanks{Z.W.B. was partially supported by the National Science Foundation (DMS 2107856). The majority of this work was completed while Z.W.B. was a PhD candidate at Boston University. M.H. was partially supported by NSF-DMS 1902854, ARO W911NF-20-1-0244, and a subgrant of NSF-OAC 2139536. The authors of the paper would like to thank both reviewers for a their careful and constructive reviews of this article.\\ }
\title{Importance Sampling for the Empirical Measure of Weakly Interacting Diffusions}
\date{\today} 

\begin{document}
\begin{abstract}
We construct an importance sampling method for computing statistics related to rare events for weakly interacting diffusions. Standard Monte Carlo methods behave exponentially poorly with the number of particles in the system for such problems. Our scheme is based on subsolutions of a Hamilton-Jacobi-Bellman (HJB) equation on Wasserstein space which arises in the theory of mean-field (McKean-Vlasov) control. We identify conditions under which such a scheme is asymptotically optimal. In the process, we make connections between the large deviations principle for the empirical measure of weakly interacting diffusions, mean-field control, and the HJB equation on Wasserstein space. We also provide evidence, both analytical and numerical, that with sufficient regularity of the HJB equation, our scheme can have vanishingly small relative error in the many particle limit.
\end{abstract}
\maketitle

\section{Introduction}\label{sec:background}

Consider the weakly interacting particle system:
\begin{align}\label{eq:IPS}
dX^{i,N,s,y}_t &= b(X^{i,N,s,y}_t,\mu^{N,s,y}_t)dt+\sigma(X^{i,N,s,y}_t,\mu^{N,s,y}_t) dW^i_t,\quad X^{i,N,s,y}_s = y_i
\end{align}
on some stochastic basis $(\W,\F,\Prob),\br{\F_t}_{t\in[s,T]}$ satisfying the usual conditions, where $b:\R^d\times\mc{P}_2(\R^d)\tto \R^d$, $\sigma:\R^d\times\mc{P}_2(\R^d)\tto \R^{d\times m}$, $\br{W^i}_{i\in \bb{N}}$ are independent $m$-dimensional standard $\F_t$-Brownian motions initialized at $W^i_s=0$, $i\in \br{1,...,N}$. Here $\mc{P}_2(\R^d)$ is the space of probability measures with finite second moment (see Definition \ref{def:lionderivative}). In \eqref{eq:IPS}, $N$ denotes the total number of particles, $i\in\{1,2,...,N\}$ indexes a particular particle, and $\mu^{N,s,y}\in C([0,T];\mc{P}(\R^d))$ is the empirical measure:
\begin{align}\label{eq:empiricalmeasure}
\mu^{N,s,y}_t\coloneqq \frac{1}{N}\sum_{i=1}^N \delta_{X^{i,N,s,y}_t},\quad t\in [s,T].
\end{align}
The $i$'th entry of the vector $y=(y_1,y_2,...)\in \oplus_{i=1}^\infty\R^d$ contains the deterministic initial condition for the $i$'th particle at time $s$ satisfying $0\leq s\leq t\leq T$. The entries $y_i$ for $i > N$ are ignored for any particular value of $N$; we include them in defining \eqref{eq:IPS} for convenience, since our purpose will be to consider the sequence of solutions to \eqref{eq:IPS} as $N \to \infty$. We will at times identify $y$ with its projection onto $\R^{dN} = [\R^{d}]^N = \oplus_{i=1}^N \R^d$ without ambiguity.

In this paper, we design a control-based importance sampling scheme for estimating functionals of the form 
\begin{align}\label{eq:desiredexpectation}
  \E\biggl[\exp\left(-NG(\mu^{N,s,y}_T)\right) \biggr],\quad G:\mc{P}(\R^d)\tto\R.
  \end{align}
The basic method goes as follows: letting $v^N_i:[s,T]\times\R^{dN}\tto \R^m$ be a chosen bounded function, or control, for each $i\in \br{1,...,N}$ and $N\in\bb{N}$, we consider the \textit{controlled} version of the interacting particle system \eqref{eq:IPS} given by
\begin{equation}\label{eq:controlledparticles}
\begin{split}
d\hat{X}^{i,N,s,y}_t &= [b(\hat{X}^{i,N,s,y}_t,\hat{\mu}^{i,N,s,y}_t)+\sigma(\hat{X}^{i,N,s,y}_t,\hat{\mu}^{i,N,s,y}_t) v^N_i(t,\hat{X}^{1,N,s,y}_t,...,\hat{X}^{N,N,s,y}_t)]dt\\ 
&+\sigma(\hat{X}^{i,N,s,y}_t,\hat{\mu}^{i,N,s,y}_t) d\hat{W}^i_t,\quad \hat{X}^{i,N,s,y}_s = y_i.
\end{split}
\end{equation}
Girsanov's theorem allows us to convert statistics for the controlled system \eqref{eq:controlledparticles} into statistics for \eqref{eq:IPS}; therefore, in order to develop a more accurate Monte Carlo estimator for \eqref{eq:IPS} our objective will be to identify controls $v_i^N$ that reduce the variance of our target statistic compared with the uncontrolled version, or in the best case cause the relative variance to vanish as $N\toinf$. By these measures, we find that subsolutions of a zero-viscosity Hamilton-Jacobi-Bellman (HJB) equation on Wasserstein space, which we derive starting with the Dawson-G\"artner large deviations principle (LDP) \cites{DG,BDF}, provide controls with good performance.

Our main results, stated rigorously in Section \ref{sec:mainresults}, show that under certain assumptions our scheme outperforms standard Monte Carlo methods by requiring a subexponential number of samples to achieve a given relative error in estimating \eqref{eq:desiredexpectation} as $N\toinf$ (Theorem \ref*{theo:logefficient}). Under stronger assumptions, we show that in fact only a \textit{vanishing} number of simulations suffices (Theorem \ref*{theo:expansionanalysis}), so that for large $N$ our method requires only a single sample. Our theoretical results are confirmed by numerical experiments in Section \ref{sec:numerics}.

Our importance sampling scheme is related to and inspired by previous work concerning the development of asymptotically optimal importance sampling methods for estimating expectations of the form
\begin{align}\label{eq:smallnoisedesiredexpectation}
\E\biggl[\exp\left(-\frac{1}{\epsilon}G\left(X^{\epsilon,s,y}_T\right)\right) \biggr],\quad G\in C_b(\R^d),
\end{align}
where $X^{\epsilon,s,y}$ is a single particle satisfying, e.g., the small-noise SDE:
\begin{equation}\label{eq:smallnoiseSDE}
  \begin{split}
dX^{\epsilon,s,y}_t &= b(X^{\epsilon,s,y}_t)dt+\sqrt{\epsilon}\sigma(X^{\epsilon,s,y}_t) dW_t,\\ 
X^{\epsilon,s,y}_s&=y\in\R^d.
  \end{split}
\end{equation} 
We will draw parallels between the small-noise setting and ours throughout this article.

A general strategy for efficiently computing \eqref{eq:smallnoisedesiredexpectation} involves exploiting the connection between the large deviations rate function of Freidlin-Wentzell (FW) \cite{FW} and a class of first-order Hamilton-Jacobi-Bellman (HJB) equations, whose subsolutions can then be leveraged to obtain optimal controls for the zero-noise process $X^{0,s,y}$. In turn, the controls can be used to design an optimal change of measure for estimating \eqref{eq:smallnoisedesiredexpectation}. The control formulation of the FW rate function, provided via the weak convergence approach to large deviations of Dupuis and Ellis \cites{DE,BD}, makes the chain linking the rate function, zero-viscosity HJB equation, and optimal control of the zero-noise process more evident. The strategy described above has been applied for a diverse range of stochastic models \cites{DSW,VW,DW1,DW2,SPDEImportanceSampling}. 

Elevating the above framework for small-noise SDE importance sampling schemes to interacting particle systems is a major contribution of this paper.  Moreover, the connections between the approaches used herein and in the small-noise setting allow us to conjecture (see Remark \ref{remark:extensions}) about extending this research to problems related to metastable dynamics that have been previously considered for the small-noise case (see, e.g., \cite{GSE} and the references therein) such as exit probabilities and mean first passage times. We therefore view this paper as the first step towards a new method of designing importance sampling schemes for studying properties of the empirical measure which are of interest to the greater scientific community, such as the free-energy differences for chemical and biological molecular systems \cites{BRS,CGHLPC,VTP,HBSBS} or exit times of the empirical measure from domains of attraction in models related to social dynamics and consensus convergence \cites{Dawson,Garnier1,Garnier2,BM,GS}.

Let us now make the comparison between the interacting particle system \eqref{eq:IPS} and the small-noise system \eqref{eq:smallnoiseSDE} more explicit. It is well known that $\mu^{N,s,y}\tto \mc{L}(X^{s,\nu})$ in distribution when considered as $\mc{P}(C([s,T];\R^d))$-valued random variables. Here $\nu\in \mc{P}(\R^d)$ is the weak limit as $N\toinf$ of the empirical measures $\frac{1}{N}\sum_{i=1}^N\delta_{y_i}$, and $X^{s,\nu}$ satisfies the McKean-Vlasov equation:
\begin{align}\label{eq:McKeanVlasovEquation}
dX^{s,\nu}_t &= b(X^{s,\nu}_t,\mc{L}(X^{s,\nu}_t))dt+\sigma(X^{s,\nu},\mc{L}(X^{s,\nu}_t))dW_t,\quad X^{s,\nu}_s\sim \nu.
\end{align}
This phenomenon, known as the propagation of chaos, goes back to the works of Kac \cite{Kac}. The empirical measure $\mu^{N,s,y}$ can also formally be seen to satisfy an \textit{infinite-dimensional} small-noise equation, analogous to \eqref{eq:smallnoiseSDE}, with parameter $\epsilon = \frac{1}{\sqrt{N}}$, as per pp. 249-250 of \cite{DG}. Thus, the problem of estimating \eqref{eq:desiredexpectation} has the same basic structure as the more well-known problem of estimating \eqref{eq:smallnoisedesiredexpectation}, but with a different state space.  


An LDP for the many-particle limit of $\br{\mu^{N,s,y}}_{N\in\bb{N}}$ was first established in the classical work of \cite{DG} in the case where the diffusion coefficient $\sigma$ does not depend on the empirical measure, with the result extended via the weak convergence approach of Dupuis and Ellis \cite{DE} to the general setting in \cite{BDF}. In the latter, the form of the rate function is given in terms of a control problem (see Theorem \ref{theo:LDP}), which further motivates our search for an asymptotically optimal importance sampling scheme for \eqref{eq:desiredexpectation} associated with an HJB equation as in the small-noise SDE setting. In contrast to the setting of small-noise SDEs, where the state space of the HJB equation is Euclidean space, we find that the appropriate HJB equation to consider in our setting of measure-valued random variables is posed on Wasserstein space. 

Perhaps due to the computational intractability of both the Dawson-G\"artner \cite{DG} and Budhiraja-Dupuis-Fischer \cite{BDF} rate functions, practical applications of the LDP for the many-particle limit of the empirical measure \eqref{eq:empiricalmeasure} associated with \eqref{eq:IPS} have been few and far between (for some exceptions see, e.g., \cites{Garnier1,Garnier2,NN,DMG}). On the other hand, in recent years the HJB equation on Wasserstein space of \cite{PW} (along with the related equations in \cites{LP,CCD}) has received an immense amount of attention both in terms of numerical applications \cites{CL,CCD_Numerical,FZ,GLPW,GMW,Lauriere,HHL} and theoretical results \cites{BP,CD,PW,CGKPR1,CGKPR2,GPW,Yong,Lacker,BIRS}. The Dawson-G\"artner rate function has previously been related to the theory of mean field games and control through the observation that it can be viewed in terms of derivatives of the free energy associated to the limiting McKean-Vlasov equation viewed as a gradient flow on Wasserstein space in some settings \cites{ADPZ1,ADPZ2,GNP,BCGL,FK,FN,FMZ,FKbook}. However, to our knowledge, this paper is the first time that the connection between the LDP rate function \eqref{eq:ratefunction} and the HJB equation on Wasserstein space \eqref{eq:HJBequation} has been made explicit in the literature (see Remark \ref{remark:connectionbetweenratefunctionandHJB}). 


As far as we are aware, this paper also represents the first time the problem of estimating \eqref{eq:desiredexpectation} with an asymptotically optimal importance sampling scheme has been considered, though many other versions of the small-noise importance sampling problem have appeared in the literature. These include, for example, discrete-time Markov chains \cites{DW1,DW2}, diffusions with multiscale structure \cites{DSW,MSImportanceSampling}, and stochastic partial differential equations \cites{EMFBG,GSS_ImportanceSampling}. See in addition \cites{GJZ,BQRS,DSZ,ERV,GSE,VW} for a diverse, but by no means comprehensive, collection of small-noise importance sampling research in a variety of settings. We in particular mention the work of \cite{dRST} on the design of importance sampling schemes for small-noise McKean-Vlasov SDEs using a decoupling approach (see also \cites{BRHAMT1,BRHAMT2}) and a complete change of measure approach. The problem posed in \cite{dRST} is significantly different from the one we consider here, principally because they seek to estimate statistics of the limiting McKean-Vlasov equation \eqref{eq:McKeanVlasovEquation} rather than the interacting particle system \eqref{eq:IPS}. The change of measure they employ is therefore based on the Freidlin-Wentzell (small-noise) LDP for Brownian motion, rather than the Sanov-type (many-particle) LDP for the empirical measure \eqref{eq:empiricalmeasure} used in our setting. Note that Freidlin-Wentzell LDPs have been derived in different settings for McKean-Vlasov equations \cites{dRSTldp,LSZZ}, though they were not directly employed in the results of \cites{dRST,BRHAMT1,BRHAMT2}.
 
Finally, we note that methods of importance sampling for high-dimensional diffusions have also been established in previous research, e.g., \cites{BQRS,NR,TS}). While in theory these could be applied to estimate \eqref{eq:desiredexpectation}, they are designed for more general problems than ours and therefore do not exploit the exchangeability of the particles, and resulting propagation of chaos, to obtain an asymptotically optimal scheme. However, it may be that a combination of these non-asymptotic techniques with ours could allow us to treat a broader range of problems in the future, as we remark in Section \ref{sec:conclusions}.



The rest of the paper is organized as follows. We establish notation in Subsection \ref{sec:notation} below, and go on to provide background on large deviations theory for the empirical measure in Subsection \ref{sec:ldp}. Then, in Subsection \ref{sec:logefficiency} we develop measures of performance for the control-based importance sampling of diffusions for computing statistics of the form \eqref{eq:desiredexpectation}, some of which can be quantified using large deviations theory. In Section \ref{sec:mainresults}, we formally derive the HJB Equation on Wasserstein space and state our main results, Theorem \ref{theo:logefficient} and Theorem \ref{theo:expansionanalysis}. In Section \ref{sec:LQ}, we introduce a class of linear-quadratic examples for which the HJB equation can be solved analytically and discuss simple cases in which our scheme can be thoroughly analyzed in comparison to standard Monte Carlo, along the way comparing with results for small-noise diffusion processes on $\R^d$. In Section \ref{sec:numerics}, we demonstrate our method with numerical examples. Section \ref{sec:proofs} contains detailed proofs of the results stated in Section \ref{sec:mainresults}, before Section \ref{sec:conclusions} provides concluding remarks. Lastly, in Appendix \ref{appendix:P2differentiation} we provide additional background on the calculus of square-integrable probability measures, $\mc{P}_2(\R^d)$.

\subsection{Notation}\label{sec:notation}
In the following paragraph, let $\bm{X}$ and $\bm{Y}$ be Polish spaces, and $(\W,\F,\mu)$ a measure space. In the course of this paper, we denote by $\mc{P}(\bm{X})$ the space of probability measures on $\bm{X}$ endowed with the topology of weak convergence. We use $\mc{P}_2(\bm{X})\subset \mc{P}(\bm{X})$ to denote the subspace of $ \mc{P}(\bm{X})$ consisting of square-integrable probability measures on $\bm{X}$ endowed with the 2-Wasserstein metric (see Definition \ref{def:lionderivative}). We additionally let $\mc{B}(\bm{X})$ be the Borel $\sigma$-algebra associated with $\bm{X}$, $C(\bm{X};\bm{Y})$ the space of continuous functions from $\bm{X}$ to $\bm{Y}$, $C_b(\bm{X})$ the space of bounded, continuous functions from $\bm{X}$ to $\R$ with norm $\norm{\phi}_\infty\coloneqq \sup_{x\in\bm{X}}|\phi(x)|$, and for $p\geq 1$ we define $L^p(\W,\F,\mu;\R^d)$ as the space of $p$-integrable functions on $(\W,\F,\mu)$ with values in $\R^d$ and norm $$\norm{\phi}_{L^p(\W,\F,\mu;\R^d)}=\biggl(\int_{\W} |\phi(z)|^p\mu(dz)\biggr)^{1/p}.$$ $C_{b,L}(\R^d)\subset C_b(\R^d)$ is the set of bounded, globally Lipschitz functions $\phi:\R^d\tto \R$ with norm $\norm{\phi}_{b,L}\coloneqq \norm{\phi}_\infty +\sup_{x,y\in\R^d,x\neq y}\frac{|\phi(x)-\phi(y)|}{|x-y|}$.  For $\phi \in L^1(\bm{X},\mu),\mu\in\mc{P}(\bm{X})$, we define the pairing $\langle \phi,\mu\rangle$ by  $$\langle \phi,\mu\rangle \coloneqq \int_{\bm{X}}\phi(x)\mu(dx).$$  

We note that, as in the above definitions, the codomain for a given function space is assumed to be $\R$ unless otherwise specified. We will also occasionally refer to a measure on a topological space $\bm{X}$ without explicitly stating the corresponding $\sigma$-algebra; in this case, the measurable sets are assumed to come from $\mathcal{B}(\bm{X})$. 

Partial derivatives of a function $\phi$ with respect to a variable $x$ are normally denoted by $\partial_x\phi$; if the variable $x$ is measure valued then the derivative $\partial_x$ should be interpreted in the Lions sense (see Appendix \ref{appendix:P2differentiation}). For time derivatives $\partial_t \phi$ with respect to $t\in I\subseteq \R^+$,  we also use the notation $\dot{\phi}$. 

For $\phi:[0,T]\times\mc{P}_2(\R^d)\tto \R$, we use $\phi\in C^{1,2}([0,T]\times \mc{P}_2(\R^d))$ to mean that $\mu\mapsto \phi(t,\mu)$ is fully $C^2$ in the sense of Definition \ref{def:fullyC2} for all $t$, with all derivatives in that definition jointly continuous in $(t,z,z',\mu)$ and the map $t\mapsto \phi(t,\mu)$ continuously differentiable in $t$. The subspace $C_b^{1,2}([0,T]\times \mc{P}_2(\R^d))\subset C^{1,2}([0,T]\times \mc{P}_2(\R^d))$ denotes the class of functions $\phi$ such that further $\phi$ and all its derivatives from Definition \ref{def:fullyC2} are uniformly bounded. 

For $x,y\in\R^d$, $x=(x_1,...,x_d),y=(y_1,...,y_d)$, $x\cdot y$ denotes the standard inner product $x\cdot y\coloneqq \sum_{i=1}^d x_iy_i$. For $A,B \in \R^{d\times d}$ matrices with entries $\br{a_{ij}}_{i,j=1}^d$ and $\br{b_{ij}}_{i,j=1}^d$ respectively, $A:B\coloneqq \sum_{i,j=1}^da_{ij}b_{ij}$.

For $x\in \oplus_{i=1}^N \R^d$ or $\oplus_{i=1}^\infty \R^d$, we will denote by $\mu^N_x$ associated empirical measure, i.e., the element of $\mc{P}(\R^d)$ given by 
\begin{align*}
\mu^N_x\coloneqq\frac{1}{N}\sum_{i=1}^N\delta_{x_i}.
\end{align*}


Finally, for given sequences $\br{f_n},\br{g_n}$ in $\R^+$, we use the following notation to compare their asymptotic behavior: $f_n = \mathcal{O}(g_n)$ if there is a constant $C > 0$ independent of $n$ such that $f_n \leq Cg_n$, and $f_n = o(g_n)$ if $\frac{f_n}{g_n} \to 0$ as $n\to\infty$. If instead of positive real numbers $f_n$ and $g_n$ are sequences on normed spaces, we use the same notation to refer to the behavior of the sequences $|f_n|$ and $|g_n|$. 

\subsection{Large deviations results and an HJB equation for the particle system}\label{sec:ldp}

We now introduce some of the key tools for proving the our main results in Section \ref{sec:mainresults}, namely, a stochastic control representation for the LDP associated with the empirical measure \eqref{eq:empiricalmeasure} as provided by \cite{BDF} and a control representation for statistics of the form \eqref{eq:desiredexpectation}. The final tool, the HJB equation on Wasserstein space, will be introduced in Section \ref{sec:mainresults}. 

The Laplace principle stated in Theorem \ref{theo:LDP} is based on the controlled version of \eqref{eq:McKeanVlasovEquation}:
\begin{equation}\label{eq:controlledMcKeanVlasov}
\begin{split}
d\hat{X}^{v,s,\nu}_t&=[b(\hat{X}^{v,s,\nu}_t,\mc{L}(\hat{X}^{v,s,\nu}_t))+\sigma(\hat{X}^{v,s,\nu}_t,\mc{L}(\hat{X}^{v,s,\nu}_t))v(t)]dt+\sigma(\hat{X}^{v,s,\nu}_t,\mc{L}(\hat{X}^{v,s,\nu}_t))dW_t,\\ 
\hat{X}^{v,s,\nu}_s&\sim \nu.
\end{split}
\end{equation}
For the LDP to hold, we need the following assumptions on the particle system \eqref{eq:IPS}, the controlled version of the limiting McKean-Vlasov equation \eqref{eq:controlledMcKeanVlasov}, and the problem data:
\begin{assumption}\label{assumption:forLDP}
Suppose:
\begin{enumerate}[label=(A\arabic*)]
\item \label{assumption:convergenceofinitialconditions} The initial conditions $y$ for \eqref{eq:IPS} are such that there exists $\nu\in \mc{P}_2(\R^d)$ with $\mu^N_y\tto \nu$ as $N\toinf$.
\item \label{assumption:continuityofcoefficients} The coefficients $b$ and $\sigma$ are continuous on $\R^d\times\mc{P}(\R^d).$
\item \label{assumption:uncontrolleduniqueness} For all $N\in \bb{N}$, existence and uniqueness of solutions holds in the strong sense for the system of SDEs \eqref{eq:IPS}.
\item \label{assumption:weaksenseuniqueness}  Weak uniqueness holds for solutions of \eqref{eq:controlledMcKeanVlasov} in the sense that if $\Theta^1,\Theta^2$  are probability measures such that for $i=1,2$: 
\begin{itemize}
\item $\Theta^i$ is the law of some process triple $(\hat{X}^{v_i,s,\nu},v_i,W^i)$ satisfying \eqref{eq:controlledMcKeanVlasov} on some filtered probability space $(\W^i,\F^i,\Prob^i),\br{\F^i_t}_{t\in[s,T]}$ satisfying the usual conditions 
\item $v_i$ is an $\F^i_t$-progressively measurable control with $\E^{\Theta^i}\left[\int_s^T |v_i(t)|^2dt\right]<\infty$ 
\item $W^i$ a standard $\F^i_t$-Brownian motion with $W^i_s=0$
\end{itemize}
and $\Theta^1\circ \vartheta^{-1}=\Theta^2\circ \vartheta^{-1}$ where $\vartheta$ is the identity mapping in the second and third coordinates and the evaluation map at time $s$ in the first coordinate, then $\Theta^1=\Theta^2$.
\item \label{assumption:tightness} For any sequence of controlled processes $\br{\hat{X}^{i,N,s,y}}$ from \eqref{eq:controlledparticles} with $\mc{F}_t$-progressively measurable controls $v^N=(v^N_1,...,v^N_N)$ satisfying  $$\sup_{N\in\bb{N}}\E\left[\frac{1}{N}\sum_{i=1}^N \int_s^T |v^N_i(t)|^2dt\right]<\infty,$$ $\hat{\mu}^N\coloneqq \frac{1}{N}\sum_{i=1}^N \delta_{\hat{X}^{i,N,s,y}}$ is tight as a sequence of $\mc{P}(C([s,T];\R^d))$-valued random variables.
\end{enumerate}
\end{assumption}

Assumption \ref{assumption:forLDP} is taken directly from \cite{BDF} with the exception of \ref{assumption:weaksenseuniqueness}, which is a mildly less restrictive form of weak uniqueness under which the arguments from \cite{BDF} still hold (see Lemma 3.4 in \cite{BCsmallnoise} or Definition 3.6 in \cite{BS}). For verifiable conditions on the data of the problem under which \ref{assumption:weaksenseuniqueness} holds, from Appendix C in \cite{FischerFormofRateFunction} we have that:
\begin{enumerate}[label=(A\arabic*')] 
\item \label{eq:sufficientcondition1} $b,\sigma$ are jointly globally Lipschitz on $\R^d\times \mc{P}(\R^d)$ where $\mc{P}(\R^d)$ is equipped with the bounded Lipschitz metric $d_{BL}(\mu,\nu)\coloneqq\sup_{\br{\phi\in C_{b,L}(\R^d):\norm{\phi}_{b,L}\leq 1}}\int_{\R^d}\phi(x)[\mu(dx)-\nu(dx)]$.
\item \label{eq:sufficientcondition2} $\sigma$ is bounded and $\sigma(x,\mu)=\sigma(\mu)$.
\end{enumerate}
are enough to imply \ref{assumption:continuityofcoefficients}-\ref{assumption:tightness}. Note that the proof given in \cite{FischerFormofRateFunction} that \ref{eq:sufficientcondition1} and \ref{eq:sufficientcondition2} together imply \ref{assumption:weaksenseuniqueness} (Proposition C.1 in \cite{FischerFormofRateFunction}), is based on an erroneous localization argument. See \cite{BCsmallnoise} for a correct proof.

We are now ready to state the LDP for the empirical measure.

\begin{theo}\label{theo:LDP}
Under Assumption \ref{assumption:forLDP}, for any $F\in C_b(C([s,T];\mc{P}(\R^d)))$:
\begin{align}\label{eq:LDP}
\lim_{N\toinf}-\frac{1}{N}\log\E\left[\exp(-NF(\mu^{N,s,y}))\right]=\mc{G}(s,\nu;F)\coloneqq \inf_{\mu\in C([s,T];\mc{P}(\R^d))}\left\{S^\nu_{s,T}(\mu)+F(\mu)\right\},
\end{align}
where $S^\nu_{s,T}:C([s,T];\mc{P}(\R^d))\tto [0,+\infty]$ is given by
\begin{align}\label{eq:ratefunction}
S^\nu_{s,T}(\mu)\coloneqq\inf_{v\in\mc{U}:\mu(t)=\mc{L}(\hat{X}^{v,s,\nu}_t),t\in[s,T]}\frac{1}{2}\E\left[\int_s^T|v(t)|^2dt\right].
\end{align}
In the above we take $\inf\emptyset=+\infty$, and in the expression for $S^\nu_{s,T}(\mu)$ the admissible set of controls $\mc{U}$ is the set of all quadruples $((\W,\F,\Prob),\br{\F_t}_{t\in[s,T]},v,W)$ such that:
\begin{itemize}
\item[1.]  $((\W,\F,\Prob),\br{\F_t})$ forms a stochastic basis satisfying the usual conditions,
\item[2.] $W$ is a standard $m$-dimensional Brownian motion initialized at $W_s=0$,
\item[3.]  $v$ is an $\R^m$-valued $\br{\F_t}$-progressively measurable process such that $\E\left[\int_s^T |v(t)|^2dt\right]<\infty$.
\end{itemize}
We write $v\in\mc{U}$ to denote $((\W,\F,\Prob),\br{\F_t}_{t\in[s,T]},v,W)\in \mc{U}$ in the above. For each $v\in \mc{U}$, $\hat{X}^{v,\nu}_t$ satisfies the controlled McKean-Vlasov equation \eqref{eq:controlledMcKeanVlasov} on the space $(\W,\F,\Prob),\br{\F_t}_{t\in[s,T]}$ with driving Brownian motion $W$.
\end{theo}
\begin{proof}
The result follows from Theorem 3.1 in \cite{BDF} (see also Theorem 5.1 in \cite{FischerFormofRateFunction}) by applying the contraction principle to pose the rate function on $C([s,T];\mc{P}(\R^d))$ rather than $\mc{P}(C([s,T];\R^d))$ as in Proposition 5.6 of \cite{BS}.
\end{proof}

The following corollary of Theorem 3.6 in \cite{BD}, which provides a prelimit expression for the left-hand side of \eqref{eq:LDP} which can be applied to prove Theorem \ref{theo:LDP}, is also critical to our arguments in this paper. In Section \ref{sec:proofs}, we apply it to several auxiliary systems of the form \eqref{eq:IPS} with different choices of the drift term in the course of proving Theorem \ref{theo:logefficient}. The drift terms in the new system have a slightly different form from those in \eqref{eq:IPS}, since they are allowed here to directly depend on the time variable.

To distinguish these auxiliary systems and their solutions from those of the original system \eqref{eq:IPS}, we establish alternative notation here for the solutions of the controlled and uncontrolled auxiliary systems and their empirical measures that we will also use in Section \ref{sec:proofs}. 

\begin{proposition}\label{prop:prelimitLDPexpression}
Let $\tilde{\mu}^{N,s,y}_t=\frac{1}{N}\sum_{i=1}^N \delta_{\tilde{X}^{i,N,s,y}_t},t\in[s,T]$ where $\tilde{X}^{i,N,s,y}_t$ solves the system \eqref{eq:IPSarbitrarydrift} (i.e., \eqref{eq:IPS} with $b$ replaced by $\tilde{b}$):
\begin{align}\label{eq:IPSarbitrarydrift}
d\tilde{X}^{i,N,s,y}_t &= \tilde{b}(t,\tilde{X}^{i,N,s,y}_t,\tilde{\mu}^{N,s,y}_t)dt+\sigma(\tilde{X}^{i,N,s,y}_t,\tilde{\mu}^{N,s,y}_t) d\tilde{W}^i_t,\quad \tilde{X}^{i,N,s,y}_s = y_i
\end{align}
and $\tilde{b},\sigma$ are such that existence and uniqueness holds in the strong sense for the system of SDEs \eqref{eq:IPSarbitrarydrift} and $\tilde{W}^i$ are IID standard $m$-dimensional Brownian motions initialized at $\tilde{W}^i_s=0$. Then for $F\in C_b\left(C([s,T];\mc{P}(\R^d))\right)$:
\begin{align}\label{eq:prelimitrep}
-\frac{1}{N}\log\E\left[\exp(-NF(\tilde{\mu}^{N,s,y}))\right]=\inf_{u^N\in \mc{U}^N}\left\{\E\left[\frac{1}{2N}\sum_{i=1}^N\int_s^T |u^N_i(t)|^2dt\right]+\E[F(\bar{\mu}^{N,s,y})]\right\},
\end{align}
where $\mc{U}^N$ is the space of adapted controls $u^N=(u^N_1,...,u^N_N),u^N_i:[s,T]\tto \R^m$ such that $\E\left[\int_s^T|u^N(t)|^2dt\right]<\infty$, and $\bar{\mu}^{N,s,y}$ is the empirical measure associated with the controlled version of \eqref{eq:IPSarbitrarydrift}, i.e.,
\begin{align*}
d\bar{X}^{i,N,s,y}_t &= \tilde{b}(t,\bar{X}^{i,N,s,y}_t,\bar{\mu}^{N,s,y}_t)dt+ \sigma(\bar{X}^{i,N,s,y}_t,\bar{\mu}^{N,s,y}_t) u(t)dt + \sigma(\bar{X}^{i,N,s,y}_t,\bar{\mu}^{N,s,y}_t) d\tilde{W}^i_t,\quad \bar{X}^{i,N,s,y}_s = y_i.
\end{align*}
\end{proposition}

\subsection{Log-efficiency and asymptotic optimality in control-based importance sampling}\label{sec:logefficiency}


Recall our definition of the controlled particle system \eqref{eq:controlledparticles}, and let 
\begin{align}\label{eq:Zs}
Z^{i,N,s,y}\coloneqq \exp\biggl(-\int_s^T v^N_i(t,\hat{X}^{i,N,s,y}_t,...,\hat{X}^{i,N,s,y}_t)\cdot d\hat{W}^i_t - \frac{1}{2}\int_s^T |v^N_i(t,\hat{X}^{i,N,s,y}_t,...,\hat{X}^{i,N,s,y}_t)|^2dt\biggr).
\end{align}
By Girsanov's theorem, an unbiased estimator for \eqref{eq:desiredexpectation} is given by
\begin{align}\label{eq:deltahat}
\hat{\delta}_N\coloneqq \frac{1}{M}\sum_{j=1}^M \exp\left(-NG(\hat{\mu}^{N,s,y,j})\right)\prod_{i=1}^N Z^{i,N,s,y,j}
\end{align}
where $(\hat{\mu}^{N,s,y,j},Z^{1,N,s,y,j},...,Z^{N,N,s,y,j})$ are $M$ independent samples of $(\hat{\mu}^{N,s,y},Z^{1,N,s,y}...,Z^{N,N,s,y})$, and
\begin{align}\label{eq:controlledempiricalmeasure}
\hat{\mu}^{i,N,s,y}_t \coloneqq \frac{1}{N}\sum_{i=1}^N \delta_{\hat{X}^{i,N,s,y}_t}.
\end{align} 
We note that the standard Monte Carlo method estimator, which we denote by $\delta_N$, corresponds to \eqref{eq:deltahat} with $v_i^N = 0$, so that $Z^{i,N,s,y,j} \equiv 1$.

We define the \textit{relative error} of the estimator \eqref{eq:deltahat} by 
\begin{align}\label{eq:importancesamplingrelativeerror}
\rho(\hat{\delta}_N) = \frac{1}{\sqrt{M}} \sqrt{ \frac{\E\hat\delta_N^2 - \left[\E\hat\delta_N\right]^2}{\left[\E\hat \delta_N\right]^2} } =  \frac{1}{\sqrt{M}}\sqrt{\frac{\E\biggl[\exp(-2NG(\hat{\mu}^{N,s,y}))\prod_{i=1}^N (Z^{i,N,s,y})^2 \biggr]}{\E\biggl[\exp(-NG(\mu^{N,s,y})) \biggr]^2}-1}.
\end{align}
To control the size of \eqref{eq:importancesamplingrelativeerror} we either need to increase the number of samples $M$ or reduce the variance by choosing $v_i^N$ so that
\begin{align}\label{eq:Rdeltahat}
R(\hat{\delta}_N) \coloneqq \frac{\E\biggl[\exp(-2NG(\hat{\mu}^{N,s,y}))\prod_{i=1}^N (Z^{i,N,s,y})^2 \biggr]}{\E\biggl[\exp(-NG(\mu^{N,s,y})) \biggr]^2},
\end{align}
is close to 1. With this in mind, we define our first measure of the efficiency for an importance sampling scheme \eqref{eq:deltahat} in terms of the asymptotic behavior of $R(\hat{\delta}_N)$ as $N\toinf$:
\begin{defi}\label{definition:logefficiency}
An importance sampling scheme of the form \eqref{eq:deltahat} is called \textbf{log-efficient} if 
\begin{align*}
\lim_{N\toinf}-\frac{1}{N}\log R(\hat{\delta}_N)=0.
\end{align*}
\end{defi}
Note that, in our definition, we do not require that $R(\hat \delta_N) \to 1$ or even $R(\hat \delta_N) = \mathcal{O}(1)$; instead we only ask that as $N\toinf$, $R(\hat \delta_N)$ does not grow exponentially in $N$. 





The notion of log-efficiency can also be described in terms of large deviations theory. For the uncontrolled system, from \eqref{eq:LDP} in Theorem \ref{theo:LDP} we know that if $G\in C_b(\mc{P}(\R^d))$ and Assumption \ref{assumption:forLDP} holds, then
\begin{equation}\label{eq:gammas}
\begin{split}
\lim_{N\toinf}\frac{1}{N}\log\E[\exp(-NG(\mu^{N,s,y}_T))] &= - \mc{G}(s,\nu;\mu\mapsto G(\mu_T))\coloneqq -\gamma_1,\\ 
\lim_{N\toinf}\frac{1}{N}\log\E[\exp(-2NG(\mu^{N,s,y}_T))] &= - \mc{G}(s,\nu;\mu\mapsto 2G(\mu_T))\coloneqq -\gamma_2.
\end{split}
\end{equation}
Thus, for the standard Monte Carlo scheme the relative error can be written in terms of $\gamma_1$ and $\gamma_2$ as follows:
\begin{align}\label{eq:expandmontecarlorelativeerror}
\rho(\delta_N) = \frac{1}{\sqrt{M}}\sqrt{\exp(N[2\gamma_1-\gamma_2+o(1)])-1}
\end{align}
Starting from the inequality $\E[\exp(-2NG(\mu^{N,s,y}_T)] \geq \E[\exp(-NG(\mu^N_T)]^2$, taking logarithms, multiplying by $-\frac{1}{N}$, and taking the limit $N\toinf$ shows that $\gamma_2\leq 2\gamma_1.$ Evidently, the relative error $\rho(\delta_N)$ grows exponentially in $N$ if the inequality is strict, i.e., $\gamma := 2\gamma_1-\gamma_2 >0$. Exponentially many samples are therefore required to reduce the relative error below a given tolerance in that case.

Now, we assume that the the relative error for the importance sampling scheme \eqref{eq:importancesamplingrelativeerror} admits the same representation as the standard Monte Carlo importance sampling scheme in \eqref{eq:expandmontecarlorelativeerror}, with analogous quantities $\hat{\gamma}_1$, $\hat{\gamma}_2$, and $\hat{\gamma} = 2\hat{\gamma}_1 - \hat{\gamma}_2$. That is:
\begin{align}\label{eq:expandimportancesamplingrelativeerror}
\rho(\hat\delta_N) = \frac{1}{\sqrt{M}}\sqrt{\exp(N[\hat\gamma+o(1)])-1}.
\end{align}
Note that in the same way as with standard Monte Carlo, we can conclude via Jensen's inequality that $\hat\gamma\geq 0$. Then we can compare the relative error $\rho(\hat \delta_N)$ to the Monte Carlo relative error $\rho(\delta_N)$ by comparing the corresponding quantities $\hat \gamma$ and $\gamma$:
\begin{enumerate}
\item[1.] If $\hat{\gamma} < \gamma$, then the importance sampling scheme is more efficient than standard Monte Carlo in the sense that asymptotically (as $N\toinf$), it yields a smaller relative error.
\item[2.] If $\hat{\gamma} = 0$, then the importance sampling scheme is log-efficient and, if $\gamma>0$, $$\rho(\hat{\delta}_N) = \frac{1}{\sqrt{M}}\sqrt{\exp(o(N))-1} = o(\rho\left(\delta_N\right)).$$
\end{enumerate}
 The object of our first result Theorem \ref{theo:logefficient} presented in the next section is to show conditions under which, for our choice of controls, $\hat{\gamma} < \gamma$ or $\hat \gamma = 0$.
 
Though the assumptions underlying Theorem \ref{theo:logefficient} preclude the relative error of our importance sampling scheme from growing exponentially, extremely fast subexponential growth rates may still be possible. To obtain a more refined result, one needs to study the behavior of the $o(1)$ terms in $R(\hat \delta_N)$. To that end, our second main result, Theorem \ref{theo:expansionanalysis}, assumes an expansion of the prelimit quantities on the left-hand side of \eqref{eq:gammas} leading to an expression for $R(\hat \delta_N)$ of the form 
\begin{align}\label{eqn:Rdeltahatexpansion}
R(\hat \delta_N)=\exp\left(a_{0}N + a_1 + a_2N^{-1} + ...\right).
\end{align}
Note that here $a_0$ plays the same role as $\hat\gamma$ in \eqref{eq:expandimportancesamplingrelativeerror}. Analysis of the coefficients in the power series appearing in the exponential shows that, under stronger regularity assumptions than required for our Theorem \ref{theo:logefficient}, $a_{0} = a_1 = 0$, which implies that the importance sampling scheme has \textit{vanishing} relative error: 
\begin{defi}\label{def:vanishingrelativeerror}
An estimator of the form \eqref{eq:deltahat} is said to have \textbf{vanishing relative error} if, for $\rho(\hat{\delta}_N)$ as in \eqref{eq:importancesamplingrelativeerror},
\begin{align*}
\lim_{N\toinf}\rho(\hat{\delta}_N)=0.
\end{align*}
\end{defi} 

We conclude this chapter by comparing to the small-noise diffusion regime. Our discussion in this section parallels that of the first section in \cite{VW}, in which similar importance sampling scheme for small-noise diffusions is presented along with numerical schemes for estimating the optimal controls. However, we note that in our paper stating the LDP and its underlying assumptions is somewhat more involved due to the measure-valued state space.

 In \cite{VW}, the authors also provide sufficient conditions for a log-efficient importance sampling scheme for small-noise diffusions to have vanishing relative error. For our Theorem \ref{theo:expansionanalysis}, we instead opt to perform an analysis in the spirit of the later \cite{NonAsymptotic}.  The paper \cite{NonAsymptotic} furthered the analysis of the small-noise case by expanding the $o(1)$ terms in the relative error into powers of the small-noise parameter $\epsilon$ to recover log-efficiency from the first-order expansion, and the vanishing relative error result from of \cite{VW} from the second-order expansion (see also the related Section 3 of \cite{HR}). 

\section{The HJB Equation and Statement of Main Results}\label{sec:mainresults}

We now turn to the full statement of our importance sampling scheme in terms of the solution to an HJB equation on Wasserstein space, explain its relationship to the rate function \eqref{eq:ratefunction}, and give a formal derivation of the importance sampling scheme and the HJB equation as the zero viscosity limit of a sequence of zero variance controls. 

Starting with the Hamiltonian function which acts on $\mu\in\mc{P}_2(\R^d)$, $p:\R^d\tto\R^d$ and $\Gamma:\R^d\tto \R^{d\times d}$ by:
\begin{align}\label{eq:Hamiltonian}
H_0(\mu,p,\Gamma)& = - \int_{\R^d}b(x,\mu)\cdot p(x)-\frac{1}{2}p(x)\cdot[\sigma\sigma^\top(x,\mu) p(x)]+\frac{1}{2}\sigma\sigma^\top(x,\mu): \Gamma(x) \mu(dx),
\end{align}
the HJB equation on Wasserstein space is given by:
\begin{equation}\label{eq:HJBequation}
\begin{split}
-\partial_t\Psi(t,\nu)+H_0(\nu,\partial_\mu \Psi(t,\nu)[\cdot],\partial_z\partial_\mu \Psi(t,\nu)[\cdot])&=0,\qquad t\in[0,T),\nu\in\mc{P}_2(\R^d),\\ 
\Psi(T,\nu)&=G(\nu),\qquad \nu\in\mc{P}_2(\R^d).
\end{split}
\end{equation}
Recall here that $\partial_\mu \Psi(t,\nu)[\cdot]:\R^d\tto\R^d$, the derivative of $\Psi$ in the measure-valued variable $\mu$, denotes the Lions derivative of $\Psi$ at $\nu$ (Definition \ref{def:lionderivative} in Appendix \ref{appendix:P2differentiation}).

We will make use of the following notion of subsolutions to \eqref{eq:HJBequation}:
\begin{defi}\label{defi:subsolution}
We call $\Psi:[0,T]\times\mc{P}_2(\R^d)\tto \R$ a \textbf{classical subsolution} to \eqref{eq:HJBequation} if:
\begin{enumerate}
\item[(i)] $\Psi\in C^{1,2}([0,T]\times\mc{P}_2(\R^d))$ 
\item[(ii)] $\partial_t\Psi(t,\mu)-H_0(\mu,\partial_\mu\Psi(t,\mu)[\cdot],\partial_z\partial_\mu\Psi(t,\mu)[\cdot])\geq 0$ for all $t\in [0,T),\mu \in \mc{P}_2(\R^d)$
\item[(iii)] $\Psi(T,\mu)\leq G(\mu),\mu\in \mc{P}_2(\R^d)$
\item[(iv)] There exists $C>0$ such that 
\begin{align*}
\sup_{t\in[0,T],\mu\in \mc{P}_2(\R^d),z\in\R^d}|\partial_\mu \Psi(t,\mu)[z]|&\leq C,\\ 
\sup_{t\in[0,T],\mu\in \mc{P}_2(\R^d)}|\partial_z\partial_\mu \Psi(t,\mu)[z]|&\leq C(1+|z|^2),\\ 
\sup_{t\in[0,T],\mu\in \mc{P}_2(\R^d)}\norm{\partial^2_\mu \Psi(t,\mu)[\cdot,\cdot]}_{L^2(\R^d;\mu)\otimes L^2(\R^d;\mu)}&\leq C.
\end{align*}
\end{enumerate}
\end{defi}
Note that under the above definition, a classical subsolution $\Psi$ to \eqref{eq:HJBequation} is also a viscosity subsolution in the sense of Definition 3.5 in \cite{CGKPR1}. Moreover, a classical \textit{solution} $\Psi$ to \eqref{eq:HJBequation} which also satisfies the bounds of item (iv) in Definition \ref{defi:subsolution} is also a classical subsolution.

\begin{remark}\label{remark:connectionbetweenratefunctionandHJB}
Under certain sufficient conditions stated in Proposition 3.1 of \cite{PW}, the HJB equation \eqref{eq:HJBequation} coincides with the large deviations rate function $S^\nu_{s,T}$ from \eqref{eq:ratefunction} in the sense that $\Psi$ solving \eqref{eq:HJBequation} is given by:
\begin{align}\label{eq:classicalsolutionrepresentation}
\Psi(s,\nu)=\inf_{\mu\in C([s,T];\mc{P}(\R^d))}\br{S^\nu_{s,T}(\mu)+G(\mu_T)}.
\end{align}
Stated more concretely, \cite{PW} shows that if $\sigma$ is constant, $b$ is Lipschitz on $\R^d\times\mc{P}_2(\R^d)$, $G(\mu)\leq C\int_{\R^d}|z|^2\mu(dz)$ for some $C>0$, then any $\Psi\in C^{1,2}_b([0,T];\mc{P}_2(\R^d))$ satisfying \eqref{eq:HJBequation} in the classical sense must be given by \eqref{eq:classicalsolutionrepresentation}. Note that while \cite{PW} restricts to the case of Markovian feedback controls for the infimization problem \eqref{eq:LDP}, we know thanks to the convexity of the running cost in the definition of the rate function that the restricted infimization problem is equivalent to the full one (see Proposition 5.8 in \cite{BS}). 

Moreover, Definition 3.5 of \cite{CGKPR1} states a natural notion of viscosity solutions to \eqref{eq:HJBequation}, and Theorem 3.8 therein establishes sufficient conditions for the value function $(s,\nu)\mapsto \inf_{\mu\in C([s,T];\mc{P}(\R^d))}\br{S^\nu_{s,T}(\mu)+G(\mu_T)}$ to be the unique viscosity solution to \eqref{eq:HJBequation}. Although that theorem does not allow for the linear growth in the control present in our drift term nor the quadratic running cost in the definition of $S^\nu_{s,T}$, it appears the results should hold under this same set of assumptions or similar; see Remarks 2.11 and 3.2 in \cite{CGKPR2}. 

Whether the value function is indeed the unique viscosity solution to \eqref{eq:HJBequation} is of no consequence to the results of this paper. However, the connection between viscosity solutions of \eqref{eq:HJBequation} and the rate function from Theorem \ref{theo:LDP} parallels the small-noise regime, and helped us to identify the correct PDE for designing our importance sampling scheme (compare \eqref{eq:classicalsolutionrepresentation} with, e.g., Equation (2.12) in \cite{VW} in the small-noise setting).
\end{remark}

\subsection{Formal derivation of the HJB equation \eqref{eq:HJBequation}}\label{sec:ontheHJBEquation}
In this section, we will see formally how to arrive at the expression for the control used in our importance sampling scheme in terms of the HJB equation on Wasserstein space \eqref{eq:HJBequation}.

Following pp. 1778-1779 of \cite{VW}, we let $\hat\Phi^N:[0,T]\times\oplus_{j=1}^N\R^{d}\tto \R$ be the solution to the backward Kolmogorov equation with terminal condition $\exp(-NG^N(x_1,...,x_N))=\exp(-NG(\mu^N_x))$ associated to the system of SDEs \eqref{eq:IPS}, so that $\hat{\Phi}^N$ satisfies:
\begin{align}\label{eq:hatPhirepresentation}
\hat\Phi^N(t,x)=\E[\exp(-NG(\mu^{N,t,x}_T))],\qquad x\in \R^{Nd},t\in [0,T],
\end{align}
and 
\begin{equation}\label{eq:prelimitBKE}
\begin{split}
\partial_t \hat\Phi^N(t,x)&+\sum_{i=1}^N b(x_i,\mu^N_x)\cdot \partial_{x_i}\hat\Phi^N(t,x)\\&\qquad+\frac{1}{2}\sigma\sigma^\top(x_i,\mu^N_x):\partial^2_{x_i}\hat\Phi^N(t,x)=0,\qquad t\in[0,T),x=(x_1,...,x_N)\in\R^{dN},
\\ 
\hat{\Phi}^N(T,x)&=\exp(-NG(\mu^N_x)),x\in\R^{dN}.
\end{split}
\end{equation}
For fixed $N$, choosing controls
\begin{align}\label{eq:zerovariancecontrolprelimfromhatPhi}
v^N_i(t,x_1,...,x_N)=\sigma^\top(x_i,\mu^N_x)\frac{\partial_{x_i}\hat\Phi^N(t,x_1,...,x_N)}{\hat\Phi^N(t,x_1,...,x_N)}
\end{align}
in the construction of $\hat\delta_N$ from \eqref{eq:deltahat} leads to a \textit{deterministic} estimator  for \eqref{eq:desiredexpectation} with \textit{zero} variance: \begin{align}\label{eq:deterministicscheme}
e^{-NG(\hat{\mu}^{N,s,y}_T)}\prod_{i=1}^N Z^{i,N,s,y}=\hat\Phi^N(s,y_1,...,y_N).
\end{align} The change of measure described above is commonly known as the Doob $h$-transform, and the equality \eqref{eq:deterministicscheme} can be seen via an application of It\^o's formula.


Next defining $\tilde{\Phi}^N(t,x_1,...,x_N)=-\frac{1}{N}\log \hat{\Phi}^N(t,x_1,...,x_N)$, we have that $\tilde\Phi^N$ is a solution to the second-order HJB equation:
\begin{equation}\label{eq:tildePhiN}
  \begin{split}
-\partial_t\tilde\Phi^N(t,x_1,...,x_N)&-\sum_{i=1}^N\biggl\lbrace b(x_i,\mu^N_x)\cdot \partial_{x_i}\tilde\Phi^N(t,x_1,...,x_N)
\\&-\frac{N}{2}|\sigma^\top(x_i,\mu^N_x)\partial_{x_i}\tilde\Phi^N(t,x_1,...,x_N)|^2
\\ &+\frac{1}{2}\sigma\sigma^\top(x_i,\mu^N_x):\partial^2_{x_i}\tilde\Phi^N(t,x_1,...,x_N)\biggr\rbrace=0,\qquad x_1,...,x_N\in\R^d,t\in[0,T),\\ 
\tilde\Phi^N(T,x_1,...,x_N)&=G(\mu^N_x),\qquad x_1,...,x_N\in\R^d
  \end{split}
\end{equation}
and the choice of controls from \eqref{eq:zerovariancecontrolprelimfromhatPhi} is given in terms of $\tilde\Phi^N$ by:
\begin{align}\label{eq:zerovariancecontrolprelim}
v^N_i(t,x_1,...,x_N)=-N\sigma^\top(x_i,\mu^N_x)\partial_{x_i}\tilde\Phi^N(t,x_1,...,x_N).
\end{align}

As per \eqref{eq:deterministicscheme}, in \eqref{eq:zerovariancecontrolprelim} we have constructed a zero-variance estimator for \eqref{eq:desiredexpectation} from solutions of \eqref{eq:tildePhiN}: 
\begin{align}\label{eq:relationtildephitodesiredexpectation}
\exp(-N\tilde{\Phi}^N(s,y_1,...,y_N))=\E[e^{-NG(\mu^{N,s,y}_T)}]=e^{-NG(\hat{\mu}^{N,s,y}_T)}\prod_{i=1}^N Z^{i,N,s,y},
\end{align}
so if one had access to $\tilde{\Phi}^N$ in order to construct $v^N$, one would already be able to compute the desired expectation explicitly.

Until this point, we have merely derived PDEs for the exact expectation that we wish to compute. To derive an importance sampling scheme that is effective for all large values of $N$, we assume that, given the symmetries in the coefficients and terminal condition in \eqref{eq:tildePhiN}, there exists $\Phi^N:[0,T]\times\mc{P}_2(\R^d)\tto \R$ which is fully $C^2$ in the sense of Definition \ref{def:fullyC2} such that $\Phi^N(t,\mu^N_x)=\tilde\Phi^N(t,x_1,...,x_N)$ for all $t\in[0,T]$, $x_1,...,x_N\in\R^d$. Then by Proposition \ref{prop:empprojderivatives}, we get:
\begin{align*}
-\partial_t\Phi^N(t,\mu^N_x)&-\int_{\R^d}b(z,\mu^N_x)\cdot \partial_\mu \Phi^N(t,\mu^N_x)[z]-\frac{1}{2}|\sigma^\top(z,\mu^N_x)\partial_\mu\Phi^N(t,\mu^N_x)[z]|^2\\ 
&+\frac{1}{2}\sigma\sigma^\top(z,\mu^N_x):\partial_z\partial_\mu\Phi^N(t,\mu^N_x)[z]
\\&+\frac{1}{2N}\sigma\sigma^\top(z,\mu^N_x):\partial^2_\mu\Phi^N(t,\mu^N_x)[z,z]\mu^N_x(dz)=0,\qquad x_1,...,x_N\in\R^d,t\in[0,T),\\ 
\Phi^N(T,\mu^N_x)&=G(\mu^N_x),\qquad x_1,...,x_N\in\R^d.
\end{align*}
The zero-variance optimal control from \eqref{eq:zerovariancecontrolprelim} can then be expressed as 
\begin{align*}
v^N_i(t,x_1,...,x_N)=-\sigma^\top(x_i,\mu^N_x)\partial_\mu\Phi^N(t,\mu^N_x)[x_i].
\end{align*}
Finally, supposing that the equation for $\Phi^N$ holds not only for measures of the form $\mu^N_x$, but for all $\nu\in \mc{P}_2(\R^d)$, the above becomes:
\begin{equation}\label{eq:prelimitHJB}
  \begin{split}
-\partial_t\Phi^N(t,\nu)&-\int_{\R^d}b(z,\nu)\cdot \partial_\mu \Phi^N(t,\nu)[z]-\frac{1}{2}|\sigma^\top(z,\nu)\partial_\mu\Phi^N(t,\nu)[z]|^2\\ 
&+\frac{1}{2}\sigma\sigma^\top(z,\nu):\partial_z\partial_\mu\Phi^N(t,\nu)[z]\\
&+\frac{1}{2N}\sigma\sigma^\top(z,\nu):\partial^2_\mu\Phi^N(t,\nu)[z,z]\nu(dz)=0,\quad\nu\in\mc{P}_2(\R^d),t\in[0,T),\\
\Phi^N(T,\nu)&=G(\nu),\qquad \nu\in\mc{P}_2(\R^d).
\end{split}
\end{equation}
As $N\toinf$, we expect that, up to leading order, solutions to \eqref{eq:prelimitHJB} should be well approximated by solutions to the \textit{zero-viscosity} HJB equation \eqref{eq:HJBequation}, which we arrive at by simply setting the $O(1/N)$ term in \eqref{eq:prelimitHJB} to $0$, leading us finally to the HJB equation on Wasserstein space \eqref{eq:HJBequation}. Rigorous justifications for this limit in the case where there is a common driving noise between the particles can be found in \cite{GPW}. 

Our discussion here parallels the construction in Section 3.1 of \cite{VW} of a log-efficient control in the small-noise setting, in which the authors start from the deterministic scheme given by the Doob $h$-transform and take a formal zero-viscosity limit in the small-noise parameter. Given the expression \eqref{eq:relationtildephitodesiredexpectation}, we expect then that at least the $O(N)$ terms in an expression for $\log R(\hat{\delta}_N)$ with the choice of controls from Theorem \ref{theo:logefficient} should vanish, since they are $0$ when considering the controls from \eqref{eq:zerovariancecontrolprelim}. The log-efficiency proved in Theorem \ref{theo:logefficient} shows that this is not only the case for classical solutions to \eqref{eq:HJBequation}, but even certain classical subsolutions.

\subsection{Statement of main results}\label{subsec:mainresultsstatements}

We are now ready to state our main results, Theorems \ref{theo:logefficient} and \ref{theo:expansionanalysis}, the proofs of which are contained in Section \ref{sec:proofs}. The first is on designing a log-efficient importance sampling scheme, and parallels Theorem 8.1 in \cite{DW2} and Theorem 4.1 in \cite{DSW}. It requires the additional two assumptions:
\begin{assumption}\label{assumption:forlogefficiency}
Suppose:
\begin{enumerate}[label=(B\arabic*)]
\item \label{assumption:Gbounded}$G$ in equations \eqref{eq:desiredexpectation} and \eqref{eq:HJBequation} is bounded and continuous on $\mc{P}(\R^d)$
\item \label{assumption:sigmabounded}$\sigma$ is bounded.
\end{enumerate}
\end{assumption}
\begin{theo}\label{theo:logefficient}
Suppose $\Psi$ is a classical subsolution to \eqref{eq:HJBequation} in the sense of Definition \ref{defi:subsolution} and that Assumptions \ref{assumption:forLDP} and \ref{assumption:forlogefficiency} hold. Then the importance sampling scheme from \eqref{eq:deltahat} with the choice of controls 
\begin{align}\label{eq:optimalcontrol}
v_i^N(t,x_1,...,x_N)=-\sigma^\top\left(x_i,\frac{1}{N}\sum_{j=1}^N \delta_{x_j}\right)\partial_\mu \Psi\left(t,\frac{1}{N}\sum_{j=1}^N \delta_{x_j}\right)[x_i]
\end{align}
has the property that 
\begin{align*}
\lim_{N\toinf}-\frac{1}{N}\log R(\hat{\delta}_N)&\geq \Psi(s,\nu)-\gamma_1,
\end{align*}
where $R(\hat{\delta}_N)$ is as in \eqref{eq:Rdeltahat}, $\gamma_1$ is as in \eqref{eq:gammas}, $s$ is the initial time in \eqref{eq:IPS}, and $\nu$ is as in \ref{assumption:convergenceofinitialconditions}. In particular, if $\Psi(s,\nu)\geq \gamma_2-\gamma_1$, then this importance sampling scheme performs better than the standard Monte Carlo estimator in the sense that it admits an expansion of the form \eqref{eq:expandimportancesamplingrelativeerror} such that $\hat{\gamma}\leq 2\gamma_1-\gamma_2$ (recalling here \eqref{eq:expandmontecarlorelativeerror}).

Moreover, if $\Psi(s,\nu)=\gamma_1$, then this importance sampling scheme is log-efficient in the sense of Definition \ref{definition:logefficiency}.
\end{theo}
As we discussed in Section \ref{sec:background}, \cite{VW} provides sufficient conditions in terms of a concept of ``uniform log-efficiency'' for the log-efficient importance sampling estimator in their small-noise setting to have vanishing relative error (Definition \ref{def:vanishingrelativeerror}). It is not immediately clear how to verify such conditions in our setting; their proof relies heavily on local regularity results from the theory of (standard) first-order HJB equations which we cannot directly apply to our infinite-dimensional HJB equations. In \cite{NonAsymptotic}, the author takes a different approach, where an asymptotic expansion in the small-noise parameter of the HJB equation associated to the Doob $h$-transform (see \eqref{eq:prelimitHJB} herein) is used to establish that their log-efficient importance sampling estimator in fact has vanishing relative error. This is done by establishing an expression of the form \eqref{eqn:Rdeltahatexpansion}, where the $a_k$'s are expressed in terms of solutions to different PDEs. Vanishing relative error can then be established by showing $a_0=a_1=0$.

In order to carry out a similar analysis here, we will likewise derive PDEs whose solutions can be used to carry out an asymptotic expansion analysis in Theorem \ref{theo:expansionanalysis}. Firstly, let $\phi_0$ denote a classical $C^{1,2}_b([0,T]\times\mc{P}_2(\R^d))$ solution to \eqref{eq:HJBequation}. Using the control from \eqref{eq:optimalcontrol} with $\phi_0$ in the place of $\Psi$, we obtain an expression for the numerator of $R(\hat{\delta}_N)$ from \eqref{eq:Rdeltahat}, which results from applying Girsanov's theorem, the Feynman-Kac's formula, and a log transformation. Concretely, letting $\Xi^N$ satisfy:

\begin{equation}\label{eq:kostas3.3}
  \begin{split}
\partial_t\Xi^N(t,\nu)&+\int_{\R^d}[b(z,\nu)-\sigma(z,\nu) v(t,\nu,z)]\cdot \partial_\mu \Xi^N(t,\nu)[z]+\frac{1}{2}\sigma\sigma^\top(z,\nu):\partial_z\partial_\mu \Xi^N(t,\nu)[z]\nu(dz)\\ 
&-\int_{\R^d}\frac{1}{2}|\sigma^\top(z,\nu) \partial_\mu \Xi^N(t,\nu)[z]|^2+|v(t,\nu,z)|^2\nu(dz)\\ 
&+\frac{1}{2N}\int_{\R^d}\sigma\sigma^\top(z,\nu):\partial^2_\mu\Xi^N(t,\nu)[z,z]\nu(dz)=0,\qquad t\in [0,T), \nu\in\mc{P}_2(\R^d),\\ 
\Xi^N(T,\nu)&=2G(\nu),\qquad \nu\in\mc{P}_2(\R^d),\\
v(t,\nu,x)&= -\sigma^\top(x,\nu)\partial_\mu\phi_0(t,\nu)[x],\qquad t\in[0,T),x\in\R^d,\nu\in\mc{P}_2(\R^d),
  \end{split}
\end{equation} 
we have 
\begin{align*}
\Xi^N(t,\mu^N_x)=-\frac{1}{N}\log\E\biggl[\exp(-2NG(\hat{\mu}^{N,t,x}_T))\prod_{i=1}^N (Z^{i,N,t,x})^2 \biggr],\quad t\in [0,T],x\in\R^{dN}.
\end{align*}

Note that this corresponds to Equation (3.3) in the small-noise setting of \cite{NonAsymptotic}. 

Expanding $\Xi^N(t,\mu)=\xi_0(t,\mu)+o(1)$, we expect that $\xi_0$ solves the ``zero viscosity HJB equation'' obtained by setting the final $\mc{O}(1/N)$ term in \eqref{eq:kostas3.3} equal to $0$, that is:
\begin{equation}\label{eq:xi0}
  \begin{split}
&\partial_t\xi_0(t,\mu)+\int_{\R^d}[b(z,\mu)-\sigma(z,\mu) v(t,\mu,z)]\cdot \partial_\mu \xi_0(t,\mu)[z]\\
&+\frac{1}{2}\sigma\sigma^\top(z,\mu):\partial_z\partial_\mu \xi_0(t,\mu)[z]\mu(dz)-\int_{\R^d}\frac{1}{2}|\sigma^\top(z,\mu) \partial_\mu \xi_0(t,\mu)[z]|^2
\\&+|\sigma^\top(z,\mu) \partial_\mu\phi_0(t,\mu)[z]|^2\mu(dz)=0,\qquad t\in [0,T),\mu\in\mc{P}_2(\R^d),\\ 
\xi_0(T,\mu)&=2G(\mu),\qquad \mu\in\mc{P}_2(\R^d),
  \end{split}
\end{equation}
where $v$ is as in \eqref{eq:kostas3.3}. Note that, by the same logic, expanding $\Phi^N$ from \eqref{eq:prelimitHJB} as $\Phi^N(t,\mu)=\phi_0(t,\mu)+o(1)$, we expect that $\phi_0(t,\mu)$ solves \eqref{eq:HJBequation}. 

We will also consider higher order terms in the following series expansions of  $\Phi^N$ from \eqref{eq:prelimitHJB} and $\Xi^N$ from \eqref{eq:kostas3.3}: 
\begin{align}\label{eq:PhiNexpansion}
\Phi^N(t,\mu) =\phi_0(t,\mu)+\frac{1}{N}\phi_1(t,\mu)+\frac{1}{N^2}\phi_2(t,\mu)+...+\frac{1}{N^k}\phi_k(t,\mu)+o(1/N^k),
\end{align}
and 
\begin{align}\label{eq:XiNexpansion}
\Xi^N(t,\mu)=\xi_0(t,\mu)+\frac{1}{N}\xi_1(t,\mu)+\frac{1}{N^2}\xi_2(t,\mu)+...+\frac{1}{N^k}\xi_k(t,\mu)+o(1/N^k)
\end{align}

for some $k\in \bb{N}$ and all $t\in [0,T],\mu\in\mc{P}_2(\R^d)$.

Matching terms of the same order upon inserting this ansatz into the equations for $\Phi^N$ and $\Xi^N$, we expect that 
\begin{equation}\label{eq:phik}
  \begin{split}
\partial_t\phi_k(t,\mu)&+\int_{\R^d}b(z,\mu)\cdot \partial_\mu \phi_k(t,\mu)[z]+\frac{1}{2}\sigma\sigma^\top(z,\mu):\partial_z\partial_\mu \phi_k(t,\mu)[z]\mu(dz)\\
&+\int_{\R^d}\frac{1}{2}\sigma\sigma^\top(z,\mu):\partial^2_\mu \phi_{k-1}(t,\mu)[z,z]\mu(dz)\\ 
&-\int_{\R^d}\sum_{i+j=k,0\leq i<j}\langle \sigma^\top(z,\mu) \partial_\mu \phi_i(t,\mu)[z],\sigma^\top(z,\mu) \partial_\mu \phi_j(t,\mu)[z]\rangle\\ 
&+\frac{1}{2}|\sigma^\top(z,\mu)\partial_\mu\phi_{k/2}(t,\mu)[z]|^2\1_{k/2\in\bb{N}}\mu(dz)=0,\quad \mu\in\mc{P}_2(\R^d),t\in[0,T),\\ 
\phi_k(T,\mu)&=0,\qquad \mu\in\mc{P}_2(\R^d)
  \end{split}
\end{equation}
and 
\begin{equation}\label{eq:xik}
  \begin{split}
\partial_t\xi_k(t,\mu)&+\int_{\R^d}\left[b(z,\mu)+\sigma(z,\mu) \sigma^\top(z,\mu) \partial_\mu\phi_0(t,\mu)[z]\right]\cdot \partial_\mu \xi_k(t,\mu)[z]\\ 
&+\frac{1}{2}\sigma\sigma^\top(z,\mu):\partial_z\partial_\mu \xi_k(t,\mu)[z]\mu(dz)\\ 
&+\int_{\R^d}\frac{1}{2}\sigma\sigma^\top(z,\mu):\partial^2_\mu \xi_{k-1}(t,\mu)[z,z]\mu(dz)\\ 
&-\int_{\R^d}\sum_{i+j=k,0\leq i<j}\langle \sigma^\top(z,\mu) \partial_\mu \xi_i(t,\mu)[z],\sigma^\top(z,\mu) \partial_\mu \xi_j(t,\mu)[z]\rangle\\ 
&+\frac{1}{2}|\sigma^\top(z,\mu)\partial_\mu\xi_{k/2}(t,\mu)[z]|^2\1_{k/2\in\bb{N}}\mu(dz)=0,\qquad\mu\in\mc{P}_2(\R^d),t\in[0,T),\\ 
\xi_k(T,\mu)&=0,\qquad \mu\in\mc{P}_2(\R^d)
\end{split}
\end{equation}
for all $k\geq 1$. Note that these correspond to the PDEs found in Theorem 3.3 of \cite{NonAsymptotic}.

Both of these equations fit the form of Equation (1.2) in \cite{DF1}, so under sufficient regularity assumptions we would have by Theorem 3.8 therein that for $k\geq 1$:
\begin{equation}\label{eq:phikstochrep}
  \begin{split}
\phi_k(t,\mu)&=\E\biggl[\int_t^T\frac{1}{2}\sigma\sigma^\top(Z^{t,\mu}_\tau,\mc{L}(Z^{t,\mu}_\tau)): \partial^2_\mu\phi_{k-1}(\tau,\mc{L}(Z^{t,\mu}_\tau))[Z^{t,\mu}_\tau,Z^{t,\mu}_\tau] \\ 
&-\sum_{i+j=k,0\leq i<j<k}\langle \sigma^\top(Z^{t,\mu}_\tau,\mc{L}(Z^{t,\mu}_\tau)) \partial_\mu \phi_i(\tau,\mc{L}(Z^{t,\mu}_\tau))[Z^{t,\mu}_\tau],\sigma^\top(Z^{t,\mu}_\tau,\mc{L}(Z^{t,\mu}_\tau)) \partial_\mu \phi_j(\tau,\mc{L}(Z^{t,\mu}_\tau))[Z^{t,\mu}_\tau]\rangle \\ 
&-\frac{1}{2}|\sigma^\top(Z^{t,\mu}_\tau,\mc{L}(Z^{t,\mu}_\tau))\partial_\mu\phi_{k/2}(\tau,\mc{L}(Z^{t,\mu}_\tau))[Z^{t,\mu}_\tau]|^2\1_{k/2\in\bb{N}}d\tau\biggr], \\ 
dZ^{t,\mu}_\tau&=\biggl[b(Z^{t,\mu}_\tau,\mc{L}(Z^{t,\mu}_\tau))-\sigma\sigma^\top(Z^{t,\mu}_\tau,\mc{L}(Z^{t,\mu}_\tau))\partial_\mu\phi_0(\tau,\mc{L}(Z^{t,\mu}_\tau))[Z^{t,\mu}_\tau]\biggr]d\tau+\sigma(Z^{t,\mu}_\tau,\mc{L}(Z^{t,\mu}_\tau)) dW_\tau,Z^{t,\mu}_t\sim \mu
  \end{split}
\end{equation}
and 
\begin{equation}\label{eq:xikstochrep}
  \begin{split}
\xi_k(t,\mu)&=\E\biggl[\int_t^T\frac{1}{2}\sigma\sigma^\top(Y^{t,\mu}_\tau,\mc{L}(Y^{t,\mu}_\tau)) :\partial^2_\mu\xi_{k-1}(\tau,\mc{L}(Y^{t,\mu}_\tau))[Y^{t,\mu}_\tau,Y^{t,\mu}_\tau] \\ 
&-\sum_{i+j=k,0\leq i<j<k}\langle \sigma^\top(Y^{t,\mu}_\tau,\mc{L}(Y^{t,\mu}_\tau)) \partial_\mu \xi_i(\tau,\mc{L}(Y^{t,\mu}_\tau))[Y^{t,\mu}_\tau],\\ 
&\hspace{6cm}\sigma^\top(Y^{t,\mu}_\tau,\mc{L}(Y^{t,\mu}_\tau)) \partial_\mu \xi_j(\tau,\mc{L}(Y^{t,\mu}_\tau))[Y^{t,\mu}_\tau]\rangle\\ 
&-\frac{1}{2}|\sigma^\top(Y^{t,\mu}_\tau,\mc{L}(Y^{t,\mu}_\tau))\partial_\mu\xi_{k/2}(\tau,\mc{L}(Y^{t,\mu}_\tau))[Y^{t,\mu}_\tau]|^2\1_{k/2\in\bb{N}}d\tau\biggr]\\ 
dY^{t,\mu}_\tau&=\biggl[b(Y^{t,\mu}_\tau,\mc{L}(Y^{t,\mu}_\tau))+\sigma\sigma^\top(Y^{t,\mu}_\tau,\mc{L}(Y^{t,\mu}_\tau))\partial_\mu \phi_0(\tau,\mc{L}(Y^{t,\mu}_\tau))[Y^{t,\mu}_\tau],\\
&-\sigma\sigma^\top(Y^{t,\mu}_\tau,\mc{L}(Y^{t,\mu}_\tau))\partial_\mu\xi_0(\tau,\mc{L}(Y^{t,\mu}_\tau))[Y^{t,\mu}_\tau]\biggr]d\tau+\sigma(Y^{t,\mu}_\tau,\mc{L}(Y^{t,\mu}_\tau)) dW_\tau,Y^{t,\mu}_t\sim \mu.
  \end{split}
\end{equation}

As we will see, the $a_k$'s from \eqref{eqn:Rdeltahatexpansion} will be realized in terms of the solutions to the PDEs \eqref{eq:phik} and \eqref{eq:xik} via the representation:
\begin{align*}
a_k=2\phi_k(s,\mu^N_y)-\xi_k(s,\mu^N_y),\quad k\in\bb{N},
\end{align*}
where $s\in [0,T]$ is the initial time for the particles \eqref{eq:IPS}, $y\in \oplus_{i=1}^\infty \R^d$ encodes their initial conditions, and we use the empirical measure notation from the end of Section \ref{sec:notation}. See the proof of Theorem \ref{theo:logefficient} in Subsection \ref{subsec:theo2proof} for more details.

 In order to generate a single realization of $\mu^{N,s,y}$, we need to simulate a system of $N$ SDEs. Thus, with $R(\hat\delta_N)$ as in \eqref{eq:Rdeltahat}, the quantity
\begin{align}\label{eq:totalnumberofparticles}
T(N)\coloneqq N[R(\hat\delta_N)-1],
\end{align}
is proportional to the maximum number of particles which need to be simulated (and therefore the computational work required) in order for the relative error from \eqref{eq:importancesamplingrelativeerror} to be below a certain threshold. Put explicitly, to achieve a certain relative error $c$, we require at least $\frac{R(\hat\delta_N)-1}{c^2}$ samples of $(\hat{\mu}^{N,s,y},Z^{1,N,s,y}...,Z^{N,N,s,y})$. 
Therefore, the total number of particles simulated will need to be at least $T(N)/c^2$. Note, of course, that even if $R(\hat\delta_N)=1,$ that we will need to simulate at least $N$ particles to get a realization of $\hat{\mu}^{N,s,y}$.  

This motivates us to go beyond studying $\phi_0,\xi_0,\phi_1,$ and $\xi_1$ in order to prove $a_0=a_1=0$ in \eqref{eqn:Rdeltahatexpansion}, and to further study $a_2$. We find that if we assume third-order expansions of the form \eqref{eq:PhiNexpansion} and \eqref{eq:XiNexpansion} that have sufficient regularity properties, then not only does the relative error vanish, but it vanishes fast enough that $T(N)$ is bounded as $N\toinf$. This further suggests that the relative error of our importance sampling scheme is expected to vanish like $\mc{O}(1/\sqrt{N})$ as $N\toinf$ if the coefficients of the interacting particle system \eqref{eq:IPS} and the target function $G$ from \eqref{eq:desiredexpectation} are sufficiently regular.

\begin{theo}\label{theo:expansionanalysis}
Let Assumption \ref{assumption:forlogefficiency} and Assumption \ref{assumption:forLDP} \ref{assumption:convergenceofinitialconditions}-\ref{assumption:weaksenseuniqueness} hold. Assume also that $|b(x,\mu)|^2\leq C(1+|x|^2+\int|z|^2\mu(dz))$ for all $x\in\R^d$ and $\mu\in\mc{P}_2(\R^d)$. Consider $\Phi^N,\Xi^N:[0,T]\times\mc{P}_2(\R^d)\tto \R$ from \eqref{eq:prelimitHJB} and \eqref{eq:kostas3.3} respectively. Suppose these are unique $C^{1,2}_b([0,T]\times\mc{P}_2(\R^d))$ solutions admitting expansions of the form \eqref{eq:PhiNexpansion} and \eqref{eq:XiNexpansion} respectively up to $k=1$, and that $\phi_0,\xi_0,\phi_1,\xi_1,$ are the unique $C^{1,2}_b([0,T]\times\mc{P}_2(\R^d))$ solutions to \eqref{eq:HJBequation},\eqref{eq:xi0},\eqref{eq:phik},\eqref{eq:xik} with $k=1$ respectively, with $\phi_1,\xi_1$ admitting the stochastic representations \eqref{eq:phikstochrep},\eqref{eq:xikstochrep} with $k=1$ respectively. Then the importance sampling scheme from Theorem \ref{theo:logefficient} with $\phi_0$ in the place of $\Psi$ has vanishing relative error in the sense of Definition \ref{def:vanishingrelativeerror}. 

If further we assume expansions of the form \eqref{eq:PhiNexpansion} and \eqref{eq:XiNexpansion} up to $k=2$ and that $\phi_2,\xi_2$ are the unique $C^{1,2}_b([0,T]\times\mc{P}_2(\R^d))$ solutions to \eqref{eq:phik},\eqref{eq:xik} with $k=2$ respectively admitting the representations \eqref{eq:phikstochrep},\eqref{eq:xikstochrep} with $k=2$ respectively, then 
\begin{align}\label{eq:TNlimit}
\lim_{N\toinf}T(N)=\E\biggl[\int_s^T|\sigma^\top(\hat{X}^{u,s,\nu}_t,\mc{L}(\hat{X}^{u,s,\nu}_t)) \partial_\mu \phi_1(t,\mc{L}(\hat{X}^{u,s,\nu}_t))[\hat{X}^{u,s,\nu}_t]|^2dt\biggr]
\end{align}
where $\hat{X}^{u,s,\nu}$ is as in \eqref{eq:controlledMcKeanVlasov} with $u(t)=-\sigma^\top(\hat{X}^{u,s,\nu}_t,\mc{L}(\hat{X}^{u,s,\nu}_t))\partial_\mu \phi_0(t,\mc{L}(\hat{X}^{u,s,\nu}_t))[\hat{X}^{u,s,\nu}_t],t\in[s,T]$.

\end{theo}

\begin{remark}\label{remark:extensions}
The results of Theorems \ref{theo:logefficient} and \ref{theo:expansionanalysis} can easily be extended to the situation where we modify the desired expectation \eqref{eq:desiredexpectation} to 
\begin{align*}
\E\left[\exp\left(-N\left[G(\mu^{N,s,y}_T)+\int_s^Tf(\mu^{N,s,y}_t)dt\right]\right) \right]
\end{align*}
for sufficiently regular $f:\mc{P}_2(\R^d)\tto \R$. In this case one takes $F(\mu)=G(\mu_T)+\int_s^Tf(\mu^{N,s,y}_t)dt$ in \eqref{eq:LDP} and modifies $H_0$ from \eqref{eq:Hamiltonian} to
\begin{align*}
H_0(\mu,p,\Gamma)=- \int_{\R^d}b(x,\mu)\cdot p(x)-\frac{1}{2}p(x)\cdot[\sigma\sigma^\top(x,\mu) p(x)]+\frac{1}{2}\sigma\sigma^\top(x,\mu): \Gamma(x) \mu(dx)-f(\mu).
\end{align*}

One should also be able to extend these methods in order to design importance sampling schemes for the probabilities of rare events. That is, rather than estimating \eqref{eq:desiredexpectation}, one may want to estimate
\begin{align*}
\Prob(\mu^{N,s,y}_T\in A)
\end{align*}
for some $A\subset \mc{P}(\R^d)$. This formally corresponds to taking 
\begin{align*}
G(\mu)=\begin{cases}
+\infty,&\mu\in A^c\\ 
0,&\mu\in A
\end{cases}
\end{align*}
in \eqref{eq:desiredexpectation} and \eqref{eq:HJBequation}. It is well known that this extension can be made in the small-noise setting --- see, e.g., \cite{DSW} Proposition 4.2. We refrain from performing this analysis here, but we plan to extend the methods presented to not only probabilities at finite time, but also exit probabilities and mean first passage times in future work. Thanks to the many parallels to the small-noise setting, we expect that designing importance sampling schemes related to exit events will require the study of HJB equations of the form \eqref{eq:HJBequation} restricted to some subset of the space $\mc{P}_2(\R^d)$ with boundary conditions (compare with, e.g. Equations (2.4)-(2.5) in \cite{DSZ}). Such equations are already beginning to be studied in the context of optimal stopping problems for McKean-Vlasov equations --- see, e.g., \cite{TTZ}.
\end{remark}
\begin{remark}\label{remark:ontheassumptionsforexpansionanalysis}
The additional assumption that $b$ has at most linear growth in the statement of Theorem \ref{theo:expansionanalysis} ensures that the unique solutions of the Feynman-Kac equations \eqref{eq:prelimitBKE} and \eqref{eq:prelimitFK} are given by their appropriate stochastic representations, i.e. the denominator and numerator of $R(\hat{\delta}_N)$ from \eqref{eq:Rdeltahat}, respectively. This can of course hold under weaker conditions. 

Moreover, although we do not necessarily assume \ref{assumption:tightness} for Theorem \ref{theo:expansionanalysis}, it is unlikely that there are situations in which the existence and uniqueness of classical solutions to the HJB equations \eqref{eq:prelimitHJB} and \eqref{eq:kostas3.3} hold for which \ref{assumption:tightness} doesn't hold. In particular, the required assumptions for Theorem \ref{theo:expansionanalysis} should be much stricter than those imposed for Theorems \ref{theo:LDP} and \ref{theo:logefficient}.

The unique stochastic representations assumed for $\phi_k,\xi_k,k=1,2$ solving the PDEs \eqref{eq:phik} and \eqref{eq:xik} are known to hold under fairly weak assumptions --- see Theorem 3.8 in \cite{DF1}. Moreover, these representations for the prefactor terms in the expansion yield parallels to those found in the small-noise setting --- see the discussion towards the end of Subsection \ref{subsec:theo2proof}. The reason for the assumptions in Theorem \ref{theo:expansionanalysis} being stated as such is that little is available in the current literature in terms of sufficient conditions for existence and uniqueness of classical solutions to \eqref{eq:HJBequation}, let alone on obtaining the desired formal expansions in $N$. Indeed, even the convergence of $\Phi^N$ to $\phi_0$ (corresponding to the first-order expansion) has only been studied in the case where the particles have common noise --- see \cite{GPW}. 

The assumptions made in terms of this expansion essentially mimic the conclusion of Theorem 3.3 in \cite{NonAsymptotic}, which is a consequence of Theorem 5.1 in \cite{FJ}. We expect that analogous conditions to those of these theorems should be able to be found to be sufficient for the expansion analysis of Theorem \ref{theo:expansionanalysis} to go through. This is an interesting avenue for future research.

More generally, on p. 1781-1782 of \cite{VW}, it is discussed how if a discontinuity of their optimal control is anything more exotic than a single curve, they are unable to prove log-efficiency. In general the interplay between regularity of solutions to the zero-viscosity HJB equation and the properties of the relative error of the importance sampling scheme is an interesting open problem even in the small-noise setting, though there there is a wealth of numerical evidence suggesting that in many situations importance sampling schemes derived from non-differentiable subsolutions of the zero-viscosity HJB equation may have bounded or even vanishing relative error. 

We contribute to these numerical findings in the context of weakly interacting diffusions in the numerical examples of Subsection \ref{subsection:nonsmooth}, where we construct subsolutions for modifications of the linear-quadratic regime discussed in Section \ref{sec:LQ} such that the corresponding optimal controls \eqref{eq:optimalcontrol} are discontinuous. We find that, depending on the nature of the discontinuity, the relative error of our importance sampling scheme can be expected to grow sublinearly, or even vanish --- see Tables \ref{table:absinside} and \ref{table:absoutside}, respectively.

\end{remark}

\section{A Class of Examples: The Linear-Quadratic Regime}\label{sec:LQ}

We now consider a class of HJB equations of the form \eqref{eq:HJBequation} for which explicit solutions are known. Consider the setting where 
\begin{equation}\label{eq:LQsetup}
    \begin{split}
b(x,\mu)&=b_0+Bx+\bar{B}\int_{\R^d}z\mu(dz),\\ 
\sigma(x,\mu)&=\sigma,\\ 
G(\mu)&=\int_{\R^d}z^\top P_2 z+p_1\cdot z\mu(dz)+\biggl[\int_{\R^d}z\mu(dz)\biggr]^\top \bar{P}_2\int_{\R^d}z\mu(dz)+p_2
    \end{split}
\end{equation}
for $b_0,p_1\in\R^d$, $\sigma\in \R^{d\times 1}$ (so $m=1$), $p_2\in\R$, and $P_2,B,\bar{B},\bar{P}_2\in\R^{d\times d}$ such that $P_2,\bar{P}_2$ are symmetric and positive semidefinite. By Section 4 of \cite{PW}, the unique classical solution to \eqref{eq:HJBequation} is given by 
\begin{align*}
\Psi(t,\mu)=\int_{\R^d}z^\top \Lambda(t)z\mu(dz)+\biggl[\int_{\R^d}z\mu(dz)\biggr]^\top [\Gamma(t)-\Lambda(t)]\int_{\R^d}z\mu(dz)+\gamma(t)\cdot \int_{\R^d}z\mu(dz)+\chi(t),
\end{align*}
where:
\begin{equation}\label{eq:RiccatiEqns}
    \begin{split}
\dot\Lambda(t)+\Lambda(t)B+B^\top \Lambda(t)-2\Lambda(t)\sigma\sigma^\top\Lambda^\top(t)&=0,\Lambda(T)=P_2,\\ 
\dot\Gamma(t)+\Gamma(t)[B+\bar{B}]+[B+\bar{B}]^\top\Gamma(t)-2\Gamma(t)\sigma\sigma^\top\Gamma^\top(t)&=0,\Gamma(T)=P_2+\bar{P}_2,\\ 
\dot\gamma(t)+[B+\bar{B}]^\top\gamma(t)-2\Gamma(t)\sigma\sigma^\top\gamma(t)+2\Gamma(t)b_0&=0,\gamma(T)=p_1,\\ 
\dot\chi(t)-\frac{1}{2}\gamma^\top(t)\sigma\sigma^\top\gamma(t)+\gamma(t)\cdot b_0+\sigma^\top\Lambda(t)\sigma&=0,\chi(T)=p_2.
    \end{split}
\end{equation}
The above Riccati equations admit unique solutions $\gamma\in\R^d,\chi\in\R,\Lambda,\Gamma\in\R^{d\times d}$ such that $\Lambda,\Gamma$ are symmetric and positive semidefinite for all $t$. 

Then, using:
\begin{align*}
\partial_\mu \Psi(t,\mu)[z]=2\Lambda(t) z+2[\Gamma(t)-\Lambda(t)]\int_{\R^d}z\mu(dz) +\gamma(t)
\end{align*}
the controls from \eqref{eq:optimalcontrol} in Theorem \ref{theo:logefficient} are given by
\begin{equation}\label{eq:optimalcontrolLQregime}
    \begin{split}
v_i^N(t,x_1,...,x_N)&=-\sigma^\top\partial_\mu \Psi(t,\mu^N_x)[x_i]\\ 
&=-\sigma^\top\biggl[2\Lambda(t)x_i+2 [\Gamma(t)-\Lambda(t)]\int_{\R^d}z\mu^N_x(dz)+\gamma(t)\biggr].
    \end{split}
    \end{equation}

Note that using Remark 5.2 in \cite{BP}, we can extend to the case where $m>1$, though we refrain from doing so for simplicity. We could also allow for time dependence in the coefficients, but for the sake of simplicity we do not make this extension here.

Despite the fact that $G$ from \eqref{eq:LQsetup} does not satisfy the Assumption \ref{assumption:forlogefficiency} \ref{assumption:Gbounded}, as we will see in Remark \ref{remark:zerorelativeerrorLQ}, the controls \eqref{eq:optimalcontrolLQregime} yield not only the log-efficiency proved in Theorem \ref{definition:logefficiency}, but in fact zero relative error for all $N$. 

Our first example from the linear-quadratic regime is chosen to have symmetries such that the relative error of the standard Monte Carlo estimator $\delta_N$ and of the importance sampling estimator $\hat{\delta}_N$ \eqref{eq:deltahat} with the choice of control \eqref{eq:optimalcontrolLQregime} are both easily computable.

\begin{example}\label{example:firstexample}
Let $G(\mu)  = \int_\R x \mu(dx),d=m=1$, and 
\begin{align*}
X^{i,N,s,y}_t=y_i+\int_s^t X^{i,N,s,y}_\tau - \frac{1}{N}\sum_{j=1}^N X^{j,N,s,y}_\tau d\tau + \sigma W^i_t.
\end{align*}
Symmetries in this problem allow us to calculate the target expectation \eqref{eq:desiredexpectation}, as well as the Monte Carlo relative error $\rho(\delta_N)$ and importance sampling relative error $\rho(\hat{\delta}_N)$ \eqref{eq:importancesamplingrelativeerror}, explicitly.

We have 
\begin{align*}
\E\biggl[\exp(-NG(\mu^N_T)) \biggr] & =\E\biggl[\exp\biggl(-\sum_{i=1}^N\biggl\lbrace y_i+\int_s^T X^{i,N,s,y}_t - \frac{1}{N}\sum_{j=1}^N X^{j,N,s,y}_t dt + \sigma W^i_T\biggr\rbrace\biggr)\biggr] \\ 
& = \exp\biggl(-\sum_{i=1}^Ny_i\biggr)\E\biggl[\exp\biggl(-\sigma W^1_T\biggr)\biggr]^N\\ 
& = \exp\biggl(-\sum_{i=1}^Ny_i\biggr)\exp(N\sigma^2 (T-s)/2),
\end{align*}
and similarly:
\begin{align*}
\E\biggl[\exp(-2NG(\mu^N_T)) \biggr] & =\E\biggl[\exp\biggl(-2\sum_{i=1}^N\biggl\lbrace y_i+\int_s^T X^{i,N,s,y}_t - \frac{1}{N}\sum_{j=1}^N X^{j,N,s,y}_t dt + \sigma W^i_T\biggr\rbrace\biggr)\biggr] \\ 
& = \exp\biggl(-2\sum_{i=1}^Ny_i\biggr)\exp(2N\sigma^2 (T-s)),
\end{align*}
where in the last step we have used the formula for the moment generating function of a normal random variable. Thus:
\begin{align*}
\rho(\delta_N)&=  \frac{1}{\sqrt{M}}\sqrt{\frac{\exp(2N\sigma^2 (T-s))}{\exp(N\sigma^2 (T-s))}-1}= \frac{1}{\sqrt{M}}\sqrt{\exp(N\sigma^2 (T-s))-1}.
\end{align*}

Now, to construct the control \eqref{eq:optimalcontrolLQregime} for our importance sampling scheme, we have that in the setup of \eqref{eq:LQsetup}, $b_0=P_2=\bar{P}_2=\bar{p}_1=p_2=0$, $B=p_1=1$, and $\bar{B}=-1$. Thus the solution to the system of ODEs \eqref{eq:RiccatiEqns} is $\Lambda(t)=\Gamma(t)=0$, $\gamma(t)=1$, and $\chi(t) = \frac{1}{2}\sigma^2(t-T)$, and so the solution to the HJB equation \eqref{eq:HJBequation} is
\begin{align*}
\Psi(t,\mu)& = \int_\R x \mu(dx)+\frac{1}{2}\sigma^2(t-T).
\end{align*}
Our control from Theorem \ref{theo:logefficient} is given by
\begin{align*}
\bar{v}(s,x,\mu)\equiv-\sigma.
\end{align*}
Then, to compute the relative error of the importance sampling scheme, we have
\begin{align*}
\rho(\hat{\delta}_N) & = \frac{1}{\sqrt{M}}\sqrt{\frac{\E\biggl[\exp(-2\sum_{i=1}^N \hat{X}^{i,N,s,y}_T)\prod_{i=1}^N \exp\biggl(2\sigma \hat{W}^i_T - \int_s^T \sigma^2dt\biggr) \biggr]}{\E\biggl[\exp(-NG(\mu^N)) \biggr]^2}-1}\\
& = \frac{1}{\sqrt{M}}\sqrt{\frac{\E\biggl[\exp(-2\sum_{i=1}^N (\hat{X}^{i,N,s,y}_T-\sigma \hat{W}^i_T))\biggr]\exp\biggl(- N\sigma^2(T-s)\biggr)}{\E\biggl[\exp(-NG(\mu^N_T)) \biggr]^2}-1}
\end{align*}
where
\begin{align*}
\hat{X}^{i,N,s,y}_t=y_i+\int_s^t \hat{X}^{i,N,s,y}_\tau - \frac{1}{N}\sum_{j=1}^N \hat{X}^{j,N,s,y}_\tau -\sigma^2 d\tau + \sigma \hat{W}^i_t.
\end{align*}
Continuing, we have
\begin{align*}
\rho(\hat{\delta}_N) & = \frac{1}{\sqrt{M}}\sqrt{\frac{\E\biggl[\exp\biggl(-2\sum_{i=1}^N\biggl\lbrace y_i+\int_s^T \hat{X}^{i,N,s,y}_t - \frac{1}{N}\sum_{j=1}^N \hat{X}^{j,N,s,y}_t -\sigma^2 dt\biggr\rbrace\biggr)\biggr]\exp\biggl(- N\sigma^2(T-s)\biggr)}{\E\biggl[\exp(-NG(\mu^N_T)) \biggr]^2}-1}\\ 
& = \frac{1}{\sqrt{M}}\sqrt{\frac{\E\biggl[\exp\biggl(-2\sum_{i=1}^N (y_i-(T-s)\sigma^2) \biggr)\biggr]\exp\biggl(- N\sigma^2(T-s)\biggr)}{\E\biggl[\exp(-NG(\mu^N_T)) \biggr]^2}-1}\\ 
& = \frac{1}{\sqrt{M}}\sqrt{\frac{\exp\biggl(-2\sum_{i=1}^N y_i\biggr)\exp\biggl(2N\sigma^2(T-s)\biggr)\exp\biggl(- N\sigma^2(T-s)\biggr)}{\exp\biggl(-2\sum_{i=1}^Ny_i\biggr)\exp(N\sigma^2 (T-s))}-1}\\ 
&=0.
\end{align*}

This shows that where the Monte Carlo estimator has relative error which increases exponentially as $N\toinf$ or $T\toinf$, our importance sampling estimator actually has zero variance for all $N$ and $T$.
\end{example}
The reader familiar with small-noise importance sampling schemes for SDEs may realize at this point that the controls computed in Example \ref{example:firstexample} can also be derived using the classical small-noise theory. In that vein, we make the following remark relating a special case of our importance sampling method to the classical small-noise importance sampling scheme:
\begin{remark}\label{remark:comparisonwiththesmallnoisecase}
{}In the linear-quadratic regime, $Y^{N,s,y}_t=\frac{1}{N}\sum_{i=1}^N X^{i,N,s,y}$ satisfies the small-noise SDE:
\begin{align*}
dY^{N,s,y}_t&=[b_0+(B+\bar{B})Y^{N,s,y}_t]dt+\frac{\sigma}{\sqrt{N}}d\mc{W}^N_t,\\ 
Y^{N,s,y}_s&=\frac{1}{N}\sum_{i=1}^Ny_i,
\end{align*}
where for each $N$, $\mc{W}^N$ is the standard one-dimensional Brownian motion $\mc{W}^N_t=\frac{1}{\sqrt{N}}\sum_{i=1}^NW^i_t$. Thus, when $G(\mu) = g(\int_{\R^d} x\mu(dx))$ for some $g: \R^d\to \R^d$, estimating functionals of the form 
\begin{align*}
\E[\exp(-NG(\mu^{N,s,y}_T))] = \E[\exp(-Ng(Y^{N,s,y}_T))]
\end{align*}
can also be done using standard small-noise importance sampling for SDEs, i.e., the appropriate control is $\hat{v}(t,y)=-\sigma^\top\partial_y\psi(t,y)$ where $\psi$ satisfies the standard first-order HJB equation:
\begin{equation}\label{eq:linearquadratichjbsmallnoise}
\begin{split}
    -\partial_t\psi(t,y)&-[b_0+(B+\bar{B})y]\cdot \partial_y\psi(t,y)+\frac{1}{2}|\sigma^\top \partial_y\psi(t,y)|^2=0,\qquad y\in\R^d,t\in[0,T),\\ 
\psi(T,y)&=g(y),\qquad y\in \R^d
\end{split}
\end{equation}

 --- see, e.g., equations (1.12) and (2.11) in \cite{VW}.  The resulting importance sampling estimator is given by
\begin{align}\label{eq:smallnoiseimportancesamplinglinearcase}
\hat{\delta}_N=\frac{1}{M}\sum_{j=1}^M \exp(-Ng(\hat{Y}^{N,s,y,j}_T))\exp\left(-\sqrt{N}\int_s^T\hat{v}(t,\hat{Y}^{N,s,y,j}_t)\cdot d\hat{\mc{W}}^N_t-\frac{N}{2}\int_s^T |\hat{v}(t,\hat{Y}^{N,s,y,j}_t)|^2dt\right)
\end{align}
where $\hat{Y}^{N,s,y,j}_t$ denotes the $j$'th realization of $\hat{Y}^{N,s,y}_t$ solving
\begin{align*}
d\hat{Y}^{N,s,y}_t&=[b_0+(B+\bar{B})\hat{Y}^{N,s,y}_t+\sigma\hat{v}(t,\hat{Y}^{N,s,y}_t)]dt+\frac{\sigma}{\sqrt{N}}d\hat{\mc{W}}^N_t,\quad \hat{Y}^{N,s,y}_s=\frac{1}{N}\sum_{i=1}^Ny_i.
\end{align*} 
In our linear-quadratic regime \eqref{eq:LQsetup}, the choice of parameters $P_2=0,p_1=[p,...,p]^\top$ (so that $g(y)=py+y^\top \bar{P}_2y+p_2$) is therefore covered by the small-noise theory. 

To show that the two importance sampling estimators \eqref{eq:smallnoiseimportancesamplinglinearcase} and \eqref{eq:deltahat} are always identical in this case, we make the ansatz $\Psi(t,\nu) = \psi\left(t,\int_{\R^d}z\nu(dz)\right)$. Then $\partial_\mu\Psi(t,\nu)[z]=\partial_y\psi\left(t,\int_{\R^d}x\nu(dx)\right),\partial_z\partial_\mu\Psi(t,\nu)[z]\equiv 0, $ and \eqref{eq:HJBequation} is given by:
\begin{align*}
-\partial_t\psi\left(t,\int_{\R^d}z\nu(dz)\right)&-\biggl(b_0+[B+\bar{B}]\int_{\R^d}z\nu(dz)\biggr)\cdot\partial_y\psi\left(t,\int_{\R^d}z\nu(dz)\right)+\frac{1}{2}\left|\sigma^\top\partial_y\psi\left(t,\int_{\R^d}z\nu(dz)\right)\right|^2=0,\\ 
&\hspace{8cm}t\in[0,T),\nu\in\mc{P}_2(\R^d),\nonumber\\ 
\psi\left(T,\int_{\R^d}z\nu(dz)\right)&=g\left(\int_{\R^d}z\nu(dz)\right),\qquad \nu\in\mc{P}_2(\R^d),\nonumber
\end{align*}
which we see holds for any $\nu$ by substituting $y=\int_{\R^d}z\nu(dz)$ into \eqref{eq:linearquadratichjbsmallnoise}. Then the controls from \eqref{eq:optimalcontrol} in Theorem \ref{theo:logefficient} are given by 
\begin{align*}
v_i^N(t,x_1,...,x_N)=-\sigma^\top\partial_y\psi\left(t,\frac{1}{N}\sum_{i=1}^N x_i\right),\forall i\in \br{1,...,N}
\end{align*}
and the importance sampling estimator \eqref{eq:deltahat} is given by 

\begin{align*}
\hat{\delta}_N&= \frac{1}{M}\sum_{j=1}^M \exp\left(-Ng\left(\frac{1}{N}\sum_{j=1}^N \hat{X}^{j,N,s,y}_t\right)\right)\exp\biggl(\sum_{i=1}^N\biggl\lbrace\int_s^T \sigma^\top\partial_y\psi\left(t,\frac{1}{N}\sum_{j=1}^N \hat{X}^{j,N,s,y}_t\right)\cdot d\hat{W}^i_t \\ 
&- \frac{1}{2}\int_s^T \left|\sigma^\top\partial_y\psi\left(t,\frac{1}{N}\sum_{j=1}^N \hat{X}^{j,N,s,y}_t\right)\right|^2dt\biggr\rbrace\biggr)\\ 
&=\frac{1}{M}\sum_{j=1}^N \exp(-Ng(\hat{Y}^{N,s,y,j}_T))\exp\left(-\sqrt{N}\int_s^T\hat{v}(t,\hat{Y}^{N,s,y,j}_t)\cdot d\hat{\mc{W}}^N_t-\frac{N}{2}\int_s^T |\hat{v}(t,\hat{Y}^{N,s,y,j}_t)|^2dt\right),
\end{align*}
where in the first line above
\begin{align*}
d\hat{X}^{i,N,s,y}_t &= \left[b_0+B\hat{X}^{i,N,s,y}_t+\frac{1}{N}\bar{B}\sum_{j=1}^N\hat{X}^{j,N,s,y}_t -\sigma\sigma^\top\partial_y\psi\left(t,\frac{1}{N}\sum_{j=1}^N\hat{X}^{j,N,s,y}_t \right) \right]dt+\sigma d\hat{W}^i_t,\\ 
\hat{X}^{i,N,s,y}_s &= y_i\nonumber.
\end{align*}

Note, however, that not all examples in the linear-quadratic case \eqref{eq:LQsetup} can be framed as a small-noise SDE problem. For instance, even when $d=1$, $P_2\neq 0$, one needs to consider both $Y^{N,s,y}$ and $\tilde{Y}^{N,s,y}_t=\frac{1}{N}\sum_{i=1}^N (X^{i,N,s,y}_t)^2$. For the latter, the martingale term will be $\frac{2\sigma}{N}\sum_{i=1}^N \int_s^T X^{i,N,s,y}_td W^i_t$, which cannot be written in terms of $\mc{W}^N$. 
\end{remark}

Our next example \ref{example:linearGandIIDparticles} shows that our importance sampling scheme can yield zero relative error even in the situation where the particles from \eqref{eq:IPS} are non-interacting, and hence IID. Again, the fact that the relative error of the importance sampling scheme is zero can be viewed as the consequence of a more general principle discussed in Remark \ref{remark:IIDandlinearG}.

\begin{example}\label{example:linearGandIIDparticles}
Consider now the system:
\begin{align*}
X^{i,N,y}_t=y-\int_0^t X^{i,N,y}_s +  W^i_t
\end{align*}
where $d=m=1$, and we fix $s=0$ and suppress it in the notation for simplicity. We are now dealing with IID diffusions, that is, there is no interaction between the particles. 

Suppose we take $G(\mu)=\int_{\R}x^2\mu(dx).$ Then in the notation of \eqref{eq:LQsetup}, $\bar{P}_2=\bar{p}_1=p_1=\bar{B}=b_0=p_2=0$, $\sigma=P_2=1$, $B=-1$ and $y_i=y$ for all $i\in\bb{N}$. Using the known density of the Ornstein–Uhlenbeck process $X^{1,N,y}_T$, we get: 
\begin{align*}
\E\left[\exp(-NG(\mu^{N,y}_T))\right]& = \E\left[\exp\left(-\sum_{i=1}^N (X^{i,N,y}_T)^2\right)\right]\\ 
&=\E\left[\exp\left(-(X^{1,N,y}_T)^2\right)\right]^N\\ 
&=\exp\left(Ny^2/(1-2e^{2T})\right)(2-e^{-2T})^{-N/2}
\end{align*}
and similarly 
\begin{align*}
\E\left[\exp\left(-2NG(\mu^{N,y}_T)\right)\right]=\exp\left(2Ny^2/(2-3e^{2T})\right)(3-2e^{-2T})^{-N/2}.
\end{align*}
\
Then:
\begin{align*}
\rho(\delta_N)&=  \frac{1}{\sqrt{M}}\sqrt{\exp\left(2Ny^2\left[(2-3e^{2T})^{-1} - (1-2e^{2T})^{-1}\right]\right)\left(\frac{(2-e^{-2T})^2}{3-2e^{-2T}}\right)^{N/2}-1}.\\
\end{align*}

Meanwhile, the optimal control can be found from solving the ODEs \eqref{eq:RiccatiEqns} to get $\Lambda(t)=\Gamma(t)=\frac{e^{2t}}{2e^{2T}-e^{2t}},\gamma=0,\chi(t)=\log(\sqrt{2-e^{2t-2T}})$, so the control \eqref{eq:optimalcontrolLQregime} is given by:
\begin{align*}
v^N_i(t,x_1,...,x_N)=\frac{2e^{2t}}{e^{2t}-2e^{2T}}x_i.
\end{align*}

Then the relative error $\rho(\hat{\delta}_N)$ is given by
\begin{align*}
&\frac{1}{\sqrt{M}}\sqrt{\frac{\E\biggl[\exp\left(-2\sum_{i=1}^N (\hat{X}^{i,N,y}_T)^2\right)\prod_{i=1}^N \exp\biggl(-4\int_0^T \frac{e^{2t}}{e^{2t}-2e^{2T}}\hat{X}^{i,N,y}_t d\hat{W}^i_t - 4\int_0^T \frac{e^{4t}}{(e^{2t}-2e^{2T})^2}(\hat{X}^{i,N,y}_t)^2 dt\biggr) \biggr]}{\E\biggl[\exp(-NG(\mu^{N,y}_T)) \biggr]^2}-1}\\
& =\frac{1}{\sqrt{M}}\sqrt{\frac{\E\biggl[\exp\biggl(-2\sum_{i=1}^N (\hat{X}^{i,N,y}_T)^2+2\int_0^T \frac{e^{2t}}{e^{2t}-2e^{2T}}\hat{X}^{i,N,y}_t d\hat{W}^i_t + 2\int_0^T \frac{e^{4t}}{(e^{2t}-2e^{2T})^2}(\hat{X}^{i,N,y}_t)^2 dt\biggr) \biggr]}{\E\biggl[\exp(-NG(\mu^N)) \biggr]^2}-1}
\end{align*}
where
\begin{align*}
\hat{X}^{i,N,y}_t=y+\int_0^t\left[\frac{2e^{2s}}{e^{2s}-2e^{2T}}-1\right] \hat{X}^{i,N,y}_s  ds +  \hat{W}^i_t.
\end{align*}
By It\^o's formula, we have 
\begin{align*}
&(\hat{X}^{i,N,y}_t)^2\frac{e^{2t}}{e^{2t}-2e^{2T}}\\ 
&=y^2\frac{1}{1-2e^{2T}}+2\int_0^t \frac{e^{4s}}{(e^{2s}-2e^{2T})^2}(\hat{X}^{i,N,y}_s)^2ds + 2\int_0^t \hat{X}^{i,N,y}_s\frac{e^{2s}}{e^{2s}-2e^{2T}}d\hat{W}^i_s+\frac{1}{2}\log\biggl(\frac{2e^{2T}-e^{2t}}{2e^{2T}-1}\biggr).
\end{align*}
Setting $t=T$, we get 
\begin{align*}
(\hat{X}^{i,N,y}_{T})^2+2\int_0^{T} \frac{e^{4s}}{(e^{2s}-2e^{2T})^2}(\hat{X}^{i,N,y}_s)^2ds + 2\int_0^{T} \hat{X}^{i,N,y}_s\frac{e^{2s}}{e^{2s}-2e^{2T}}d\hat{W}^i_s=-y^2\frac{1}{1-2e^{2T}}-\frac{1}{2}\log\biggl(\frac{e^{2T}}{2e^{2T}-1}\biggr),
\end{align*}
so, continuing:
\begin{align*}
\rho(\hat{\delta}_N) & = \frac{1}{\sqrt{M}}\sqrt{\frac{\exp\biggl(\frac{2N}{1-2e^{2T}}y^2\biggr)\frac{e^{2NT}}{(2e^{2T}-1)^{N}}}{\exp\left(2Ny^2/(1-2e^{2T})\right)}(2-e^{-2T})^N-1} = 0.
\end{align*}

This shows that in this situation as well, while the standard Monte Carlo method yields a relative error which grows exponentially in $N$ for fixed $y\in\R$, $T>0$, and grows exponentially in $|y|$ for fixed $N\in\bb{N}$, $T>0$, our importance sampling scheme has zero error.

In this trivial case, the controlled particles are IID, so for both the importance sampling scheme and standard Monte Carlo we only need to simulate one particle (making the analysis in $N$ a bit useless in practice, but the analysis in $|y|$  still holds true when $N=1$).
\end{example}

\begin{remark}\label{remark:IIDandlinearG}
This last example can be viewed through the following lens. Suppose that $b(x,\mu)=b(x),\sigma(x,\mu)=\sigma(x)$ (so that the particles are IID when given the same initial condition) and that $G(\mu)=\langle g,\mu\rangle$ for some $g:\R^d\tto \R$. Making the ansatz $\Psi(t,\mu)=\langle \psi(t,\cdot),\mu\rangle$ in \eqref{eq:HJBequation}, we get:
\begin{align*}
\langle -\dot\psi(t,\cdot)-b(\cdot)\cdot\partial_x\psi(t,\cdot)+\frac{1}{2}|\sigma^\top(\cdot)\partial_x\psi(t,\cdot)|^2-\frac{1}{2}\sigma\sigma^\top(\cdot):\partial^2_x\psi(t,\cdot),\mu\rangle&=0,\qquad t\in[0,T),\mu\in\mc{P}_2(\R^d),\\ 
\langle \psi(T,\cdot),\mu\rangle &= \langle g(\cdot),\mu\rangle,\qquad \mu \in \mc{P}_2(\R^d),
\end{align*}
which of course is satisfied if we have a unique solution $\psi$ to 
\begin{align}\label{eq:zerovariancesmallnoiseHJB}
 -\dot\psi(t,x)-b(x)\cdot\partial_x\psi(t,x)+\frac{1}{2}|\sigma^\top(x)\partial_x\psi(t,x)|^2-\frac{1}{2}\sigma\sigma^\top(x):\partial^2_x\psi(t,x)&=0,\qquad t\in[0,T),x\in\R^d,\\ 
\psi(T,x)&= g(x),\qquad x\in\R^d.\nonumber
\end{align}
For a related discussion, see Remark 3.3 in \cite{PW}. Equation \eqref{eq:zerovariancesmallnoiseHJB} is the HJB equation corresponding to the zero-variance estimator resulting from the Doob $h$-transform in the single particle setting --- see Equation (2.7) in \cite{VW}. It can thus be seen, as in, e.g., \cite{VW} pp. 1778-1779, that letting $v(t,x)=-\sigma^\top(x) \partial_x\psi(t,x)$,
\begin{align*}
\E\left[\exp\left(-2g(\hat{X}^{s,y}_T)\right)\exp\left(-2\int_s^T v(t,\hat{X}^{s,y}_t)\cdot d\hat{W}_t-\int_0^T|v(t,\hat{X}^{s,y}_t)|^2dt\right)\right]=\E[\exp(-g(X^{s,y}_T))]^{2}
\end{align*}
where 
\begin{align*}
dX^{s,y}_t=b(X^{s,y}_t)dt+\sigma(X^{s,y}_t) dW_t,X^{s,y}_s=y
\end{align*}
and 
\begin{align*}
d\hat{X}^{s,y}_t=[b(\hat{X}^{s,y}_t)+\sigma v(t,\hat{X}^{s,y}_t)]dt+\sigma(\hat{X}^{s,y}_t) d\hat{W}_t,\hat{X}^{s,y}_s=y.
\end{align*}
Here $y\in\R^d$ and $W,\hat{W}$ are standard $m$-dimensional Brownian motions initialized at $W_s=\hat{W}_s=0$.
Then we have, for $\mu^{N,s,y}_T,\hat{\mu}^{N,s,y}_T$ the empirical measure on IID copies of solutions $X^{s,y}_T$,$\hat{X}^{s,y}_T$ respectively (which we denote by $X^{i,s,y}_T$,$\hat{X}^{i,s,y}_T,i\in\bb{N}$ and by $W^i,\hat{W}^i$ their driving Brownian motions) and using the controls from \eqref{eq:optimalcontrol} in Theorem \ref{theo:logefficient} are given by $v^N_i(t,x_1,x_2,...,x_N)=-\sigma^\top(x_i)\partial_\mu \Psi(t,\mu^N_x)[x_i]=-\sigma^\top(x_i)\partial_x \psi(t,x_i)=v(t,x_i)$, the importance sampling relative error from \eqref{eq:Rdeltahat} satisfies:
\begin{align*}
 &\frac{\E\biggl[\exp\left(-2NG(\hat{\mu}^{N,s,y}_T)\right)\prod_{i=1}^N (Z^{i,N,s,y})^2 \biggr]}{\E\biggl[\exp\left(-NG(\mu^{N,s,y}_T)\right) \biggr]^2}-1\\ 
 &=\frac{\E\left[\exp\left(\sum_{i=1}^N-2g(\hat{X}^{i,s,y}_T)-2\int_s^Tv(t,\hat{X}^{i,s,y}_t)d\hat{W}^i_t-\int_s^T|v(t,\hat{X}^{i,s,y}_t)|^2dt \right)\right]}{\E\biggl[\exp(-\sum_{i=1}^Ng(X^{i,s,y}_T)) \biggr]^2}-1\\ 
 & = \left(\E\left[\exp\left(-2g(\hat{X}^{s,y}_T))\exp(-2\int_s^T v(t,\hat{X}^{s,y}_t)\cdot d\hat{W}_t-\int_s^T|v(t,\hat{X}^{s,y}_t)|^2dt\right)\right]\cdot\E[\exp\left(-g(X^{s,y}_T)\right)]^{-2}\right)^N-1\\ 
 &=1-1\\ 
 &=0.
\end{align*}
Thus our importance sampling scheme yields an estimator with zero relative error in the setting of non-interacting diffusions and linear $G$.

This also shows that in the non-interacting regime, even when $G$ is linear, solving \eqref{eq:HJBequation} is as difficult as solving for the exact (non-zero viscosity) solution to the HJB equation associated with the importance sampling scheme for one particle \eqref{eq:zerovariancesmallnoiseHJB}. We take this remark as motivation to construct importance sampling schemes for the empirical measure of weakly interacting diffusions based on the joint small-noise and large $N$ limit large deviation principles derived in \cite{Orrieri} and \cite{BCcurrents}. We expect that using such a scheme, in the non-interacting and linear $G$ regime, the small noise importance sampling scheme of, e.g., \cite{VW} will be recovered, and solving the resulting first-order HJB equation on Wasserstein space will be equivalent to solving the standard zero-viscosity HJB equation found as Equation (2.10) therein. 
\end{remark}

\begin{remark}
Note, despite the framework of our linear-quadratic example and the discussion in Remarks \ref{remark:comparisonwiththesmallnoisecase} and \ref{remark:IIDandlinearG}, that in Theorem \ref{theo:logefficient}, we have the freedom to choose $G(\mu)$ to be arbitrarily non-linear so long as it is sufficiently smooth to have a classical subsolution in the sense of Definition \eqref{defi:subsolution}. That is, it need not take the form $G(\mu)=\langle g,\mu\rangle$ or even $G(\mu)=g_1(\langle g_2,\mu\rangle)$. Also, as evidenced in the examples of Subsection \ref{subsection:nonsmooth}, we expect that even if $G$ is not smooth, so that only a (weak) subsolution to \eqref{eq:HJBequation} can be recovered, our importance sampling scheme can be expected to still yield sub-exponentially growing (log-efficient --see Definition \ref{definition:logefficiency}), or even vanishing relative error (Definition \ref{def:vanishingrelativeerror}) --- see Tables \ref{table:absinside} and \ref{table:absoutside}, respectively.

Moreover, our importance sampling scheme treats the empirical measure of IID diffusions and of weakly interacting particles uniformly. The former case can in some sense be seen as a ``Sanov's Theorem'' type generalization of the ``Cram\'er's Theorem'' type results found in \cite{DW1}, since using the ideas outlined in that paper would essentially correspond to taking independent particles and $G(\mu)=g(\int_{\R^d}z\mu(dz))$.

\end{remark}

We end this section by observing that the above examples are not anomalous within the linear-quadratic setting. Namely, the exact desired expectation \eqref{eq:desiredexpectation}, and hence the relative error of the standard Monte Carlo estimator, can always easily be found by solving a system of ODEs. Moreover, from this we can make the observation that the relative error of the importance sampling estimator with control \eqref{eq:optimalcontrolLQregime} will in fact always have zero relative error.

Note that in the linear-quadratic regime, the right hand side of prelimit representation from \eqref{eq:prelimitrep} in Proposition \ref{prop:prelimitLDPexpression} with $\tilde{b}=b$ and $F(\mu)=G(\mu_T)$ takes the form of a standard linear-quadratic stochastic control problem. We see then that 
\begin{align*}
-\frac{1}{N}\log \E[\exp\left(-NG(\mu^{N,s,y}_T)\right)]=\bar{y}^\top \Lambda_N(s)\bar{y}+\gamma_N(s)\cdot \bar{y} +\chi_N(s),
\end{align*}
so the desired expectation from \eqref{eq:desiredexpectation} is given by 
\begin{align*}
\E\left[\exp\left(-NG(\mu^{N,s,y}_T)\right)\right]=\exp\left(-N[\bar{y}^\top \Lambda_N(s)\bar{y}+\gamma_N(s)\cdot \bar{y} +\chi_N(s)]\right).
\end{align*}
In the above we denote by  $(y_1,...,y_N)=\bar{y}\in \R^{dN}$, and $\Lambda_N\in\R^{dN\times dN},\gamma_N\in \R^{dN},\chi_N\in\R$ satisfy the Riccati equations: 
\begin{align*}
\dot\Lambda_N(t)+\Lambda_N(t)B_N+B_N^\top \Lambda_N(t)-2N\Lambda_N(t)\sigma_N\sigma_N^\top\Lambda_N^\top(t)&=0,\Lambda_N(T)=P_{N2},\\ 
\dot\gamma_N(t)+B_N^\top\gamma(t)-2N\Lambda_N(t)\sigma_N\sigma^\top_N\gamma_N(t)+2\Lambda_N(t)b_{0N}&=0,\gamma_N(T)=p_{N1},\\ 
\dot\chi_N(t)-\frac{N}{2}\gamma_N^\top(t)\sigma_N\sigma_N^\top\gamma_N(t)+\gamma_N(t)\cdot b_{0N}+\sum_{i=1}^N\bar{\sigma}_{i,N}^\top\Lambda_N(t)\bar{\sigma}_{i,N}&=0,\chi_N(T)=p_2.
\end{align*}
Here 
\begin{align*}
b_{0N}&=[b^\top_0,...,b^\top_0]^\top\in \R^{dN},\\ 
\sigma_N&=\text{diag}(\sigma)\in\R^{dN\times N},\\ 
p_{1N}&=\frac{1}{N}[p_1^\top,...,p_1^\top]^\top\in \R^{dN},\\ 
B_N&=\text{diag}(B)+\frac{1}{N}\bar{B}_N\in\R^{dN\times dN},\\ 
P_{2N}&=\frac{1}{N}\text{diag}(P_2)+\frac{1}{N^2}\bar{P}_{2N}\in \R^{dN\times dN}
\end{align*}
and $\bar{B}_N$ is the block matrix with every $d\times d$ entry given by $\bar{B}$, $\bar{P}_{2N}$ is the block matrix with every $d\times d$ entry given by $\bar{P}_2$, $\bar{\sigma}_{i,N}\in\R^{dN}$ has i'th entry $\sigma$ and the rest $0$.

It may seem initially that we need to solve $dN\times dN$ Riccati equations in order to get this value, but in fact we can find that, letting $\Lambda_N^{i,j}\in\R^{d\times d}$ denote the $d\times d$ block matrix in the $i,j$ position, that $\Lambda_N^{i,j}=\bar{\Lambda}_N$ for all $i\neq j$ and $\Lambda_N^{i,i}=\hat{\Lambda}_N$ for all $i\in \br{1,...,N}$, and hence $\gamma_N^i=\bar{\gamma}_N\in\R^d$ for all $i$. Thus we need only solve:
\begin{align*}
\dot{\bar{\Lambda}}_N&+\bar{\Lambda}_N[\bar{B}+B]+\frac{1}{N}[\hat{\Lambda}_N-\bar{\Lambda}_N]\bar{B}+[B+\bar{B}]^\top \bar{\Lambda}_N+\frac{1}{N}\bar{B}^\top[\hat{\Lambda}_N-\bar{\Lambda}_N]-2N[(N-2)\bar{\Lambda}_N\sigma\sigma^\top\bar{\Lambda}_N^\top,\\ 
&\hspace{8cm}+\hat{\Lambda}_N\sigma\sigma^\top\bar{\Lambda}_N^\top+\bar{\Lambda}_N\sigma\sigma^\top\hat{\Lambda}_N^\top]=0,\\ 
\bar{\Lambda}_N(T)&=\frac{1}{N^2}\bar{P}_2,\\ 
\dot{\hat{\Lambda}}_N&+(1-\frac{1}{N})\bar{\Lambda}_N\bar{B}+(1-\frac{1}{N})\bar{B}^\top\bar{\Lambda}_N+\hat{\Lambda}_N[\bar{B}/N+B]+[\bar{B}/N+B]^\top\hat{\Lambda}_N-2N[(N-1)\bar{\Lambda}_N\sigma\sigma^\top\bar{\Lambda}_N^\top,\\ 
&\hspace{8cm}+\hat{\Lambda}_N\sigma\sigma^\top\hat{\Lambda}_N^\top]=0,\\ 
\hat{\Lambda}_N(T)&=\frac{1}{N^2}\bar{P}_2+\frac{1}{N}P_2,\\ 
\dot{\bar{\gamma}}_N&+[\bar{B}+B]^\top\bar{\gamma}_N-2N[(N-1)\bar{\Lambda}_N+\hat{\Lambda}_N]\sigma\sigma^\top\bar{\gamma}_N+2[(N-1)\bar{\Lambda}_N+\hat{\Lambda}_N]b_0=0,\\ 
\bar{\gamma}_N(T)&=p_1/N,\\
\dot\chi_N(t)&-\frac{N^2}{2}\bar{\gamma}_N^\top \sigma\sigma^\top \bar{\gamma}_N+N\bar{\gamma}_N\cdot b_0+N\sigma^\top\hat{\Lambda}_N\sigma=0,\\ 
\chi_N(T)&=p_2.
\end{align*}
Then we have 
\begin{align*}
\E\left[\exp\left(-NG(\mu^{N,s,y}_T)\right)\right]&=\exp\left(-N\left[\bar{y}^\top \Lambda_N(s)\bar{y}+\gamma_N(s)\cdot \bar{y} +\chi_N(s)\right]\right)\\ 
&=\exp\left(-N\left[\sum_{i\neq j}^Ny_i^\top\bar{\Lambda}_N(s)y_j+\sum_{i=1}^Ny^\top_i\hat{\Lambda}_N(s)y_i+\bar{\gamma}^\top(s)\sum_{i=1}^Ny_i +\chi_N(s)\right]\right)\\ 
&=\exp\Biggl(-N\Biggl[\left(\sum_{i=1}^Ny_i\right)^\top\bar{\Lambda}_N(s)\left(\sum_{i=1}^Ny_i\right)+\sum_{i=1}^Ny^\top_i\left[\hat{\Lambda}_N(s)-\bar{\Lambda}_N(s)\right]y_i\\ 
&\qquad+\bar{\gamma}^\top(s)\sum_{i=1}^Ny_i +\chi_N(s)\Biggr]\Biggr)
\end{align*}

Note that this is useful not only for establishing the true desired expectation \eqref{eq:desiredexpectation}, but also allows for us to calculate the relative error for the standard Monte Carlo estimator $\rho(\delta_N)$ by substituting 2$G$ for $G$ in \eqref{eq:LQsetup} with $P_2,p_1,$ and $\bar{P}_2$ modified to $2P_2,2p_1,$ and $2\bar{P}_2$, respectively. We will use this fact in the numerical example in Subsection \ref{subsection:LQNumerics}. 

Also note that we can see directly how to write the right-hand side of the previous display as a function of $\mu^N_{y}$. Letting
\begin{equation}\label{eq:prelimitHJBsolLQ}
    \begin{split}
\Phi^N(t,\nu)&=N^2\left[\int_{\R^d}z\mu(dz)\right]^\top\bar{\Lambda}_N(t)\left[\int_{\R^d}z\mu(dz)\right]+N\int_{\R^d}z^\top[\hat{\Lambda}_N(t)-\bar{\Lambda}_N(t)]z\mu(dz)\\ &\hspace{8cm}+N\bar{\gamma}_N^\top \int_{\R^d}z\mu(dz)+\tilde{\chi}_N(t)\\ 
&=\int_{\R^d}z^\top\tilde{\Lambda}_N(t)z\mu(dz)+\left(\int_{\R^d}z\mu(dz)\right)^\top\left[\tilde{\Gamma}_N(t)-\tilde{\Lambda}_N(t)\right]\left(\int_{\R^d}z\mu(dz)\right)+\tilde{\gamma}_N^\top \int_{\R^d}z\mu(dz)+\tilde{\chi}_N(t),
    \end{split}
\end{equation}
where $\tilde{\Lambda}_N(t)=N\left[\hat{\Lambda}_N(t)-\bar{\Lambda}_N(t)\right],\tilde{\Gamma}_N=N(N-1)\bar{\Lambda}_N(t)+N\hat{\Lambda}(t),\tilde{\gamma}_N(t)=N\gamma(t)$, we have 
\begin{align*}
\exp\left(-N\Phi^N(s,\mu^N_x)\right)=\E\left[\exp\left(-NG(\mu^{N,s,x}_T)\right)\right],
\end{align*}
and get the Riccati equations:
\begin{align}\label{eq:prelimitRiccatiEquations}
\dot{\tilde{\Lambda}}_N&+\tilde{\Lambda}_NB+B^\top\tilde{\Lambda}_N-2\tilde{\Lambda}_N\sigma\sigma^\top\tilde{\Lambda}_N^\top=0,\quad\tilde{\Lambda}_N(T)=P_2,\\ 
\dot{\tilde{\Gamma}}_N&+\tilde{\Gamma}_N[B+\bar{B}]+[B+\bar{B}]^\top\tilde{\Gamma}_N-2\tilde{\Gamma}_N\sigma\sigma^\top\tilde{\Gamma}_N^\top=0,\quad\tilde{\Gamma}_N(T)=\bar{P}_2+P_2,\nonumber \\ 
\dot{\tilde{\gamma}}_N(t)&+[B+\bar{B}]^\top\tilde{\gamma}_N(t)-2\tilde{\Gamma}_N(t)\sigma\sigma^\top\tilde{\gamma}_N(t)+2\tilde{\Gamma}_N(t)b_0=0,\quad\tilde{\gamma}_N(T)=p_1,\nonumber\\ 
\dot{\tilde{\chi}}_N(t)&-\frac{1}{2}\tilde{\gamma}_N^\top \sigma\sigma^\top \tilde{\gamma}_N+\tilde{\gamma}_N\cdot b_0+\sigma^\top\tilde{\Lambda}_N\sigma+\frac{1}{N}\sigma^\top[\tilde{\Gamma}_N-\tilde{\Lambda}_N]\sigma=0,\quad \tilde{\chi}_N(T)=p_2.\nonumber 
\end{align}
Note firstly that it is easily verified that $\Phi^N(t,\nu)$ from \eqref{eq:prelimitHJBsolLQ} satisfies the prelimit HJB equation \eqref{eq:prelimitHJB} for all $\nu$. Also observe that, other than the perturbation of the ODE for $\tilde{\chi}_N$ by the addition of the term $\frac{1}{N}\sigma^\top[\tilde{\Gamma}_N-\tilde{\Lambda}_N]\sigma$, these are the same Riccati equations as \eqref{eq:RiccatiEqns} for the ``zero viscosity'' HJB equation \eqref{eq:HJBequation}. The zero variance control from \eqref{eq:zerovariancecontrolprelim} is thus:
\begin{align*}
v^N_i(t,x_1,...,x_N)&=-\sigma^\top\biggl[2\tilde\Lambda(t)x_i+2 [\tilde\Gamma(t)-\tilde\Lambda(t)]\int_{\R^d}z\mu^N_x(dz)+\tilde\gamma(t)\biggr]\\ 
&=-\sigma^\top\biggl[2\Lambda(t)x_i+2 [\Gamma(t)-\Lambda(t)]\int_{\R^d}z\mu^N_x(dz)+\gamma(t)\biggr],
\end{align*}
that is, it is the same as the control from Theorem \ref{theo:logefficient} in the linear-quadratic regime (see  \eqref{eq:optimalcontrolLQregime}).

\begin{remark}\label{remark:zerorelativeerrorLQ}
The above discussion implies that in the linear-quadratic regime, the relative error of our importance sampling scheme is always zero. 

This can also be derived via the following observation: In \eqref{eq:prelimitHJB}, when we use the ansatz that the solution takes the form of \eqref{eq:prelimitHJBsolLQ}, since $\partial^2_\mu\Phi^N(\mu)[z_1,z_2]=2[\tilde{\Gamma}_N(t)-\tilde{\Lambda}_N(t)]$, the only term appearing from the second-order term $\int_{\R^d}\frac{1}{2N}\sigma\sigma^\top :\partial^2_\mu\Phi^N(t,\mu)[z,z]\mu(dz)$ is $\frac{1}{N}\sigma\sigma^\top:[\tilde{\Gamma}_N(t)-\tilde{\Lambda}_N(t)]=\frac{1}{N}\sigma^\top[\tilde{\Gamma}_N(t)-\tilde{\Lambda}_N(t)]\sigma$. Thus, when deriving the form of the Riccati equations \eqref{eq:prelimitRiccatiEquations}, it will only show up for $\tilde{\chi}_N$, which collects the constant in $\mu$ terms.

More generally, if we have a solution to \eqref{eq:HJBequation} such that $\partial_\mu\Psi=\partial_\mu \Phi^N$ for all $N$, where $\Phi^N$ solves \eqref{eq:prelimitHJB}, then the importance sampling scheme has zero relative error for all $N$, as per the discussion in Section \ref{sec:ontheHJBEquation}. Supposing, as in the linear-quadratic case, that we have a solution $\Psi$ to \eqref{eq:HJBequation} such that $\partial^2_\mu\Psi(t,\mu)[z_1,z_2]=C(t)\in\R^{d\times d}$ and that $\sigma$ does not depend on $x,\mu$, then $\Psi(T,\nu)=G(\nu)$ and:
\begin{align*}
-\partial_t\Psi(t,\nu)&-\int_{\R^d}b(z,\nu)\cdot \partial_\mu \Psi(t,\nu)[z]-\frac{1}{2}|\sigma^\top\partial_\mu\Psi(t,\nu)[z]|^2\\ 
&+\frac{1}{2}\sigma\sigma^\top(z,\nu):\partial_z\partial_\mu\Psi(t,\nu)[z]+\frac{1}{2N}\sigma\sigma^\top:\partial^2_\mu\Psi(t,\nu)[z,z]\nu(dz)=-\frac{1}{2N}\sigma\sigma^\top:C(t),
\end{align*} 
so letting $\lambda\in C([0,T];\R)$ solve $\dot{\lambda}(t)=-\frac{1}{2N}\sigma\sigma^\top:C(t),\lambda(T)=0$, $\Phi^N(t,\mu)=\Psi(t,\mu)+\lambda(t)$ is a solution to \eqref{eq:prelimitHJB}, and $\partial_\mu\Phi^N=\partial_\mu \Psi$ for all $N$.

It thus clear that letting $m > 1$ or having the coefficients $b_0,B,\bar{B},\sigma$ in \eqref{eq:LQsetup} depend on time does not change the fact that our importance sampling scheme yields a zero-variance estimator in the linear-quadratic regime. Also note that this discussion also applies in the linear-quadratic setting for small noise importance sampling schemes --- compare Equations (2.7) and (2.10) in \cite{VW} and recall that if $b$ is linear in $x$, $\sigma$ is constant, and $G$ is quadratic in $x$ therein that solutions to equation (2.10) will have second derivatives which are constant in $x$.

This reflects a difficulty present in both our setting and the small noise setting for importance sampling schemes based on large deviation principles and zero-viscosity HJB equations --- the hurdle of actually computing the derivative of a solution to the HJB equation and hence the control to be used, numerically or otherwise, persists, and in the main class of problems for which a solution can be computed analytically (the linear-quadratic regime), the solution to the zero-viscosity HJB equation is no easier to solve for than the desired expectation. This means that efficient methods for numerically computing $\partial_\mu\Psi$ appearing in the definition of the control \eqref{eq:optimalcontrol} in Theorem \ref{theo:logefficient}, or otherwise numerically obtaining the control, are highly desirable. Designing and implementing such schemes is an interesting avenue for future research.

\end{remark}

\section{Numerical Results}\label{sec:numerics}
Here we present numerical results based on examples in the linear-quadratic regime of \eqref{sec:LQ} and perturbations thereof. In all three examples, the importance sampling scheme greatly outperforms standard Monte Carlo methods. In the standard linear-quadratic regime, we observe that the importance sampling relative error is near $0$ for all $N$, as expected, in Table \ref{table:LQ}. When a ``global'' lack of differentiability is introduced to $G$, we observe in Table \ref{table:absoutside} that the importance sampling relative error appears to be vanishing. When a ``local'' lack of differentiability  is introduced to $G$, we observe in Table \ref{table:absinside} that the importance sampling relative error does not seem to vanish, but grows sublinearly, still a massive improvement over the standard Monte Carlo relative error which increases exponentially with $N$ in all three examples.

We simulate all the SDE systems using the Euler-Maruyama method with time step $\Delta t=.01/N$, which we chose empirically by iteratively refining until the estimates $\delta_N$ and $\hat{\delta}_N$ stabilized (i.e., consecutive estimates were nearly equal). We use $M\approx 10^7$ samples for all our simulations. In each case we also use the same final time $T=1$ and diffusion constant $\sigma=.5$.

We note that our tabled results stop at relatively modest values of $N$; although the importance sampling relative error remained small for much larger values of $N$ than those presented in the tables, the exponential growth in the standard Monte Carlo relative errors required too many simulations to resolve for large $N$ for us to make accurate comparisons.

A Python code implementing the examples in this section is provided on the public Gitlab repository \href{https://gitlab.com/mheldman/IS-interacting-particles}{https://gitlab.com/mheldman/IS-interacting-particles}. 

\subsection{The Linear-Quadratic Regime}\label{subsection:LQNumerics}

For our first example, we consider \begin{align}\label{eq:LQnumericsG}
G(\mu) = \int_{\mathbb{R}} z^2 \mu(dz).\end{align}
This corresponds to the linear-quadratic regime in Section \ref{sec:LQ} with $d=m=1,s=0,B=-1,\bar{B}=2,P_2=1,b_0=\bar{P}_2=p_1=p_2=0$ and $y_i=y$ for all $i\in \bb{N}$.

We seek to estimate
\begin{align*}
\E\left[\exp\left(-NG(\mu^{N,y}_T)\right)\right]=\E\biggl[\exp\biggl(-\sum_{i=1}^N |X^{i,N,y}_T|^2\biggr)\biggr]
\end{align*}
where 
\begin{align}\label{eq:IPSforNumerics}
dX^{i,N,y}_t=\left[-X^{i,N,y}_t+\frac{2}{N}\sum_{j=1}^NX^{j,N,y}_t\right]dt+\sigma dW^i_t,X^{i,N,y}_0=y.
\end{align}

Note that, as discussed in Remark \ref{remark:comparisonwiththesmallnoisecase}, the importance sampling scheme for this case cannot be derived from previous work on small-noise SDEs since $P_2\neq 0$.

The typical dynamics from \eqref{eq:McKeanVlasovEquation} will have $\E[X^{\delta_y}_t]=ye^t$ for all $t\in [0,T]$ (we suppress the initial condition $s$ in the notation here) and thus will be given by the time-inhomogeneous Ornstein–Uhlenbeck-like process with drift coefficient $b(t,x)=-x+2ye^t$ and diffusion coefficient $\sigma$. The result of evaluating $G$ from \eqref{eq:LQnumericsG} at $\mu^{N,y}$:
\begin{align}\label{eq:sumforfirstnumericalexample}
G(\mu^{N,y}_T)=\frac{1}{N}\sum_{i=1}^N |X^{i,N,y}_T|^2
\end{align}
is thus expected to be close to the second moment of such a process at time $T$ as $N$ becomes large, which is 
\begin{align}\label{eq:McKeanVlasovSecondMoment}
\E\left[|X^{\delta_y}_T|^2\right]=y^2e^{2T}+\frac{\sigma^2}{2}(1-\exp(-2T)). 
\end{align}
The atypical event we are seeking to sample, represented by $G$ being close to zero, is that the particles interact in such a way that they are closer to the origin than expected via \eqref{eq:McKeanVlasovSecondMoment}.


The role of the importance sampling scheme from Theorem \ref{theo:logefficient} will thus be to modify the interaction and dynamics of each particle to predictably remain closer to the origin. Indeed, as discussed in Remark \ref{remark:zerorelativeerrorLQ}, the theoretical relative error of the importance sampling scheme is zero, so any contributing relative error will be due to numerical discretization. Indeed, we observed in numerical experiments that taking $\Delta t\downarrow 0$ also decreases the relative error of the importance scheme to $0$. 

Solving the Riccati equations \eqref{eq:RiccatiEqns}, we get 
\begin{align*}
\Lambda(t)&=\frac{e^{2t}}{e^{2T}(1+\sigma^2)-e^{2t}\sigma^2},\\ 
\Gamma(t)&=\frac{e^{2T}}{e^{2T}\sigma^2-e^{2t}(\sigma^2-1)},\\ 
\gamma(t)&=0,\\ 
\chi(t)&=\frac{1}{2}\log\biggl(\frac{e^{2T}}{e^{2T}(1+\sigma^2)-e^{2t}\sigma^2}\biggr)
\end{align*}
Thus the optimal control $v$ from \eqref{eq:optimalcontrolLQregime} and Theorem \ref{theo:logefficient} is given by:
\begin{align*}
v^N_i(t,x_1,...,x_N)=-2\sigma\biggl[\frac{e^{2t}}{e^{2T}(1+\sigma^2)-e^{2t}\sigma^2}x_i+\frac{1}{N}\biggl(\frac{e^{2T}}{e^{2T}\sigma^2-e^{2t}(\sigma^2-1)}-\frac{e^{2t}}{e^{2T}(1+\sigma^2)-e^{2t}\sigma^2}\biggr)\sum_{i=1}^N x_i \biggr].
\end{align*}

In Table \ref{table:LQ} below, we provide for various $N$ the estimated expectation from standard Monte Carlo $\delta_N$ and importance sampling $\hat{\delta}_N$ \eqref{eq:deltahat}, the empirical relative error from standard Monte Carlo $\tilde{\rho}(\delta_N)$ and for importance sampling $\tilde{\rho}(\hat{\delta}_N)$, and the exact value of the expectation \eqref{eq:desiredexpectation} from \eqref{eq:prelimitHJBsolLQ} and \eqref{eq:prelimitRiccatiEquations}. Here $\tilde{\rho}(\hat{\delta}_N)$ is as in \eqref{eq:importancesamplingrelativeerror} but where $M=1$ and the expectation and second moment are computed empirically. That is: 
\begin{align}\label{eq:IStilderho}
\tilde{\rho}(\hat{\delta}_N)\coloneqq \sqrt{\frac{\frac{1}{M}\sum_{j=1}^M \exp(-2NG(\hat{\mu}^{N,s,y,j}_T))\biggl(\prod_{i=1}^N Z^{i,N,s,y,j}\biggr)^2}{\hat{\delta}_N^2}-1}.
\end{align}
and similarly:
\begin{align}\label{eq:MCtilderho}
\tilde{\rho}(\delta_N)\coloneqq \sqrt{\frac{\frac{1}{M}\sum_{j=1}^M \exp(-2NG(\mu^{N,s,y,j}_T))}{\delta_N^2}-1},
\end{align}
\begin{table}[h!]
\begin{center}
\begin{tabular}{||c| c c| c c| c||} 
\hline 
&\multicolumn{2}{c|}{IS Scheme}&\multicolumn{2}{c|}{Standard Monte Carlo}&\\ 
 \hline\hline
 $N$ & $\hat{\delta}_N$ & $\tilde{\rho}(\hat{\delta}_N)$ & $\delta_N$& $\tilde{\rho}(\delta_N)$ & Exact Value \\ [0.5ex] 
 \hline\hline
 5 & $2.3816\cdot 10^{-1}$& $3.4601\cdot 10^{-2}$& $2.3807\cdot 10^{-1}$& $1.0380$ &$2.3747\cdot 10^{-1}$\\ 
 \hline
 10 & $8.2550\cdot 10^{-2}$& $2.7101\cdot 10^{-2}$& $8.2551\cdot 10^{-2}$& $1.5721$ &$8.2412\cdot 10^{-2}$\\ 
 \hline
 15 & $2.8641\cdot 10^{-2}$& $2.4001\cdot 10^{-2}$& $2.8661\cdot 10^{-2}$& $2.1957$ &$2.8600\cdot 10^{-2}$\\ 
 \hline
 20 & $9.9373\cdot 10^{-3}$& $2.2369\cdot 10^{-2}$& $9.9165\cdot 10^{-3}$& $2.9589$ &$9.9254\cdot 10^{-3}$\\
 \hline
 25 & $3.4486\cdot 10^{-3}$& $2.1310\cdot 10^{-2}$& $3.4824\cdot 10^{-3}$& $3.9086$ &$3.4445\cdot 10^{-3}$\\ 
 \hline
 30 & $1.1968\cdot 10^{-3}$& $2.0550\cdot 10^{-2}$& $1.1951\cdot 10^{-1}$& $5.1415$ &$1.1954\cdot 10^{-3}$\\ 
 \hline
 50 & $1.7361\cdot 10^{-5}$& $1.8934\cdot 10^{-2}$& $1.7569\cdot 10^{-5}$& $14.5202$ &$1.7339\cdot 10^{-5}$\\
 \hline
 80 & $3.0339\cdot 10^{-8}$& $1.8003\cdot 10^{-2}$& $3.2426\cdot 10^{-8}$& $74.1871$&$3.0289\cdot 10^{-8}$\\ [1ex] 
 \hline
\end{tabular}
\caption{Comparison of the standard Monte Carlo estimator $\delta_N$ with the importance sampling estimator $\hat{\delta}_N$  \eqref{eq:deltahat} and their corresponding relative errors $\tilde{\rho}(\delta_N)$ \eqref{eq:MCtilderho} and $\tilde{\rho}(\hat{\delta}_N)$ \eqref{eq:IStilderho} with $G$ as in \eqref{eq:LQnumericsG} and particles obeying \eqref{eq:IPSforNumerics} with initial condition $y=.2$.}
\label{table:LQ}
\end{center}
\end{table}

In Figure \ref{fig:TypicalSumAndRelError}(a) we plot an average of trajectories of the sum \eqref{eq:sumforfirstnumericalexample} for the uncontrolled and controlled particles, respectively, computed using the same noise. In Figure \ref{fig:TypicalSumAndRelError}(b), we plot the analytical relative error for the Monte Carlo scheme ($\rho(\delta_N)$ from \eqref{eq:importancesamplingrelativeerror} with $v_i^N\equiv 0$) on a log scale for various values of $N$, again as computed via \eqref{eq:prelimitHJBsolLQ} and \eqref{eq:prelimitRiccatiEquations}. One can see that the first few values of $N$ considered in Table \ref{table:LQ} are in a region where the growth for the standard Monte Carlo relative error has not yet shifted from linear to exponential, where importance sampling becomes even more valuable. 

\begin{figure}
    \centering
    \subfloat[\centering An average of five trajectories of the sum \eqref{eq:sumforfirstnumericalexample}]{{\includegraphics[width=7cm]{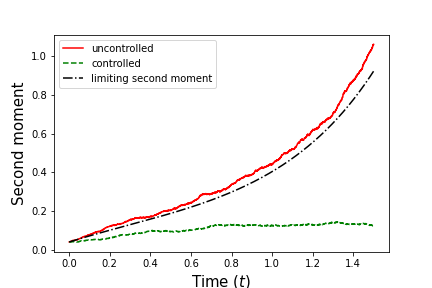} }}%
    \hspace{10pt}
    \subfloat[\centering Log-scale analytical single sample Monte Carlo relative error]{{\includegraphics[width=7cm]{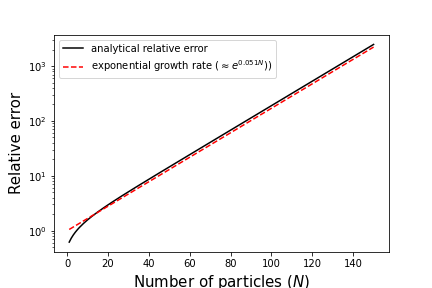} }}%
    \caption{Left: An average of 5 realizations of the sum \eqref{eq:sumforfirstnumericalexample} for the uncontrolled (red) and controlled (green) particles respectively against the variance of the limiting McKean-Vlasov equation from \eqref{eq:McKeanVlasovSecondMoment} (black) for $T\in[0,1.5]$, $y=.2$, and $N=50$. Right: the log-scaled exact relative error $\rho(\delta_N)$ as $N$ varies in the setting of Table \ref{table:LQ}, with $M=1$ fixed in \eqref{eq:importancesamplingrelativeerror} (black) against the exact asymptotic rate $\approx e^{.051 N}$ (red).}
    \label{fig:TypicalSumAndRelError}
\end{figure}

\subsection{Examples with Nonsmooth $G$}\label{subsection:nonsmooth}
Here we provide two examples where $G$ is not differentiable in the Lions sense (see Definition \ref{def:lionderivative}). The first example is constructed via a perturbation of the linear-quadratic regime which reduces the smoothness of $G$ in a ``global'' way, whereas the second reduces the smoothness of $G$ in a ``local'' way. In both examples, despite the lack of regularity we observe that the relative error of the importance sampling scheme is much smaller than that of the standard Monte Carlo method for all $N$. In Example \ref{example:absinside}, the relative error grows sublinearly, and in Example \ref{example:absoutside}, the importance sampling relative error even appears to be vanishing, exhibiting that the assumptions in Theorems \ref{theo:logefficient} and \ref{theo:expansionanalysis} are sufficient but not necessary. Thanks to the simple nature of the target function $G$, we are able to provide an interpretation of the role of the controls in each example --- see Remark \ref{remark:comparisonofabsinvsout}.
\begin{example}\label{example:absoutside}
In our second example, we consider \begin{align}\label{eq:absoutsideG}
G(\mu)=\biggl|\int_\R z \mu(dz)\biggr|,
\end{align} with drift and diffusion conforming to the linear-quadratic regime \eqref{eq:LQsetup} with $d=m=1$.

Note that while $G$ is continuous, it is not Lions differentiable at measures $\mu$ such that $\int_\R z\mu(dz)=0$, so we are not able to obtain an explicit solution to the HJB equation \eqref{eq:HJBequation} as in Section 3. However, for $\mu\in \mc{P}_2(\R)$ with nonzero mean, it is plain to see that $\partial_\mu G(\mu)[z]=\text{sign}\biggl(\int_{\R}x\mu(dx) \biggr)$ for all $z\in \R$, so $\partial_z\partial_\mu G(\mu)\equiv 0$ and $\partial^2_\mu G(\mu)\equiv 0$. Using this ansatz, taking $b_0=0$, and performing a similar computation to \cite{PW} Section 4, we obtain the formal (weak) viscosity solution to \eqref{eq:HJBequation} in this setting:
\begin{align*}
\Psi(t,\mu)=e^{(B+\bar{B})[T-t]}\biggl|\int_\R z\mu(dz) \biggr|+\frac{\sigma^2}{4(B+\bar{B})}[1-e^{2(B+\bar{B})(T-t)}]
\end{align*}
so that we expect the control from \eqref{eq:optimalcontrol} in Theorem \ref{theo:logefficient}  to be of the form
\begin{align*}
v(t,x_1,...,x_N)&=-\sigma \partial_\mu \Psi(t,\mu^N_x)[x_i]\\ 
& = -\sigma e^{(B+\bar{B})(T-t)}\text{sign}\biggl(\frac{1}{N}\sum_{j=1}^N x_j\biggr).
\end{align*}
Indeed, via the equivalence with the small-noise regime as discussed in Remark \ref{remark:comparisonwiththesmallnoisecase}, such a choice of control can be seen to solve the minimization problem \eqref{eq:LDP} with $F(\mu)=G(\mu_T)$.

In Table \ref{table:absoutside}, we provide for various $N$ the estimated expectation and empirical relative error for our importance sampling scheme compared with standard Monte Carlo, selecting  $B=-1$, $\bar{B}=2,$ $s=0$, and $y_i=.4$ for all $i\in\bb{N}$. As we see, the relative error for the importance sampling scheme seems to be vanishing as $N$ increases, whereas the standard Monte Carlo relative error increases exponentially.
\begin{table}[h!]
\begin{center}
\begin{tabular}{||c | c c| c c||} 
\hline 
&\multicolumn{2}{c|}{IS Scheme}&\multicolumn{2}{c||}{Standard Monte Carlo}\\ 
 \hline\hline
 $N$ & $\hat{\delta}_N$ & $\tilde{\rho}(\hat{\delta}_N)$ & $\delta_N$& $\tilde{\rho}(\delta_N)$  \\ [0.5ex] 
 \hline\hline
 5 & $2.6644\cdot 10^{-2}$&$3.7070\cdot 10^{-1}$& $2.6653\cdot 10^{-2}$& $2.9930$\\ 
 \hline
 10 & $9.0804\cdot 10^{-4}$&$3.1427\cdot 10^{-1}$& $9.2009\cdot 10^{-4}$& $12.8032$\\ 
 \hline
 15 & $3.0269\cdot 10^{-5}$& $2.6350\cdot 10^{-1}$& $3.0645\cdot 10^{-5}$& $53.9410$\\
 \hline
 20 & $9.9533\cdot 10^{-7}$& $2.2227\cdot 10^{-1}$& $1.0625\cdot 10^{-6}$& $337.0401$\\
 \hline
 25 & $3.2471\cdot 10^{-8}$& $1.8900\cdot 10^{-1}$& $3.5791\cdot 10^{-8}$& $561.9164$\\
 \hline
 30 & $1.0537\cdot 10^{-9}$& $1.6154\cdot 10^{-1}$& $1.2253\cdot 10^{-9}$& $1016.1656$\\
 \hline
\end{tabular}
\caption{Comparison of the standard Monte Carlo estimator $\delta_N$  with the importance sampling estimator $\hat{\delta}_N$  \eqref{eq:deltahat} and their corresponding relative errors $\tilde{\rho}(\delta_N)$ \eqref{eq:MCtilderho} and $\tilde{\rho}(\hat{\delta}_N)$ \eqref{eq:IStilderho} with $G$ as in  \eqref{eq:absoutsideG} and particles obeying \eqref{eq:IPSforNumerics}  with initial condition $y=.4$ (Example \ref{example:absoutside}).}
\label{table:absoutside}
\end{center}
\end{table}
\end{example}

\begin{example}\label{example:absinside}

\textcolor{white}{.}

\begin{table}[h!]
\begin{center}
\begin{tabular}{||c | c c |c c||} 
\hline 
&\multicolumn{2}{c|}{IS Scheme}&\multicolumn{2}{c||}{Standard Monte Carlo}\\ 
 \hline\hline
 $N$ & $\hat{\delta}_N$ & $\tilde{\rho}(\hat{\delta}_N)$ & $\delta_N$& $\tilde{\rho}(\delta_N)$  \\ [0.5ex] 
 \hline\hline
 5 & $2.1355\cdot 10^{-2}$&$4.7870\cdot 10^{-1}$& $2.1382\cdot 10^{-2}$& $2.3263$\\ 
 \hline
 10 & $5.6703\cdot 10^{-4}$& $6.0605\cdot 10^{-1}$& $5.6806\cdot 10^{-4}$& $6.6374$\\ 
 \hline
 15 & $1.5087\cdot 10^{-5}$& $7.2452\cdot 10^{-1}$& $1.5040\cdot 10^{-5}$& $17.3337$\\ 
 \hline
 20 & $4.0157\cdot 10^{-7}$& $8.3416\cdot 10^{-1}$& $4.1260\cdot 10^{-7}$& $51.1058$\\
 \hline
 25 & $1.0687\cdot 10^{-8}$& $9.4327\cdot 10^{-1}$& $1.0296\cdot 10^{-8}$& $97.8820$\\
 \hline
 30 & $2.8470\cdot 10^{-10}$& $1.0524$& $2.8800\cdot 10^{-10}$& $280.5945$\\
 \hline
\end{tabular}
\caption{Comparison of the standard Monte Carlo estimator $\delta_N$  with the importance sampling estimator $\hat{\delta}_N$  \eqref{eq:deltahat} and their corresponding relative errors $\tilde{\rho}(\delta_N)$ \eqref{eq:MCtilderho} and $\tilde{\rho}(\hat{\delta}_N)$ \eqref{eq:IStilderho} with $G$ as in \eqref{eq:absinsideG} and particles obeying \eqref{eq:IPSforNumerics} with initial condition $y=.4$ (Example \ref{example:absinside}).}
\label{table:absinside}
\end{center}
\end{table}

Finally, we consider the same regime as Example \ref{example:absoutside}, but now taking $G$ in the target expectation \eqref{eq:desiredexpectation} to be:
\begin{align}\label{eq:absinsideG}
G(\mu)=\int_{\R}|z|\mu(dz).
\end{align}
Once again $G$ is not everywhere Lions differentiable, in this case at measures such that the singleton $\br{0}\in \R$ has positive measure. Indeed, for $\mu\in\mc{P}_2(\R)$ such that $\mu(\br{0})=0$, it is clear to see that $\partial_\mu G(\mu)[z]=\text{sign}(z)$, so $\partial^2_\mu G(\mu)=0$ and for $z\neq 0$, $\partial_z\partial_\mu G(\mu)[z]=0$ (this follows via a similar computation to \cite{CD} Example 1 in Section 5.2.2).

Again we follow the methods of \cite{PW}, taking $b_0=0$, and obtain the formal (weak) viscosity solution to \eqref{eq:HJBequation}:
\begin{align*}
\Psi(t,\mu)&=e^{B(T-t)}\int_{\R}|z|\mu(dz)+[e^{(B+\bar{B})(T-t)}-e^{B(T-t)}]\int_{\R}z\mu(dz)\int_{\R}\text{sign}(z)\mu(dz)\\ 
&+\frac{\sigma^2}{4B(B+\bar{B})}[e^{2B(T-t)}(B+\bar{B})-Be^{2(B+\bar{B})(T-t)}-\bar{B}]\biggl[\int_{\R}\text{sign}(z)\mu(dz)\biggr]^2+\frac{\sigma^2}{4B}[1-e^{2B(T-t)}].
\end{align*}

Using that for $\mu\in\mc{P}_2(\R)$ with $\mu(\br{0})=0$, $\mu\mapsto \int_\R \text{sign}(z)\mu(dz)$ is Lions differentiable with derivative $0$, we expect the control from  \eqref{eq:optimalcontrol} in Theorem \ref{theo:logefficient}  to be of the form
\begin{align*}
v(t,x_1,...,x_N)&=-\sigma \partial_\mu \Psi(t,\mu^N_x)[x_i]\\ 
& = -\sigma \biggl[e^{B(T-t)}\text{sign}(x_i)+\biggl(e^{(B+\bar{B})(T-t)}-e^{B(T-t)}\biggr)\frac{1}{N}\sum_{j=1}^N \text{sign}(x_j)\biggr].
\end{align*}

In Table \ref{table:absinside}, we provide for various $N$ the estimated expectation and empirical relative error for our importance sampling scheme compared with standard Monte Carlo, selecting $B=-1$, $\bar{B}=2$, $s=0$, and $y_i=.4$ for all $i\in \N$. As we see, the importance sampling scheme has relative error which increases sublinearly as $N$ becomes large, in contrast to the standard Monte Carlo relative error which increases exponentially. 

\end{example}
\begin{remark}\label{remark:comparisonofabsinvsout}
We can think about the controls from Examples \ref{example:absoutside} and \ref{example:absinside} in terms of the role they play in sampling the ``rare event'' of $G$ being close to $0$ for the particle system in each case. When the absolute value appears outside the empirical mean, as in Example \ref{example:absoutside}, we only need to be concerned with controlling the magnitude of the empirical mean, a characteristic of the particle ensemble. This is reflected in the controls, which are identical for every particle and force them in the opposite direction of the sign of the empirical mean. Meanwhile, in the situation of Example \eqref{example:absinside}, to make $G$ close to zero we need the position of each individual particle to be near $0$. Thus, the control uses information about the position of each particle relative to the origin. However, information about the ensemble of particles also still must appear because of the global interaction in the particles' dynamics, which we see in the form of the difference in the proportions of the particles which are positive and negative at a given time. 

We note that a similar combination of global information on the ensemble of particles and local information about the $i$'th particle is also used for controls in the standard linear-quadratic regime of Section \ref{sec:LQ}. See in particular the control used in our simple numerical example in Subsection \ref{subsection:LQNumerics}. Numerical experiments confirm that the role of both the global and local information captured in the control is necessary for variance reduction in such cases. For example, using the simplified control $\tilde{v}(t,x_1,...,x_N)=-\sigma e^{B(T-t)}\text{sign}(x_i)$ in Example \ref{example:absinside}, we observed performance that was comparable or even worse than standard Monte Carlo.

It is also worth commenting on the choice of initial conditions in the examples of Subsection \ref{subsection:nonsmooth}. We choose $y_i=.4$ for all $i\in \bb{N}$, so that $\nu=\delta_{.4}$ is away from the region of discontinuity of $\partial_\mu \Psi(t,\cdot)[z]$. We observed that if we took initial conditions such that $\nu$ is closer to the region of discontinuity many more rogue trajectories appeared in the controlled dynamics and an extremely high sample size was needed before the importance sampling scheme could be seen to perform better than standard Monte Carlo. This is once again similar to the case of small-noise diffusions---see Table 3.1 compared to Table 3.2 in \cite{VW} and the discussion on pp. 1786-1787 therein.  
\end{remark}
We end this section by remarking on the interesting problem of verifying that the formal viscosity solutions to the HJB equation \eqref{eq:HJBequation} constructed in Examples \ref{example:absoutside} and \ref{example:absinside} are indeed viscosity solutions in the sense of \cite{CGKPR1}. To our knowledge, there are no existing examples in the current literature of explicit viscosity solutions which are not also classical solutions to an HJB equation on Wasserstein space.

\section{Proof of the Main Results}\label{sec:proofs}
\allowdisplaybreaks
We begin with a lemma which allows us to express the numerator of $R(\hat{\delta}_N)$ {(see \eqref{eq:Rdeltahat})} in an alternative way:
\begin{lem}\label{lemma:RhatnumeratorGirsanov}
{Let Assumption \ref{assumption:forLDP} \ref{assumption:uncontrolleduniqueness} and Assumption \ref{assumption:forlogefficiency} \ref{assumption:Gbounded} hold. Then,} for any $s\in [0,T],y_1,...,y_N\in \R^d$, and for $\hat{\mu}$ as in \eqref{eq:controlledempiricalmeasure}, $Z^{i,N,s,y}$ as in \eqref{eq:Zs}, and $v_i^N:[s,T]\times\R^{dN}\tto \R^m,i=1,...,N$ bounded:
\begin{align*}
\E\biggl[\exp\left(-2NG(\hat{\mu}^{N,s,y}_T)\right)\prod_{i=1}^N (Z^{i,N,s,y})^2 \biggr]=\E\biggl[\exp\left(-2NG(\tilde{\mu}^{N,s,y}_T)\right)\exp\biggl(\int_s^T \sum_{i=1}^N|v^N_i(t,\tilde{X}^{1,N,s,y}_t,...,\tilde{X}^{N,N}_t)|^2dt\biggr)\biggr],
\end{align*}
where 
$\tilde{\mu}^{N,s,y}_t = \frac{1}{N}\sum_{i=1}^N \delta_{\tilde{X}^{i,N,s,y}_t}$ and $\tilde{X}^{i,N,s,y}$ satisfy: 
\begin{align*}
d\tilde{X}^{i,N,s,y}_t &= [b(\tilde{X}^{i,N,s,y}_t,\tilde{\mu}^{N,s,y}_t) - \sigma(\tilde{X}^{i,N,s,y}_t,\tilde{\mu}^{N,s,y}_t) v^N_i(t,\tilde{X}^{1,N,s,y}_t,...,\tilde{X}^{N,N,s,y}_t)]dt + \sigma(\tilde{X}^{i,N,s,y}_t,\tilde{\mu}^{N,s,y}_t) d\tilde{W}^i_t,\\ 
\tilde{X}^{i,N,s,y}_s&=y_i,
\end{align*}
and $\tilde{W}^i$ are IID standard $m$-dimensional Brownian motions initialized at $\tilde{W}^i_s=0$.
\end{lem}
\begin{proof}
This follows by a simple application of the Girsanov theorem, which is valid due to the boundedness of the $v^N_i$'s { and $G$}. We have:
\begin{align*}
&\E\biggl[\exp\left(-2NG(\hat{\mu}^{N,s,y}_T)\right)\prod_{i=1}^N (Z^{i,N,s,y})^2 \biggr]\\
&= \E\biggl[\exp\left(-2NG(\hat{\mu}^{N,s,y}_T)\right)\exp\biggl(-2\sum_{i=1}^N\int_s^T v^N_i(t,\hat{X}^{1,N,s,y}_t,...,\hat{X}^{N,N,s,y}_t)\cdot d\hat{W}^i_t \\ 
&- \int_s^T \sum_{i=1}^N|v^N_i(t,\hat{X}^{1,N,s,y}_t,...,\hat{X}^{N,N,s,y}_t)|^2dt\biggr)\biggr]\nonumber\\ 
&= \E\biggl[\exp\left(-2NG(\hat{\mu}^{N,s,y}_T)\right)\exp\biggl(-2\sum_{i=1}^N\int_s^T v^N_i(t,\hat{X}^{1,N,s,y}_t,...,\hat{X}^{N,N,s,y}_t)\cdot d\hat{W}^i_t \\ 
&- \frac{1}{2}\int_s^T \sum_{i=1}^N|2v^N_i(t,\hat{X}^{1,N,s,y}_t,...,\hat{X}^{N,N,s,y}_t)|^2dt\biggr)\exp\biggl(\int_s^T \sum_{i=1}^N|v^N_i(t,\hat{X}^{1,N,s,y}_t,...,\hat{X}^{N,N,s,y}_t)|^2dt\biggr)\biggr]\nonumber\\ 
& = \E\biggl[\exp\left(-2NG(\tilde{\mu}^{N,s,y}_T)\right)\exp\biggl(\int_s^T \sum_{i=1}^N|v^N_i(t,\tilde{X}^{1,N,s,y}_t,...,\tilde{X}^{N,N}_t)|^2dt\biggr)\biggr].\nonumber\\ 
\end{align*}
\end{proof}
\subsection{Proof of Theorem \ref{theo:logefficient}}
Let $\Psi$ be a classical subsolution to \eqref{eq:HJBequation}. 
We make the choice of control from \eqref{eq:optimalcontrol} for \eqref{eq:controlledparticles} and \eqref{eq:Zs} throughout this proof. Due to the assumed boundedness of these controls, we have via Lemma \ref{lemma:RhatnumeratorGirsanov}:
\begin{align*}
&\E\biggl[\exp\left(-2NG(\hat{\mu}^{N,s,y}_T)\right)\prod_{i=1}^N (Z^{i,N,s,y})^2 \biggr]\\ 
&=\E\biggl[\exp\left(-2NG(\tilde{\mu}^{N,s,y}_T)\right)\exp\biggl(N\int_s^T \int_{\R^d}|\sigma^\top(x,\tilde{\mu}^{N,s,y}_t) \partial_\mu\Psi(t,\tilde{\mu}^{N}_t)[x]|^2\tilde{\mu}^{N,s,y}_t(dx)dt\biggr)\biggr],
\end{align*}
where $\tilde{\mu}^{N,s,y}_t = \frac{1}{N}\sum_{i=1}^N \delta_{\tilde{X}^{i,N,s,y}_t}$ and $\tilde{X}^{i,N,s,y}$ satisfy: 
\begin{align*}
d\tilde{X}^{i,N,s,y}_t &=[b(\tilde{X}^{i,N,s,y}_t,\tilde{\mu}^{N,s,y}_t) + \sigma\sigma^\top(\tilde{X}^{i,N,s,y}_t,\tilde{\mu}^{N,s,y}_t) \partial_\mu\Psi(t,\tilde{\mu}^{N,s,y}_t)[\tilde{X}^{i,N,s,y}_t]]dt \\ 
&+ \sigma(\tilde{X}^{i,N,s,y}_t,\tilde{\mu}^{N,s,y}_t) d\tilde{W}^i_t,\quad \tilde{X}^{i,N,s,y}_s=y_i.
\end{align*}
Strong existence and uniqueness of the above system of SDEs of all $N$ follows from the assumed boundedness of $\sigma$ and $\partial_\mu \Psi$ along with Assumption \ref{assumption:uncontrolleduniqueness} and an application of Girsanov theorem---see the discussion on p.81 of \cite{BDF}.

Applying Proposition \eqref{prop:prelimitLDPexpression} with $F(\mu) = 2G(\mu_T)-\int_s^T  \int_{\R^d}|\sigma^\top(z,\mu_t) \partial_\mu\Psi(t,\mu_t)[z]|^2\mu_t(dz)dt$ ($F$ is bounded by assumption and its continuity follows from, e.g., \cite{DE} Theorem A.3.18) and $\tilde{b}(t,x,\nu) = b(x,\nu)+\sigma\sigma^\top(x,\nu) \partial_\mu \Psi(t,\nu)[x]$, we get 
\begin{equation}\label{eq:variationalrepfordesiredterm}
  \begin{split}
-\frac{1}{N}\log\E\biggl[\exp\left(-2NG(\hat{\mu}^{N,s,y}_T)\right)\prod_{i=1}^N (Z^{i,N,s,y})^2 \biggr] = \inf_{u^N\in \mc{U}^N}\left\{\E\left[\frac{1}{2N}\sum_{i=1}^N\int_s^T |u^N_i(t)|^2dt\right]+2\E\left[G(\bar{\mu}^{N,s,y}_T)\right]\right.&\\ 
\left.-\E\left[\frac{1}{N}\sum_{i=1}^N\int_s^T |\sigma^\top(\bar{X}^{i,N,s,y}_t,\bar{\mu}^{N,s,y}) \partial_\mu \Psi(t,\bar{\mu}^{N,s,y}_t)[\bar{X}^{i,N,s,y}_t]|^2dt\right]\right\}&,
\end{split}
\end{equation}
where $\bar{\mu}^{N,s,y}_t = \frac{1}{N}\sum_{i=1}^N \delta_{\bar{X}^{i,N,s,y}_t}$ and
\begin{align*}
d\bar{X}^{i,N,s,y}_t &= [b(\bar{X}^{i,N,s,y}_t,\bar{\mu}^{N,s,y}_t)+\sigma\sigma^\top(\bar{X}^{i,N,s,y}_t,\bar{\mu}^{N,s,y}_t)\partial_\mu \Psi(t,\bar{\mu}^{N,s,y}_t)[\bar{X}^{i,N,s,y}_t]+ \sigma(\bar{X}^{i,N,s,y}_t,\bar{\mu}^{N,s,y}_t) u^N_i(t)]dt\\ 
&+\sigma(\bar{X}^{i,N,s,y}_t,\bar{\mu}^{N,s,y}_t) d\tilde{W}^i_t,\\ 
\bar{X}^{i,N,s,y}_s&=y_i.
\end{align*}

Fix any $u^N\in\mc{U}^N$ and let $\br{\bar{X}^{i,N,s,y}}_{i=1}^N$ satisfy the above controlled equations with this choice of control.

Letting $\tilde{\Psi}(t,x_1,...,x_N)=\Psi(t,\frac{1}{N}\sum_{i=1}^N \delta_{x_i})$, where $\Psi$ is our classical subsolution, we apply It\^o's formula to get 
\begin{align*}
&d\Psi(t,\bar{\mu}^{N,s,y}_t)\\ 
&=\biggl[\partial_t\Psi(t,\bar{\mu}^{N,s,y}_t)+\sum_{i=1}^N\biggl\lbrace[b(\bar{X}^{i,N,s,y}_t,\bar{\mu}^{N,s,y}_t)+\sigma\sigma^\top(\bar{X}^{i,N,s,y}_t,\bar{\mu}^{N,s,y}_t)\partial_\mu \Psi(t,\bar{\mu}^{N,s,y}_t)[\bar{X}^{i,N,s,y}_t]\\ 
&+\sigma(\bar{X}^{i,N,s,y}_t,\bar{\mu}^{N,s,y}_t) u^N_i(t)]\cdot \partial_{x_i}\tilde{\Psi}(t,\bar{X}^{i,N,s,y}_t,...,\bar{X}^{N,N,s,y}_t)\\ 
&+\frac{1}{2}\sigma\sigma^\top(\bar{X}^{i,N,s,y}_t,\bar{\mu}^{N,s,y}_t) : \partial^2_{x_i}\tilde{\Psi}(t,\bar{X}^{i,N,s,y}_t,...,\bar{X}^{N,N,s,y}_t)\biggr\rbrace \biggr]dt\\ 
& + \sum_{i=1}^N[\partial_{x_i}\tilde{\Psi}]^\top(t,\bar{X}^{i,N,s,y}_t,...,\bar{X}^{N,N,s,y}_t)\sigma(\bar{X}^{i,N,s,y}_t,\bar{\mu}^{N,s,y}_t) d\tilde{W}^i_t\\ 
& = \biggl[\partial_t\Psi(t,\bar{\mu}^{N,s,y}_t)+\frac{1}{N}\sum_{i=1}^N\biggl\lbrace[b(\bar{X}^{i,N,s,y}_t,\bar{\mu}^{N,s,y}_t)+\sigma\sigma^\top(\bar{X}^{i,N,s,y}_t,\bar{\mu}^{N,s,y}_t)\partial_\mu \Psi(t,\bar{\mu}^{N,s,y}_t)[\bar{X}^{i,N,s,y}_t]\\ 
&+\sigma(\bar{X}^{i,N,s,y}_t,\bar{\mu}^{N,s,y}_t) u^N_i(t)]\cdot \partial_\mu\Psi(t,\bar{\mu}^{N,s,y}_t)[\bar{X}^{i,N,s,y}_t]+\frac{1}{2}\sigma\sigma^\top(\bar{X}^{i,N,s,y}_t,\bar{\mu}^{N,s,y}_t) :\partial_z\partial_\mu\Psi(t,\bar{\mu}^{N,s,y}_t)[\bar{X}^{i,N,s,y}_t]\biggr\rbrace \biggr]dt \\ 
&+ \frac{1}{N}\sum_{i=1}^N[\partial_\mu\Psi]^\top(t,\bar{\mu}^{N,s,y}_t)[\bar{X}^{i,N,s,y}_t]\sigma(\bar{X}^{i,N,s,y}_t,\bar{\mu}^{N,s,y}_t) d\tilde{W}^i_t\\ 
&+\frac{1}{N^2}\sum_{i=1}^N\sigma\sigma^\top(\bar{X}^{i,N,s,y}_t,\bar{\mu}^{N,s,y}_t):\partial^2_\mu\Psi(t,\bar{\mu}^{N,s,y}_t)[\bar{X}^{i,N,s,y}_t,\bar{X}^{i,N,s,y}_t]dt,\\ 
\end{align*} 
where in the second step we applied Proposition \ref{prop:empprojderivatives}. Note here we have used the assumed regularity of $\Psi$ from (i) in Definition \ref{defi:subsolution}.

Now using that $\Psi$ satisfies (ii) in Definition \ref{defi:subsolution}, we have: 
\begin{align*}
&\partial_t\Psi(t,\bar{\mu}^{N,s,y}_t)+\frac{1}{N}\sum_{i=1}^N\biggl\lbrace[b(\bar{X}^{i,N,s,y}_t,\bar{\mu}^{N,s,y}_t)\\ 
&\hspace{5cm}+\sigma\sigma^\top(\bar{X}^{i,N,s,y}_t,\bar{\mu}^{N,s,y}_t)\partial_\mu \Psi(t,\bar{\mu}^{N,s,y}_t)[\bar{X}^{i,N,s,y}_t]]\cdot \partial_\mu\Psi(t,\bar{\mu}^{N,s,y}_t)[\bar{X}^{i,N,s,y}_t]\\ 
&+\frac{1}{2}\sigma\sigma^\top(\bar{X}^{i,N,s,y}_t,\bar{\mu}^{N,s,y}_t) :\partial_z\partial_\mu\Psi(t,\bar{\mu}^{N,s,y}_t)[\bar{X}^{i,N,s,y}_t]\biggr\rbrace\\ 
&\geq \frac{3}{2}\frac{1}{N}\sum_{i=1}^N \partial_\mu \Psi(t,\bar{\mu}^{N,s,y}_t)[\bar{X}^{i,N,s,y}_t]\cdot[\sigma\sigma^\top(\bar{X}^{i,N,s,y}_t,\bar{\mu}^{N,s,y}_t) \partial_\mu \Psi(t,\bar{\mu}^{N,s,y}_t)[\bar{X}^{i,N,s,y}_t]]\\ 
& = \frac{3}{2}\frac{1}{N}\sum_{i=1}^N|\sigma^\top(\bar{X}^{i,N,s,y}_t,\bar{\mu}^{N,s,y}_t) \partial_\mu \Psi(t,\bar{\mu}^{N,s,y}_t)[\bar{X}^{i,N,s,y}_t]|^2.
\end{align*}

Combining the above two displays:
\begin{align*}
\Psi(T,\bar{\mu}^{N,s,y}_T)&\geq \Psi\left(s,\frac{1}{N}\sum_{i=1}^N\delta_{y_i}\right)+ \frac{1}{N}\sum_{i=1}^N \int_s^T u^N_i(t)\cdot [\sigma^\top(\bar{X}^{i,N,s,y}_t,\bar{\mu}^{N,s,y}_t) \partial_\mu \Psi(t,\bar{\mu}^{N,s,y}_t)[\bar{X}^{i,N,s,y}_t]]dt \\ 
&+ \frac{1}{N}\sum_{i=1}^N\int_s^T [\partial_\mu\Psi]^\top(t,\bar{\mu}^{N,s,y}_t)[\bar{X}^{i,N,s,y}_t]\sigma(\bar{X}^{i,N,s,y}_t,\bar{\mu}^{N,s,y}_t) d\tilde{W}^i_t\\ 
&+\frac{1}{N^2}\sum_{i=1}^N\int_s^T \sigma\sigma^\top(\bar{X}^{i,N,s,y}_t,\bar{\mu}^{N,s,y}_t):\partial^2_\mu\Psi(t,\bar{\mu}^{N,s,y}_t)[\bar{X}^{i,N,s,y}_t,\bar{X}^{i,N,s,y}_t]dt \\ 
&+ \int_s^T  \frac{3}{2}\frac{1}{N}\sum_{i=1}^N|\sigma^\top(\bar{X}^{i,N,s,y}_t,\bar{\mu}^{N,s,y}_t) \partial_\mu \Psi(t,\bar{\mu}^{N,s,y}_t)[\bar{X}^{i,N,s,y}_t]|^2 dt\\ 
& = \Psi\left(s,\frac{1}{N}\sum_{i=1}^N\delta_{y_i}\right)+ \frac{1}{2N}\sum_{i=1}^N \int_s^T |u^N_i(t)+\sigma^\top(\bar{X}^{i,N,s,y}_t,\bar{\mu}^{N,s,y}_t) \partial_\mu \Psi(t,\bar{\mu}^{N,s,y}_t)[\bar{X}^{i,N,s,y}_t]|^2dt\\ 
&+ \int_s^T  \frac{1}{N}\sum_{i=1}^N|\sigma^\top(\bar{X}^{i,N,s,y}_t,\bar{\mu}^{N,s,y}_t) \partial_\mu \Psi(t,\bar{\mu}^{N,s,y}_t)[\bar{X}^{i,N,s,y}_t]|^2 dt\\ 
&- \int_s^T  \frac{1}{2N}\sum_{i=1}^N|u^{N}_i(t)|^2 dt+
\frac{1}{N}\sum_{i=1}^N\int_s^T [\partial_\mu\Psi]^\top(t,\bar{\mu}^{N,s,y}_t)[\bar{X}^{i,N,s,y}_t]\sigma(\bar{X}^{i,N,s,y}_t,\bar{\mu}^{N,s,y}_t) d\tilde{W}^i_t]\\ 
&+\frac{1}{N^2}\sum_{i=1}^N\int_s^T \sigma\sigma^\top(\bar{X}^{i,N,s,y}_t,\bar{\mu}^{N,s,y}_t):\partial_\mu\Psi(t,\bar{\mu}^{N,s,y}_t)[\bar{X}^{i,N,s,y}_t,\bar{X}^{i,N,s,y}_t]dt.
\end{align*}
Then, for any choice of $u^N\in\mc{U}^N$:
\begin{align*}
&\E[\frac{1}{2N}\sum_{i=1}^N\int_s^T |u^N_i(t)|^2dt]+2\E\left[G(\bar{\mu}^{N,s,y}_T)\right]-\E\left[\frac{1}{N}\sum_{i=1}^N\int_s^T |\sigma^\top(\bar{X}^{i,N,s,y}_t,\bar{\mu}^{N,s,y}_t) \partial_\mu \Psi(t,\bar{\mu}^{N,s,y}_t)[\bar{X}^{i,N,s,y}_t]|^2dt\right]\\ 
&\geq\E[G(\bar{\mu}^{N,s,y}_T)]+\E[\Psi(T,\bar{\mu}^{N,s,y}_T)]+ \E\left[\frac{1}{2N}\sum_{i=1}^N\int_s^T |u^N_i(t)|^2dt\right]\\ 
&-\E\left[\frac{1}{N}\sum_{i=1}^N\int_s^T |\sigma^\top(\bar{X}^{i,N,s,y}_t,\bar{\mu}^{N,s,y}_t) \partial_\mu \Psi(t,\bar{\mu}^{N,s,y}_t)[\bar{X}^{i,N,s,y}_t]|^2dt\right]\\ 
&\geq \E[G(\bar{\mu}^{N,s,y}_T)]+ \Psi\left(s,\frac{1}{N}\sum_{i=1}^N \delta_{y_i}\right)+\E\left[\frac{1}{2N}\sum_{i=1}^N \int_s^T |u^N_i(t)+\sigma^\top(\bar{X}^{i,N,s,y}_t,\bar{\mu}^{N,s,y}_t) \partial_\mu \Psi(t,\bar{\mu}^{N,s,y}_t)[\bar{X}^{i,N,s,y}_t]|^2dt\right]\\ 
&+\E\left[\frac{1}{N}\sum_{i=1}^N\int_s^T [\partial_\mu\Psi]^\top(t,\bar{\mu}^{N,s,y}_t)[\bar{X}^{i,N,s,y}_t]\sigma(\bar{X}^{i,N,s,y}_t,\bar{\mu}^{N,s,y}_t) d\tilde{W}^i_t\right]\\ 
&+\E\left[\frac{1}{N^2}\sum_{i=1}^N\int_s^T \sigma\sigma^\top(\bar{X}^{i,N,s,y}_t,\bar{\mu}^{N,s,y}_t):\partial^2_\mu\Psi(t,\bar{\mu}^{N,s,y}_t)[\bar{X}^{i,N,s,y}_t,\bar{X}^{i,N,s,y}_t]dt\right],
\end{align*}
where in the first inequality we used (iii) from Definition \ref{defi:subsolution}.

$u^N_i(t)+\sigma^\top(\bar{X}^{i,N,s,y}_t,\bar{\mu}^{N,s,y}_t) \partial_\mu \Psi(t,\bar{\mu}^{N,s,y}_t)[\bar{X}^{i,N,s,y}_t],i=1,...,N$ is a valid choice of control in the prelimit representation \eqref{eq:prelimitrep} from Proposition \ref{prop:prelimitLDPexpression} since $\sigma$, $\partial_\mu\Psi$ are bounded, and taking $F(\mu)=G(\mu_T)$, $\tilde{b}=b$ therein (using $G\in C_b(\mc{P}(\R^d))$ and \ref{assumption:uncontrolleduniqueness}) we can continue:
\begin{align*}
&\E\left[\frac{1}{2N}\sum_{i=1}^N\int_s^T |u^N_i(t)|^2dt\right]+2\E\left[G(\bar{\mu}^{N,s,y}_T)\right]-\E\left[\frac{1}{N}\sum_{i=1}^N\int_s^T |\sigma^\top(\bar{X}^{i,N,s,y}_t,\bar{\mu}^{N,s,y}_t) \partial_\mu \Psi(t,\bar{\mu}^{N,s,y}_t)[\bar{X}^{i,N,s,y}_t]|^2dt\right]\\ 
&\geq  \Psi\left(s,\frac{1}{N}\sum_{i=1}^N \delta_{y_i}\right)+\E\left[\frac{1}{N}\sum_{i=1}^N\int_s^T [\partial_\mu\Psi]^\top(t,\bar{\mu}^{N,s,y}_t)[\bar{X}^{i,N,s,y}_t]\sigma(\bar{X}^{i,N,s,y}_t,\bar{\mu}^{N,s,y}_t) d\tilde{W}^i_t\right]\\ 
&+\E\left[\frac{1}{N^2}\sum_{i=1}^N\int_s^T \sigma\sigma^\top(\bar{X}^{i,N,s,y}_t,\bar{\mu}^{N,s,y}_t):\partial^2_\mu\Psi(t,\bar{\mu}^{N,s,y}_t)[\bar{X}^{i,N,s,y}_t,\bar{X}^{i,N,s,y}_t]dt\right]\\ 
&-\frac{1}{N}\log\E\left[\exp\left(-NG(\mu_T^{N,s,y})\right)\right]\\ 
&=\Psi\left(s,\frac{1}{N}\sum_{i=1}^N \delta_{y_i}\right)+\E\left[\frac{1}{N^2}\sum_{i=1}^N\int_s^T \sigma\sigma^\top(\bar{X}^{i,N,s,y}_t,\bar{\mu}^{N,s,y}_t):\partial^2_\mu\Psi(t,\bar{\mu}^{N,s,y}_t)[\bar{X}^{i,N,s,y}_t,\bar{X}^{i,N,s,y}_t]dt\right]\\ 
&-\frac{1}{N}\log\E\left[\exp\left(-NG(\mu_T^{N,s,y})\right)\right]
\end{align*}
since the martingale term is bounded in square expectation for all $N$.

Infimizing over $u^{N}\in \mc{U}^N$ and using \eqref{eq:variationalrepfordesiredterm} we get:
\begin{align*}
&\liminf_{N\toinf}-\frac{1}{N}\log \biggl(\E\biggl[\exp\left(-2NG(\hat{\mu}^{N,s,y}_T)\right)\prod_{i=1}^N (Z^{i,N,s,y})^2 \biggr]\biggr)\\ 
&\geq \liminf_{N\toinf}\biggl\lbrace\Psi\left(s,\frac{1}{N}\sum_{i=1}^N \delta_{y_i}\right)\\
&+\inf_{u^N\in\mc{U}^N}\E\left[\frac{1}{N^2}\sum_{i=1}^N\int_s^T \sigma\sigma^\top(\bar{X}^{i,N,s,y}_t,\bar{\mu}^{N,s,y}_t):\partial^2_\mu\Psi(t,\bar{\mu}^{N,s,y}_t)[\bar{X}^{i,N,s,y}_t,\bar{X}^{i,N,s,y}_t]dt\right]\\ 
&-\frac{1}{N}\log\E\left[\exp(-NG(\mu_T^{N,s,y}))\right]\biggr\rbrace\\ 
&=\Psi(s,\nu)+\gamma_1,
\end{align*}
where we used the assumed continuity of $\Psi$ and the convergence of $\frac{1}{N}\sum_{i=1}^N\delta_{y_i}$ to $\nu$, \eqref{eq:gammas}, and the fact that by boundedness of $\sigma$ and (iv) in Definition \ref{defi:subsolution}, there is a constant $C$ independent on the choice of $u^N\in\mc{U}^N$ (which may change from line to line) such that:
\begin{align*}
&\E\left[\frac{1}{N^2}\sum_{i=1}^N\int_s^T \sigma\sigma^\top(\bar{X}^{i,N,s,y}_t,\bar{\mu}^{N,s,y}_t):\partial^2_\mu\Psi(t,\bar{\mu}^{N,s,y}_t)[\bar{X}^{i,N,s,y}_t,\bar{X}^{i,N,s,y}_t]dt\right]\\ 
&\leq C\E\left[\frac{1}{N^2}\sum_{i=1}^N\int_s^T |\partial^2_\mu\Psi(t,\bar{\mu}^{N,s,y}_t)[\bar{X}^{i,N,s,y}_t,\bar{X}^{i,N,s,y}_t]|dt\right]\\ 
&= C\E\left[\frac{1}{N^2}\sum_{i=1}^N\sum_{j=1}^N\int_s^T |\partial^2_\mu\Psi(t,\bar{\mu}^{N,s,y}_t)[\bar{X}^{i,N,s,y}_t,\bar{X}^{j,N,s,y}_t]|\1_{i=j}dt\right]\\ 
&\leq C\E\left[\int_s^T \biggl({\frac{1}{N^2}}\sum_{i,j=1}^N|\partial^2_\mu\Psi(t,\bar{\mu}^{N,s,y}_t)[\bar{X}^{i,N,s,y}_t,\bar{X}^{j,N,s,y}_t]|^2\biggr)^{1/2}\biggl({\frac{1}{N^2}}\sum_{i,j=1}^N\1_{i=j}\biggr)^{1/2}dt\right]\\ 
&= C\E\left[\frac{1}{N^{1/2}}\int_s^T \biggl(\frac{1}{N^2}\sum_{i,j=1}^N|\partial_\mu\Psi(t,\bar{\mu}^{N,s,y}_t)[\bar{X}^{i,N,s,y}_t,\bar{X}^{j,N,s,y}_t]|^2\biggr)^{1/2}dt\right]\\ 
&= C\E\left[\frac{1}{N^{1/2}}\int_s^T \norm{\partial^2_\mu\Psi(t,\bar{\mu}^{N,s,y}_t)}_{L^2(\R^d;\bar{\mu}^{N,s,y}_t)\otimes L^2(\R^d;\bar{\mu}^{N,s,y}_t)}dt\right]\\ 
&\leq \frac{C}{N^{1/2}}(T-s).
\end{align*}
Note that if we, instead of just the bound $\sup_{t\in[0,T],\mu\in \mc{P}_2(\R^d)}\norm{\partial^2_\mu \Psi(t,\mu)[\cdot,\cdot]}_{L^2(\R^d;\mu)\otimes L^2(\R^d;\mu)}\leq C$, have further that e.g. $\sup_{t\in[0,T],\mu\in \mc{P}_2(\R^d)}\int_{\R^d}|\partial^2_\mu \Psi(t,\mu)[z,z]|^2\mu(dz)\leq C$, then this remainder is $O(1/N)$.
{Also note that the Laplace Principle of Theorem \ref{theo:LDP} is used here to establish the equality \eqref{eq:gammas}, which yields the appearance of $\gamma_1$ in the limit.}

This gives that 
\begin{align*}
&\lim_{N\toinf}-\frac{1}{N}\log R(\hat{\delta}_N)\\ 
&=\lim_{N\toinf}\biggl\lbrace-\frac{1}{N}\log\E\biggl[\exp(-2NG(\hat{\mu}^{N,s,y}))\prod_{i=1}^N (Z^{i,N,s,y})^2 \biggr]+\frac{2}{N}\log\E\biggl[\exp(-NG(\mu^{N,s,y})) \biggr] \biggr\rbrace\\ 
& = \lim_{N\toinf}-\frac{1}{N}\log\E\biggl[\exp(-2NG(\hat{\mu}^{N,s,y}))\prod_{i=1}^N (Z^{i,N,s,y})^2 \biggr]-2\gamma_1 \quad\text{ (by \eqref{eq:gammas})}\\ 
&\geq \Psi(s,\nu)-\gamma_1.
\end{align*}
{We thus have that $\hat{\delta}_N$ admits an expansion of the form \eqref{eq:expandimportancesamplingrelativeerror} with $\hat{\gamma}\leq \gamma_1-\Psi(s,\nu)$}. Note that the case $\Psi\equiv 0$ corresponds to the standard Monte Carlo estimator $\delta_N$ in the above, from which we obtain $\gamma =2\gamma_1-\gamma_2\leq \gamma_1$ as expected -- {see \eqref{eq:expandmontecarlorelativeerror} and the discussion thereafter.} The rest of the claims stated in the theorem are now immediate. {In particular, by Jensen's inequality we know $-\frac{1}{N}\log R(\hat{\delta}_N)\leq 0$, so if $\Psi(s,\nu)=\gamma_1$ we have log-efficiency in the sense of Definition \ref{definition:logefficiency}.


%
%


\subsection{Proof of Theorem \ref{theo:expansionanalysis}}\label{subsec:theo2proof}
First, we will establish a prelimit PDE expression for the numerator in the expression for $R(\hat\delta_N)$ in \eqref{eq:Rdeltahat} in a manner along the lines of the discussion in Section \ref{sec:ontheHJBEquation}.
Define
\begin{align*}
\hat{\Xi}^N(t,x_1,...,x_N)=\E\biggl[\exp\left(-2NG(\hat{\mu}^{N,t,x}_T)\right)\prod_{i=1}^N (Z^{i,N,t,x})^2 \biggr],\quad x_1,...,x_N\in\R^d,t\in[0,T]
{}\end{align*}
where $Z^{i,N,t,x},i=1,...,N$ and $\hat{\mu}^{N,t,x}$ are as in \eqref{eq:controlledempiricalmeasure} and \eqref{eq:Zs} with the choice of controls from \eqref{eq:optimalcontrol} in Theorem \ref{theo:logefficient}. Define $v(t,\mu,x)\coloneqq -\sigma^\top(x,\mu)\partial_\mu\phi_0(t,\mu)[x]$. Then $v(t,\mu^N_x,x_i)=v_i^N(t,x_1,...,x_N)$ for all $N$, where $v_i^N$ are from \eqref{eq:optimalcontrol} with $\phi_0$ in the place of $\Psi$. Recall that in \eqref{eq:optimalcontrol} we had only assumed $\Psi$ be a classical subsolution of \eqref{eq:HJBequation}, whereas now we assume $\phi_0$ is a classical solution.

From Lemma \ref{lemma:RhatnumeratorGirsanov} (using the assumed boundedness of $\sigma$ and $\partial_\mu \phi_0$), we have:
\begin{align}\label{eq:hatXirepresentation}
\hat{\Xi}^N(t,x_1,...,x_N)&= \E\biggl[\exp(-2NG(\tilde{\mu}^{N,t,x}_T))\exp\biggl(N\int_t^T \int_{\R^d}|v(\tau,\tilde{\mu}^{N,t,x}_\tau,z)|^2\tilde{\mu}^{N,t,x}_\tau(dz)d\tau\biggr)\biggr]
\end{align}
where $\tilde{\mu}^{N,t,x}_\tau = \frac{1}{N}\sum_{i=1}^N \delta_{\tilde{X}^{i,N,t,x}_\tau}$, 
\begin{align*}
d\tilde{X}^{i,N,t,x}_\tau &= [b(\tilde{X}^{i,N,t,x}_\tau,\tilde{\mu}^{N,t,x}_\tau)- \sigma(\tilde{X}^{i,N,t,x}_\tau,\tilde{\mu}^{N,t,x}_\tau) v(\tau,\tilde{\mu}^{N,t,x}_\tau,\tilde{X}^{i,N,t,x}_\tau)]d\tau + \sigma(\tilde{X}^{i,N,t,x}_\tau,\tilde{\mu}^{N,t,x}_\tau) d\tilde{W}^i_\tau,\\ 
\tilde{X}^{i,N,t,x}_t=x_i.
\end{align*}

Applying Feynman-Kac, we have $\hat{\Xi}^N$ is the unique solution to:
\begin{equation}\label{eq:prelimitFK}
  \begin{split}
\partial_t\hat{\Xi}^N(t,x_1,...,x_N)&+\sum_{i=1}^N\biggl\lbrace [b(x_i,\mu^N_x)-\sigma(x_i,\mu^N_x) v(t,\mu^N_x,x_i)]\cdot \partial_{x_i}\hat{\Xi}^N(t,x_1,...,x_N) \\ 
&+\frac{1}{2}\sigma\sigma^\top(x_i,\mu^N_x):\partial^2_{x_i} \hat{\Xi}^N(t,x_1,...,x_N)+ |v(t,\mu^N_x,x_i)|^2\hat{\Xi}^N(t,x_1,...,x_N)\biggr\rbrace=0,\\ 
&\hspace{6cm}t\in [0,T),x_1,...,x_N\in \R^{d},\\ 
\hat{\Xi}^N(T,x_1,...,x_N)&=\exp(-2NG(\mu^N_x)),\qquad x_1,...,x_N\in \R^{d}.
\end{split}
\end{equation}
This is where the linear growth condition on $b$ is used---see, e.g., Theorem 7.6 in \cite{KS}.

Then $\tilde{\Xi}^N(t,x_1,...,x_N)\coloneqq-\frac{1}{N}\log \hat{\Xi}^N(t,x_1,...,x_N)$ satisfies
\begin{align*}
\partial_t\tilde{\Xi}^N(t,x_1,...,x_N)&+\sum_{i=1}^N\biggl\lbrace [b(x_i,\mu^N_x)-\sigma(x_i,\mu^N_x) v(t,\mu^N_x,x_i)]\cdot \partial_{x_i}\tilde{\Xi}^N(t,x_1,...,x_N) \\ 
&+\frac{1}{2}\sigma\sigma^\top(x_i,\mu^N_x):\partial^2_{x_i} \tilde{\Xi}^N(t,x_1,...,x_N)-\frac{1}{N} |v(t,\mu^N_x,x_i)|^2\nonumber\\ 
&-\frac{N}{2}|\sigma^\top(x_i,\mu^N_x)\partial_{x_i} \tilde{\Xi}^N(t,x_1,...,x_N)|^2\biggr\rbrace =0,t\in [0,T),x_1,...,x_N\in \R^{d},\nonumber\\ 
\tilde{\Xi}^N(T,x_1,...,x_N)&=2G(\mu^N_x),\qquad x_1,...,x_N\in \R^{d}.\nonumber
\end{align*}

We consider now the PDE \eqref{eq:kostas3.3} from Subsection \ref{subsec:mainresultsstatements}. Note that this agrees with \eqref{eq:prelimitHJB} if we set $v\equiv 0$, other than the fact that the terminal condition is $2G$ rather than $G$. Under the current regularity assumptions, we see via Proposition \ref{prop:empprojderivatives} that for all $N$, $\tilde{\Xi}^N$ is the empirical projection of $\Xi^N$ from \eqref{eq:kostas3.3}. That is: 
\begin{align}\label{eq:XiNrep}
\Xi^N(t,\mu^N_x)=\tilde{\Xi}^N(t,x_1,...,x_N)=-\frac{1}{N}\log\E\biggl[\exp(-2NG(\hat{\mu}^{N,t,x}_T))\prod_{i=1}^N (Z^{i,N,t,x})^2 \biggr].
\end{align}

We now observe how to uncover the log-efficiency established in Theorem \ref{theo:logefficient} (though under stronger regularity assumptions) using the method of \cite{NonAsymptotic}, and set up how higher order terms in such an expansion can be used to prove the results in Theorem \ref{theo:expansionanalysis}.

Suppose, as in \cite{NonAsymptotic}, that we have sufficient regularity to have the asymptotic expansions for  \eqref{eq:prelimitHJB} and \eqref{eq:kostas3.3} of the forms \eqref{eq:PhiNexpansion} and \eqref{eq:XiNexpansion}  for some $k$.

Then by \eqref{eq:hatXirepresentation} and \eqref{eq:hatPhirepresentation} (noting that indeed the representation \eqref{eq:hatPhirepresentation} holds for the unique solution $\hat{\Phi}^N$ to \eqref{eq:prelimitBKE}, again by Theorem 7.6 in \cite{KS}), we have 
\begin{align*}
R(\hat{\delta}_N)&=\frac{\hat{\Xi}^N(s,y_1,...,y_N)}{[\hat{\Phi}^N(s,y_1,...,y_N)]^2}\\
&=\exp\biggl(N[2\phi_0(s,\mu^N_y)-\xi_0(s,\mu^N_y)]+[2\phi_1(s,\mu^N_y)-\xi_1(s,\mu^N_y)]+[2\phi_2(s,\mu^N_y)-\xi_2(s,\mu^N_y)]/N+...\\ 
&\hspace{6cm}+[2\phi_k(s,\mu^N_y)-\xi_k(s,\mu^N_y)]/N^{k-1}+o(1/N^{k-1})\biggr),
\end{align*}
where $R(\hat{\delta}_N)$ is as in \eqref{eq:Rdeltahat}. Here we have used, by the same logic as to obtain \eqref{eq:XiNrep}, 
\begin{align*}
\Phi^N(t,\mu^N_x)=\tilde{\Phi}^N(t,x_1,...,x_N)=-\frac{1}{N}\log\E[\exp(-NG(\mu^{N,t,x}_T))],
\end{align*}
where $\Phi^N$ is as in \eqref{eq:prelimitHJB} and $\tilde{\Phi}^N$ is as in \eqref{eq:tildePhiN}.

We now see that $2\phi_0$ (recalling here that $\phi_0$ is the unique classical solution to \eqref{eq:HJBequation}) and $\xi_0$ {from \eqref{eq:xi0}} satisfy the same equation. Indeed, inserting this ansatz into \eqref{eq:xi0} and using $v(t,\mu,x)= -\sigma^\top(x,\mu)\partial_\mu\phi_0(t,\mu)[x]$, we get 
\begin{align*}
2\partial_t\phi_0(t,\mu)&+2\int_{\R^d}[b(z,\mu)+\sigma \sigma^\top(z,\mu)\partial_\mu\phi_0(t,\mu)[z]]\cdot \partial_\mu \phi_0(t,\mu)[z]+\frac{1}{2}\sigma\sigma^\top(z,\mu):\partial_z\partial_\mu \phi_0(t,\mu)[z]\mu(dz)\\ 
&-\int_{\R^d}2|\sigma^\top(z,\mu) \partial_\mu \phi_0(t,\mu)[z]|^2+|\sigma^\top(z,\mu)\partial_\mu\phi_0(t,\mu)[z]|^2\mu(dz)\\ 
&=2\partial_t\phi_0(t,\mu)+2\int_{\R^d}b(z,\mu)\cdot \partial_\mu \phi_0(t,\mu)[z]+\frac{1}{2}\sigma\sigma^\top(z,\mu):\partial_z\partial_\mu \phi_0(t,\mu)[z]\mu(dz)\nonumber\\ 
&-\int_{\R^d}|\sigma^\top(z,\mu)\partial_\mu\phi_0(t,\mu)[z]|^2\mu(dz)=0,\qquad t\in [0,T),\mu \in \mc{P}_2(\R^d),\nonumber\\ 
2\phi_0(T,\mu)&=2G(\mu),\qquad \mu \in \mc{P}_2(\R^d).\nonumber
\end{align*}
which dividing by $2$ gives the same equation as \eqref{eq:HJBequation}. 

This shows that indeed $2\phi_0(t,\mu)=\xi_0(t,\mu)$ for all $t\in [0,T],\mu\in \mc{P}_2(\R^d)$ via our uniqueness assumption. 

Thus, since $2\phi_0(s,\mu^N_y)=\xi_0(s,\mu^N_y)$ for all $N$, we have $-\frac{1}{N}\log R(\hat{\delta}_N)\tto 0$ as $N\toinf$. 

Moreover, when $k=1$, knowing $2\phi_0(s,\mu^N_y)=\xi_0(s,\mu^N_y)$, we have for any $M$ 
\begin{align*}
\lim_{N\toinf}\rho(\hat{\delta}_N)=\frac{1}{\sqrt{M}}\sqrt{\lim_{N\toinf}\exp(2\phi_1(s,\mu^N_y)-\xi_1(s,\mu^N_y)+o(1))-1}.
\end{align*}
As we will show, under our current assumptions $2\phi_1(s,\mu^N_y)=\xi_1(s,\mu^N_y)$ for all $N$, so we get $\lim_{N\toinf}\rho(\hat{\delta}_N)=0${, yielding vanishing relative error (Definition \ref{def:vanishingrelativeerror})}.

Lastly, knowing $2\phi_0(s,\mu^N_y)=\xi_0(s,\mu^N_y)$ and  $2\phi_1(s,\mu^N_y)=\xi_1(s,\mu^N_y)$, letting $k=2$ we have:
\begin{align*}
\lim_{N\toinf}T(N)&=\lim_{N\toinf}N[\exp([2\phi_2(s,\mu^N_y)-\xi_2(s,\mu^N_y)]/N+o(1/N))-1]\\ 
&= \lim_{N\toinf}[2\phi_2(s,\mu^N_y)-\xi_2(s,\mu^N_y)]+o(1)\\ 
& = \lim_{N\toinf}[2\phi_2(s,\mu^N_y)-\xi_2(s,\mu^N_y)]\\ 
& = 2\phi_2(s,\nu)-\xi_2(s,\nu)
\end{align*}
where $T(N)$ is as in \eqref{eq:totalnumberofparticles} and here we have used the assumed continuity of $\phi_2,\xi_2$. This will yield the limit \eqref{eq:TNlimit}.

We now insert the ansatz of the expansions \eqref{eq:PhiNexpansion} for {$\Phi^N(t,\mu)$} and \eqref{eq:XiNexpansion} for $\Xi^N(t,\mu)$  into their respective {equations \eqref{eq:prelimitHJB} and \eqref{eq:kostas3.3} } and match $\mc{O}(1/N^k)$ to obtain formal recursive formulas for $\xi_k$ and $\phi_k$ as in Theorem 3.3 in \cite{NonAsymptotic}. {We arrive at the expressions \eqref{eq:phik} and \eqref{eq:xik} from Subsection \ref{subsec:mainresultsstatements}.}

Now recalling $2\phi_0(t,\mu)=\xi_0(t,\mu)$, we see that {in \eqref{eq:phikstochrep} and \eqref{eq:xikstochrep} from Subsection \ref{subsec:mainresultsstatements},} $Y^{t,\mu}\overset{d}{=}Z^{t,\mu}$ (using \ref{assumption:weaksenseuniqueness}), and moreover we note that $Z^{t,\mu}$ is the solution to the (optimally) controlled limiting McKean-Vlasov equation initialized at time $t$ with distribution $\mu$ from Theorem \ref{theo:LDP}, that is $Z^{t,\mu}\overset{d}{=}\hat{X}^{u,t,\mu}$ from \eqref{eq:controlledMcKeanVlasov} with\\ 
$u(\tau)=-\sigma(\hat{X}^{u,t,\nu}_\tau,\mc{L}(\hat{X}^{u,t,\nu}_\tau))\partial_\mu \phi_0(\tau,\mc{L}(\hat{X}^{u,t,\nu}_\tau))[\hat{X}^{u,t,\nu}_\tau],{\tau\in [t,T]}$.

Under the assumption that the {expressions \eqref{eq:phikstochrep},\eqref{eq:xikstochrep}} are valid with $k=1$, we have: 
\begin{align*}
\phi_1(t,\mu)&=\frac{1}{2}\E\biggl[\int_t^T\sigma\sigma^\top(Z^{t,\mu}_\tau,\mc{L}(Z^{t,\mu}_\tau)): \partial^2_\mu\phi_{0}(\tau,\mc{L}(Z^{t,\mu}_\tau))[Z^{t,\mu}_\tau,Z^{t,\mu}_\tau]d\tau\biggr]
\end{align*}
and
\begin{align*}
\xi_1(t,\mu)&=\E\biggl[\int_t^T\frac{1}{2}\sigma\sigma^\top(Z^{t,\mu}_\tau,\mc{L}(Z^{t,\mu}_\tau)): \partial^2_\mu\xi_{0}(\tau,\mc{L}(Z^{t,\mu}_\tau))[Z^{t,\mu}_\tau,Z^{t,\mu}_\tau]d\tau\biggr]\\ 
&=\E\biggl[\int_t^T\sigma\sigma^\top(Z^{t,\mu}_\tau,\mc{L}(Z^{t,\mu}_\tau)): \partial^2_\mu\phi_{0}(\tau,\mc{L}(Z^{t,\mu}_\tau))[Z^{t,\mu}_\tau,Z^{t,\mu}_\tau]d\tau\biggr],\nonumber 
\end{align*}

so indeed $\xi_1(t,\mu)=2\phi_1(t,\mu)$ for all $t\in [0,T]$ and $\mu\in \mc{P}_2(\R^d)$. As discussed previously, this establishes that we indeed have vanishing relative error (Definition \ref{def:vanishingrelativeerror}). Also, this is the expected expression to arrive at for these first-order correction terms---compare with equations (A.7) and (A.8) in \cite{VW}, where the analogous expression is given in terms of the second derivative of their zero-viscosity (standard) HJB equation evaluated at the optimally controlled path for the limiting ODE in the small-noise regime.

Under the additional assumption that our representations {\eqref{eq:phikstochrep},\eqref{eq:xikstochrep}} for $\phi_k,\xi_k$ hold for the next order correction $k=2$, we have:
\begin{align*}
\phi_2(t,\mu)&=\frac{1}{2}\E\biggl[\int_t^T\sigma\sigma^\top(Z^{t,\mu}_\tau,\mc{L}(Z^{t,\mu}_\tau)): \partial^2_\mu\phi_{1}(\tau,\mc{L}(Z^{t,\mu}_\tau))[Z^{t,\mu}_\tau,Z^{t,\mu}_\tau]\\ 
&\hspace{4cm}-|\sigma^\top(Z^{t,\mu}_\tau,\mc{L}(Z^{t,\mu}_\tau)) \partial_\mu \phi_1(\tau,\mc{L}(Z^{t,\mu}_\tau))[Z^{t,\mu}_\tau]|^2d\tau\biggr]\nonumber
\end{align*}
and
\begin{align*}
\xi_2(t,\mu)&=\E\biggl[\int_t^T\frac{1}{2}\sigma\sigma^\top(Z^{t,\mu}_\tau,\mc{L}(Z^{t,\mu}_\tau)): \partial^2_\mu\xi_{1}(\tau,\mc{L}(Z^{t,\mu}_\tau))[Z^{t,\mu}_\tau,Z^{t,\mu}_\tau]\\ 
&\hspace{4cm}-\frac{1}{2}|\sigma^\top(Z^{t,\mu}_\tau,\mc{L}(Z^{t,\mu}_\tau)) \partial_\mu \xi_1(\tau,\mc{L}(Z^{t,\mu}_\tau))[Z^{t,\mu}_\tau]|^2d\tau\biggr]\nonumber\\ 
&=\E\biggl[\int_t^T\sigma\sigma^\top(Z^{t,\mu}_\tau,\mc{L}(Z^{t,\mu}_\tau)): \partial^2_\mu\phi_{1}(\tau,\mc{L}(Z^{t,\mu}_\tau))[Z^{t,\mu}_\tau,Z^{t,\mu}_\tau]d\tau\nonumber\\ 
&\hspace{4cm}-2|\sigma^\top(Z^{t,\mu}_\tau,\mc{L}(Z^{t,\mu}_\tau)) \partial_\mu \phi_1(\tau,\mc{L}(Z^{t,\mu}_\tau))[Z^{t,\mu}_\tau]|^2d\tau\biggr],\nonumber 
\end{align*}

so 
\begin{align*}
2\phi_2(t,\mu)-\xi_2(t,\mu)=\E\biggl[\int_t^T|\sigma^\top(Z^{t,\mu}_\tau,\mc{L}(Z^{t,\mu}_\tau)) \partial_\mu \phi_1(\tau,\mc{L}(Z^{t,\mu}_\tau))[Z^{t,\mu}_\tau]|^2d\tau\biggr].
\end{align*}

By our previous discussion and our identification of $Z^{t,\mu}$ as related to $\hat{X}^{u,t,\mu}$, we get 
\begin{align*}
\lim_{N\toinf}T(N)=\E\biggl[\int_s^T|\sigma^\top(\hat{X}^{u,s,\nu}_t,\mc{L}(\hat{X}^{u,s,\nu}_t)) \partial_\mu \phi_1(t,\mc{L}(\hat{X}^{u,s,\nu}_t))[\hat{X}^{u,s,\nu}_t]|^2dt\biggr]
\end{align*}
where $\phi_1$ is as in \eqref{eq:phik} with $k=1$ and $u(t)=-\sigma^\top(\hat{X}^{u,s,\nu}_t,\mc{L}(\hat{X}^{u,s,\nu}_t))\partial_\mu \phi_0(t,\mc{L}(\hat{X}^{u,s,\nu}_t))[\hat{X}^{u,s,\nu}_t]{,t\in[s,T]}$.

\section{Conclusions and Future Work}\label{sec:conclusions}
We have derived an importance sampling scheme for exponential functionals of the empirical measure of weakly interacting diffusions. Using the connection between the large deviations rate function {of \cite{BDF}} and mean-field optimal control, the asymptotic performance of the proposed scheme is characterized in terms of subsolutions of the Hamilton-Jacobi-Bellman equation on Wasserstein space. We provide both numerical and analytical evidence that sufficient smoothness of such a solution can yield relative error which vanishes as {the number of particles} becomes large. We also numerically explore the impact of lack of smoothness of the solution of the HJB equation on the performance of the proposed importance sampling scheme.

In future work, a major hurdle to overcome will be adapting the scheme to situations where subsolutions of the HJB equation cannot be constructed analytically. In the standard small noise setting, this is also a major issue, and when the dimension of the system becomes large advanced techniques such as the use of machine learning, optimization software, and neural networks to identify the solution of the HJB equation and/or the optimal control are used \cites{BQRS,NR,TS,VW}. In the mean-field setting, methods of stochastic control are already being used to address the issue of numerically constructing solutions to the HJB equation \eqref{eq:HJBequation}, see e.g. \cites{GMW,GLPW,CCD_Numerical,FZ,HHL,CL,Lauriere}. Moreover, using the calculus of variations form of the rate function of Dawson-G\"artner \cite{DG}, there are some examples in the literature where perturbation expansions have been made to approximate the optimal path of the controlled McKean-Vlasov equation corresponding to certain types of rare events \cites{Garnier1,BGN}. The marriage of such techniques with our proposed scheme would allow for applications beyond the linear quadratic regime and perturbations thereof, and perhaps even allow for extensions to finite-time probabilities and problems of metastability as discussed in Remark \ref{remark:extensions}. 

It may also prove useful for some target statistics to design an importance sampling scheme using the moderate deviations principle for the empirical measure \cites{BW,BS_MDP2022}. This is known to aid with the problems discussed above in the small noise setting due to the linearization of the HJB equation under the moderate deviations scaling \cites{GSS_ImportanceSampling,MSImportanceSampling}. As discussed in Remark \ref{remark:IIDandlinearG}, this can likely also be supplemented via use of an importance sampling scheme arising from large deviations of the empirical measure in the joint small noise and large $N$ limit as derived in \cites{BCcurrents,Orrieri}.

Another interesting avenue for future research is to see how our methodology can be extended to the setting where the interacting particles also share a common driving noise \cites{BCsmallnoise,DLR}.

Lastly, it is of great interest to establish rigorously the asymptotic expansions for the prelimit HJB equation presumed in Theorem \ref{theo:expansionanalysis}---see Remark \ref{remark:ontheassumptionsforexpansionanalysis}.

\appendix
\section{Differentiation on Spaces of Measures}\label{appendix:P2differentiation}

\begin{defi}
\label{def:lionderivative}
Given a function $u:\mc{P}_2(\R^d)\tto \R$, we may define a lifting of $u$ to $\tilde{u}:L^2(\tilde\W,\tilde\F,\tilde\Prob;\R^d)\tto \R$ via $\tilde u (X) = u(\mc{L}(X))$ for $X\in L^2(\tilde\W,\tilde\F,\tilde\Prob;\R^d)$. Here we assume $\tilde\W$ is a Polish space, $\tilde\F$ its Borel $\sigma$-field, and $\tilde\Prob$ is an atomless probability measure (since $\tilde\W$ is Polish, this is equivalent to every singleton having zero measure).

Here:
\begin{align*}
\mc{P}_2(\R^d) \coloneqq \br{ \mu\in \mc{P}(\R^d): \int_{\R^d}|x|^2 \mu(dx)<\infty}.
\end{align*}
$\mc{P}_2(\R^d)$ is a Polish space under the $L^2$-Wasserstein distance
\begin{align*}
\bb{W}_2 (\mu_1,\mu_2)\coloneqq \inf_{\pi \in\mc{C}_{\mu_1,\mu_2}} \biggl[\int_{\R^d\times\R^d} |x-y|^2 \pi(dx,dy)\biggr]^{1/2},
\end{align*}
where $\mc{C}_{\mu_1,\mu_2}$ denotes the set of all {probability measures on $\R^{d}\times\R^d$ with first marginal $\mu_1$ and second marginal $\mu_2$}.

We say $u$ is \textbf{L-differentiable} or \textbf{Lions-differentiable} at $\mu_0\in\mc{P}_2(\R^d)$ if there exists a random variable $X_0$ on some $(\tilde\W,\tilde\F,\tilde\Prob)$ satisfying the above assumptions such that $\mc{L}(X_0)=\mu_0$ and $\tilde u$ is Fr\'echet differentiable at $X_0$.

The Fr\'echet derivative of $\tilde u$ can be viewed as an element of $L^2(\tilde\W,\tilde\F,\tilde\Prob;\R^d)$ by identifying $L^2(\tilde\W,\tilde\F,\tilde\Prob;\R^d)$ and its dual. From this, one can find that if $u$ is L-differentiable at $\mu_0\in\mc{P}_2(\R^d)$, there is a deterministic measurable function $\xi: \R^d\tto \R^d$ such that $D\tilde{u}(X_0)=\xi(X_0)$, and that $\xi$ is uniquely defined $\mu_0$-almost everywhere on $\R^d$. We denote this equivalence class of $\xi\in L^2(\R^d,\mu_0;\R^d)$ by $\partial_\mu u(\mu_0)$ and call $\partial_\mu u(\mu_0)[\cdot]:\R^d\tto \R^d$ the \textbf{Lions derivative} of $u$ at $\mu_0$. Note that this definition is independent of the choice of $X_0$ and $(\tilde\W,\tilde\F,\tilde\Prob)$. See \cite{CD} Section 5.2.

To avoid confusion when $u$ depends on more variables than just $\mu$, if $\partial_\mu u(\mu_0)$ is differentiable at $z_0\in\R^d$, we denote its derivative at $z_0$ by $\partial_z\partial_\mu u(\mu_0)[z_0]$.
\end{defi}
\begin{defi}
\label{def:fullyC2}
(\cite{CD} Definition 5.83) We say $u:\mc{P}_2(\R^d)\tto \R$ is \textbf{Fully} $\mathbf{C^2}$ if the following conditions are satisfied:
\begin{enumerate}
\item $u$ is $C^1$ in the sense of L-differentiation, and its first derivative has a jointly continuous version $\mc{P}_2(\R^d)\times \R^d\ni (\mu,z)\mapsto \partial_\mu u(\mu)[z]\in\R^d$.
\item For each fixed $\mu\in\mc{P}_2(\R^d)$, the version of $\R^d\ni z\mapsto \partial_\mu u(\mu)[z]\in\R^d$ from the first condition is differentiable on $\R^d$ in the classical sense and its derivative is given by a jointly continuous function $\mc{P}_2(\R^d)\times \R^d\ni (\mu,z)\mapsto \partial_z\partial_\mu u(\mu)[z]\in\R^{d\times d}$.
\item For each fixed $z\in \R^d$, the version of $\mc{P}_2(\R^d)\ni \mu\mapsto \partial_\mu u(\mu)[z]\in \R^d$ in the first condition is continuously L-differentiable component-by-component, with a derivative given by a function $\mc{P}_2(\R^d)\times \R^d\times \R^d\ni(\mu,z,z')\mapsto \partial^2_\mu u(\mu)[z][z']\in\R^{d\times d}$ such that for any $\mu\in\mc{P}_2(\R^d)$ and $X\in L^2(\tilde\W,\tilde\F,\tilde\Prob;\R^d)$ with $\mc{L}(X)=\mu$, $\partial^2_{\mu}u(\mu)[z][X]$ gives the Fr\'echet derivative at $X$ of $L^2(\tilde\W,\tilde\F,\tilde\Prob;\R^d)\ni X'\mapsto \partial_\mu u(\mc{L}(X'))[z]$ for every $z\in\R^d$. Denoting $\partial^2_\mu u(\mu)[z][z']$ by $\partial^2_\mu u(\mu)[z,z']$, the map $\mc{P}_2(\R^d)\times \R^d\times \R^d\ni(\mu,z,z')\mapsto \partial^2_\mu u(\mu)[z,z']$ is also assumed to be continuous in the product topology.
\end{enumerate}
\end{defi}

We recall now a useful connection between the Lions derivative as defined in \ref{def:lionderivative} and the empirical measure.
\begin{proposition}\label{prop:empprojderivatives}
For $g:\mc{P}_2(\R^d)\tto \R$ which is fully $C^2$ in the sense of definition \ref{def:fullyC2}, we can define the empirical projection of $g$, as $g^N: (\R^d)^N\tto \R$ given by
\begin{align*}
g^N(x_1,...,x_N)\coloneqq g\left(\frac{1}{N}\sum_{i=1}^N \delta_{x_i}\right).
\end{align*}

Then $g^N$ is twice differentiable on $(\R^d)^N$, and for each $x_1,..,x_N\in\R^d$, $(i,j)\in \br{1,...,N}^2$:
\begin{align}
\label{eq:empfirstder}
\partial_{x_i} g^N(x_1,...,x_N)= \frac{1}{N} \partial_\mu  g\left(\frac{1}{N}\sum_{i=1}^N \delta_{x_i}\right) [x_i]
\end{align}
and
\begin{align}
\label{eq:empsecondder}
\partial_{x_i} \partial_{x_j} g^N(x_1,...,x_N)= \frac{1}{N} \partial_z \partial_\mu  g\left(\frac{1}{N}\sum_{i=1}^N \delta_{x_i}\right) [x_i] \1_{i=j} + \frac{1}{N^2} \partial^2_\mu g\left(\frac{1}{N}\sum_{i=1}^N \delta_{x_i}\right)[x_i,x_j].
\end{align}
\end{proposition}
\begin{proof}
This follows from Propositions 5.35 and 5.91 of \cite{CD}.
\end{proof}

\section*{Declarations}
Bezemek was partially supported by NSF-DMS 2107856. Heldman was partially supported by NSF-DMS 1902854, ARO W911NF-20-1-0244, and a subgrant of NSF-OAC 2139536. The funding agencies are not expected to gain or lose financially through publication of this manuscript. The authors have no relevant financial interests to disclose. 

\end{document}